\documentclass{amsart}

\usepackage[T1]{fontenc}
\usepackage{amsfonts}
\usepackage{amsfonts,amsmath,amssymb}
\usepackage{mathrsfs,mathtools,stmaryrd,wasysym}
\usepackage{enumerate,enumitem}
\usepackage{esint}
\usepackage{graphicx,psfrag}
\usepackage{hyperref}
\usepackage{stackengine}
\usepackage{yfonts}
\usepackage{color}

\usepackage{fullpage}

\usepackage{pgfplots}
\usepgfplotslibrary{groupplots}
\pgfplotsset{compat=newest}
\usepgfplotslibrary{external}
\usepgfplotslibrary{colorbrewer}

\newtheorem{theorem}{Theorem}[section]
\newtheorem{lemma}[theorem]{Lemma}
\newtheorem{corollary}[theorem]{Corollary}
\newtheorem{proposition}[theorem]{Proposition}

\theoremstyle{definition}

\theoremstyle{remark}
\newtheorem{remark}[theorem]{Remark}

\numberwithin{equation}{section}

\renewcommand{\div}{\operatorname{div}}

\newcommand{\bsigma}{\boldsymbol{\sigma}}
\newcommand{\btau}{\boldsymbol{\tau}}
\newcommand{\bchi}{\boldsymbol{\chi}}
\newcommand{\Hdivset}[1]{\boldsymbol{H}(\mathrm{div};#1)}
\newcommand{\ip}[2]{(#1\hspace*{.5mm},#2)}
\newcommand{\dual}[2]{\langle#1\hspace*{.5mm},#2\rangle}
\newcommand{\R}{\ensuremath{\mathbb{R}}}
\newcommand{\N}{\ensuremath{\mathbb{N}}}
\newcommand{\cT}{\ensuremath{\mathcal{T}}}
\newcommand{\diam}{\mathrm{diam}}
\newcommand{\RT}{\mathrm{RT}}
\newcommand{\normal}{{\boldsymbol{n}}}
\newcommand{\Grad}{{\boldsymbol{\nabla}}}

\newcommand{\dDiv}{\operatorname{div\mathbf{Div}}}
\newcommand{\bM}{\boldsymbol{M}}
\newcommand{\bN}{\boldsymbol{N}}

\newcommand{\bbLtwosym}{\mathbb{L}^2_{\mathrm{sym}}}

\newcommand{\logLogSlopeTriangle}[6]
{

    \pgfplotsextra
    {
        \pgfkeysgetvalue{/pgfplots/xmin}{\xmin}
        \pgfkeysgetvalue{/pgfplots/xmax}{\xmax}
        \pgfkeysgetvalue{/pgfplots/ymin}{\ymin}
        \pgfkeysgetvalue{/pgfplots/ymax}{\ymax}

        \pgfmathsetmacro{\xArel}{#1}
        \pgfmathsetmacro{\yArel}{#3}
        \pgfmathsetmacro{\xBrel}{#1-#2}
        \pgfmathsetmacro{\yBrel}{\yArel}
        \pgfmathsetmacro{\xCrel}{\xArel}

        \pgfmathsetmacro{\lnxB}{\xmin*(1-(#1-#2))+\xmax*(#1-#2)} 
        \pgfmathsetmacro{\lnxA}{\xmin*(1-#1)+\xmax*#1} 
        \pgfmathsetmacro{\lnyA}{\ymin*(1-#3)+\ymax*#3} 
        \pgfmathsetmacro{\lnyC}{\lnyA+#4*(\lnxA-\lnxB)}
        \pgfmathsetmacro{\yCrel}{\lnyC-\ymin)/(\ymax-\ymin)} 

        \coordinate (A) at (rel axis cs:\xArel,\yArel);
        \coordinate (B) at (rel axis cs:\xBrel,\yCrel);
        \coordinate (C) at (rel axis cs:\xCrel,\yCrel);

        \draw[#5]   (A)--
                    (B)-- 
                    (C)-- node[pos=0.5,anchor=west] {#6}
                    cycle;
    }
}

\newcommand{\logLogSlopeTriangleBelow}[6]
{

    \pgfplotsextra
    {
        \pgfkeysgetvalue{/pgfplots/xmin}{\xmin}
        \pgfkeysgetvalue{/pgfplots/xmax}{\xmax}
        \pgfkeysgetvalue{/pgfplots/ymin}{\ymin}
        \pgfkeysgetvalue{/pgfplots/ymax}{\ymax}

        \pgfmathsetmacro{\xArel}{#1}
        \pgfmathsetmacro{\yArel}{#3}
        \pgfmathsetmacro{\xBrel}{#1-#2}
        \pgfmathsetmacro{\yBrel}{\yArel}
        \pgfmathsetmacro{\xCrel}{\xArel}

        \pgfmathsetmacro{\lnxB}{\xmin*(1-(#1-#2))+\xmax*(#1-#2)} 
        \pgfmathsetmacro{\lnxA}{\xmin*(1-#1)+\xmax*#1} 
        \pgfmathsetmacro{\lnyA}{\ymin*(1-#3)+\ymax*#3} 
        \pgfmathsetmacro{\lnyC}{\lnyA+#4*(\lnxA-\lnxB)}
        \pgfmathsetmacro{\yCrel}{\lnyC-\ymin)/(\ymax-\ymin)} 

        \coordinate (A) at (rel axis cs:\xArel,\yArel);
        \coordinate (B) at (rel axis cs:\xBrel,\yCrel);
        \coordinate (C) at (rel axis cs:\xBrel,\yArel);

        \draw[#5]   (A)--
                    (B)-- node[pos=0.5,anchor=east] {#6}
                    (C)-- 
                    cycle;
    }
}

\begin{document}

\title{Mixed finite element methods for elliptic obstacle problems}

\author{Thomas F\"uhrer}
\address{Facultad de Matem\'aticas, Pontificia Universidad Cat\'olica de Chile, Avenida Vicu\~{n}a Mackenna 4860, Santiago, Chile.}
\email{thfuhrer@uc.cl}
\thanks{TF is supported by ANID through FONDECYT project 1210391.}

\author{Francisco Fuica}
\address{Facultad de Matem\'aticas, Pontificia Universidad Cat\'olica de Chile, Avenida Vicu\~{n}a Mackenna 4860, Santiago, Chile.}
\email{francisco.fuica@uc.cl}
\thanks{FF is supported by ANID through FONDECYT postdoctoral project 3230126.}



\date{}

\dedicatory{}

\begin{abstract}
  Mixed variational formulations for the first-order system of the elastic membrane obstacle problem and the second-order system of the Kirchhoff--Love  plate obstacle problem are proposed. 
  The force exerted by the rigid obstacle is included as a new unknown.
  A priori and a posteriori error estimates are derived for both obstacle problems.
  The a posteriori error estimates are based on conforming postprocessed solutions.
  Numerical experiments conclude this work.
\end{abstract}

\maketitle


\section{Introduction}\label{sec:intro}

  Variational inequalities have been a topic of interest for several decades due to their use in modelling important physical problems \cite{MR0880369,MR1786735}. 
  A particular instance is the \emph{obstacle problem}, in which one tries to determine the equilibrium position of an elastic membrane constrained to lie over an obstacle. 
  Another significant example is the bending of a plate over an obstacle.

  Numerical methods for approximating solutions of second-order displacement obstacle problems have been widely studied over the last decades. 
  For a variety of finite element-based solution techniques used to discretize these problems, we refer the interested reader to the following non-comprehensive list of references: \cite{MR0391502,MR0448949,MR0635927,MR0854380,MR0952855,MR1935809,MR2073936}. 
 Additionally, stabilized and least-squares methods have been proposed in \cite{MR3667082,MR3577944,MR4050087}. 
 Regarding mixed variational formulations for the membrane obstacle problem and more general variational inequalities, we mention the work \cite{MR0508584}, where the authors studied a mixed finite element approximation for variational inequalities and proved optimal error bounds.
  The mixed formulation proposed for the obstacle problem relied on using as auxiliary unknowns the gradient $\nabla u$ and the difference $u-g$, where $u$ and $g$ represent the deflection and the obstacle, respectively.  
 Dualization techniques based on Lagrange multipliers for variational inequalities were studied in \cite{MR0634283,MR1422506}. 
 A suitable extension of the results obtained in these works was provided in \cite{MR2073936}, where the authors proved the well-posedness of general variational inequalities by adapting the common Babu\v{s}ka--Brezzi conditions. 
   More recently, in \cite{MR3723328}, the authors studied a variational formulation using a Lagrange multiplier.
  Their formulation introduces an additional unknown which is the reaction force between the membrane and the obstacle, that can give useful information in the context of contact mechanics. 
  The authors designed mixed and stabilized finite element methods based on the deflection, i.e., the primal variable, to discretize the considered problem, and derived a priori and a posteriori error estimates. 

  Regarding mixed finite element approximation strategies for fourth-order obstacle-type problems, we mention  \cite{MR0596780}, where the authors extended the mixed finite element methods developed in \cite{MR0635927} for unilateral second-order problem of elliptic type. 
  Later, in \cite{MR0744475}, based on using the non-conforming Morley finite element, the authors derived convergence results without rates. 
  A suboptimal convergence rate in the energy norm was proved in \cite{MR0887070} by using the penalty method and piecewise quadratic elements.
  The main source of difficulty to derive a priori error estimates for finite element methods for the fourth-order obstacle problem as, e.g., a displacement obstacle problem for clamped Kirchhoff plates, is the lack of $H^{4}(\Omega)$-regularity for its solution $u$.
 More precisely, under suitable assumptions on data, it can be proved that $u\in H^{2+\alpha}(\Omega)$, for some $\alpha\in(0.5,1]$ determined by the interior angles of the domain $\Omega$.
 Based on the previous regularity, in \cite{MR2904578,MR3022265,MR3061064}, Brenner et al. analyzed and derived a priori error estimates utilizing conforming and non-conforming ($C^1$-continuous and $C^0$-continuous) finite element methods, and $C^0$ interior penalty methods.
 We also mention the very recent article \cite{2024arXiv240520338P}, where the authors proposed and analyzed a new mixed finite element method for approximating the solutions of fourth-order variational problems subject to a constraint as, e.g., a two-dimensional variational problem for linearly elastic shallow shells.
 Their method mainly relies on using a penalized mixed formulation and discretizing it by means of Courant triangles. 
 On the other hand, regarding a posteriori error estimates for a displacement obstacle problem for clamped Kirchhoff plates, we mention the recent works \cite{MR3595879,MR3921374}. 
 In \cite{MR3595879}, the authors proposed a residual-based a posteriori error estimator for $C^0$ interior penalty methods and proved reliability and efficiency estimates, whereas in \cite{MR3921374} the authors introduced a stabilized finite element formulation and derived a priori and a posteriori error estimates using conforming $C^1$-continuous finite elements. 
 We conclude this paragraph by mentioning the articles \cite{MR2753242,MR3128576,MR3434617}, where a posteriori error estimates for finite element methods for other fourth-order variational inequalities were investigated.
   
  In this work, we propose a novel mixed variational formulation for two instances of elliptic obstacle problems: the elastic membrane obstacle problem and the Kirchhoff--Love plate obstacle problem.
  More precisely, given two Hilbert spaces $V$ and $M$, a closed and convex set $K\subset V$ such that $0_{V}\in K$, bilinear forms $\mathfrak{a}: V\times V \to \mathbb{R}$ and $\mathfrak{b}: V\times M \to \mathbb{R}$, and linear forms $\mathfrak{G}: V'\to \mathbb{R}$ and $\mathfrak{F}: M' \to \mathbb{R}$, we study problems within the following framework: Find $(\boldsymbol\Sigma,\mathfrak{u})\in K\times M$ such that
\begin{align*}
\begin{array}{lcll}
\mathfrak{a}(\boldsymbol\Sigma, \boldsymbol\Psi - \boldsymbol\Sigma) + \mathfrak{b}(\boldsymbol\Psi - \boldsymbol\Sigma, \mathfrak{u}) & \geq & \mathfrak{G}(\boldsymbol\Psi - \boldsymbol\Sigma) \quad  &\forall \boldsymbol\Psi\in K, \\
\mathfrak{b}(\boldsymbol\Sigma, \mathfrak{v}) & = &  \mathfrak{F}(\mathfrak{v}) \quad &\forall \mathfrak{v} \in M.
\end{array}
\end{align*}
With this formulation at hand, in what follows, we comment the main differences between our method and the ones previously proposed in the literature. 
First, we note that this variational formulation, in contrast to the aforementioned references \cite{MR0508584,MR0635927}, where mixed formulations for variational inequalities were considered, does not incorporate the inequality nor the convex set $K$ in the second line of the previous system, but rather in the first.
Second, we only assume that $\mathfrak{a}$ is coercive on a strict subspace of $V$ related to the kernel of the bilinear form $\mathfrak{b}$ (see assumption \eqref{eq:assumption_coercivity_a}), as opposed to the mixed variational formulations proposed in \cite{MR0634283,MR0952855,MR1422506}, where the coercivity of $\mathfrak{a}$ on the whole space $V$ is assumed.
We also stress that to derive the formulation from~\cite{MR0508584} one requires sufficient regularity, which for the membrane obstacle problem is guaranteed for convex domains and sufficient smooth data. An extension of~\cite{MR0508584} to plate obstacle problems seems almost impossible due to the reduced regularity of fourth-order problems as discussed above.
Our formulation does not require additional regularity of solutions.
Third, we derive mixed formulations for the membrane and plate obstacle problems, based on the introduction of the additional unknown $\lambda$, which represents the reaction force between the membrane (resp. plate) and the obstacle.
  We mention that a similar approach was previously proposed in \cite{MR3723328}, where the authors introduced a mixed formulation for the membrane obstacle problem including the reaction force $\lambda$ as a further unknown; the authors only approximate the deflection and the reaction force.
  Instead, the method proposed in our work approximates simultaneously the deflection $u$, its gradient $\nabla u$, and the reaction force $\lambda$ and it extends to other obstacle problems such as the Kirchhoff--Love  plate obstacle problem.
 
  We prove well-posedness of the formulation above by assuming that the bilinear forms $\mathfrak{a}$ and $\mathfrak{b}$ satisfy the standard inf-sup assumptions \eqref{eq:assumption_non_negativity_a}--\eqref{eq:assumption_inf_sup_b} related to the classical Babu\v{s}ka--Brezzi theory, in a similar fashion as is performed in \cite{MR2073936}.
  After proving the well-posedness of the above problem, we present an appropriate discrete formulation and show its well-posedness.
  Then, after defining suitable Hilbert spaces, we derive mixed formulations for the membrane and plate obstacle problems, based on the introduction of the additional unknown $\lambda$.
  In particular, for the elliptic membrane obstacle problem, we prove that the solution associated to the mixed formulation coincides with the classical solution. 
  To approximate such a solution, we propose a discrete mixed formulation relying on the lowest-order Raviart--Thomas finite element. 
  The corresponding discrete solution presents, with respect to the previous approximations schemes in the literature, some advantages such as: it provides a direct approximation of the reaction force $\lambda$, a discrete counterpart of the complementary condition $\int_{\Omega}(u - g)\mathrm{d}\lambda = 0$, and a discrete version of $u\geq g$ in $\Omega$; see Proposition \ref{prop:properties}. 
  Moreover, we prove a priori error estimates for the approximation error and, under appropriate assumptions on data, we derive estimates in a weaker norm. 
  For the plate obstacle problem, we show that the solution obtained for the proposed mixed formulation coincides with the classical solution. 
  To approximate the solution of this problem, we propose a discrete scheme based on the discrete space $\mathbb{X}(\mathcal{T})$, recently introduced in \cite{2023arXiv230508693F}.
   This discretization technique presents similar advantages to the ones mentioned for the membrane obstacle problem. 
   In particular, it provides a direct approximation of the reaction force $\lambda$ and a discrete counterpart of $\int_{\Omega}(u - g)\mathrm{d}\lambda = 0$; see Proposition \ref{prop:plate:properties}.
   Due to the limited regularity of the solution of the plate obstacle problem, convergence rates for the approximation error are reduced (not explicitly shown here).
   Motivated by the latter observation, we develop and analyze reliable a posteriori error estimators for both, the membrane and plate obstacle problems. The estimators are defined using conforming postprocessed solutions.
      
  The presentation of the work is as follows.
   In section \ref{sec:not_and_prel}, we establish notation, recall some preliminaries, and present results concerning a particular class of variational inequalities.
   In section \ref{sec:obst_problem}, we recall the elliptic membrane obstacle problem and propose a novel mixed formulation. 
   A discrete formulation is studied and error estimates are derived as well. 
   Similarly, in section \ref{sec:plate}, we analyze a suitable mixed formulation for the plate obstacle problem and study its corresponding discrete scheme.
   A posteriori error estimates for both obstacle problems are proved in section \ref{sec:apost}.
   Finally, in section \ref{sec:num_ex}, we present a series of numerical experiments that illustrate the theory showing competitive performances of the devised a posteriori error estimators.


\section{Preliminaries}\label{sec:not_and_prel}
In section~\ref{sec:notation} we set some notation and in section~\ref{sec:fem_and_interp} we introduce the finite element spaces used in this work together with a discussion of some basic properties.
Section~\ref{sec:var_ineq_and_mix_FEM} deals with a framework for mixed obstacle problems.

\subsection{Notation and function spaces}\label{sec:notation}
Let $\Omega\subset \mathbb{R}^{n}$ ($n = 2, 3$) be a Lipschitz domain with boundary $\partial\Omega$. 
Throughout this work, we use common notations for Lebesgue and Sobolev spaces.
Define 
\begin{equation*}
\bbLtwosym(\Omega) := \{\bN\in [L^2(\Omega)]^{n\times n}\,:\, \bN=\bN^\top\}, \qquad \mathbb{H}(\dDiv;\Omega):=\{ \bN\in \bbLtwosym(\Omega) \,:\, \dDiv \bN \in L^{2}(\Omega)\}.
\end{equation*}
The inner product and norm of $L^2(\Omega)$, $[L^2(\Omega)]^n$, or $\bbLtwosym(\Omega)$ are denoted by $\ip{\cdot}{\cdot}_\Omega$ and $\|\cdot\|_{\Omega}$, respectively. 
The space $H_0^k(\Omega)$ with $k\in\N$ denotes the closure of $C_0^\infty(\Omega)$ under $\|v\|_k := \|D^k v\|_\Omega$. 
In the work at hand, we focus on the two cases $k=1$ and $k=2$, where we use the notation $\nabla v$ for $Dv$ and $\Grad\nabla v$ for $D^2 v$.
The dual space of $H_0^k(\Omega)$ is denoted by $H^{-k}(\Omega)$ and the duality on $H^{-k}(\Omega)\times H_0^k(\Omega)$ is written as
\begin{align*}
  \dual{\phi}{v} \quad\text{for } \phi\in H^{-k}(\Omega), \, v \in H_0^k(\Omega).
\end{align*}
Note that $\dual{\phi}{v} = \ip{\phi}v_\Omega$ for all $v\in H_0^k(\Omega)$ and $\phi\in L^2(\Omega)$. 
The dual space $H^{-k}(\Omega)$ is equipped with norm
\begin{align*}
  \|\phi\|_{-k}:= \sup_{0\neq v\in H_0^{k}(\Omega)} \frac{\dual{\phi}v}{\|v\|_k}, \quad \phi \in H^{-k}(\Omega).
\end{align*}
Friedrich's inequality reads
\begin{align*}
  \|v\|_\Omega \leq C_F^k \| v \|_{k} \qquad\forall v\in H_0^k(\Omega), ~ \text{ where } ~ 0<C_F\leq \diam(\Omega).
\end{align*}

For $v\in H_0^k(\Omega)$ (or similar spaces) we write $v\geq 0$ if $v\geq 0$ a.e. in $\Omega$. 
Given $\mu\in H^{-k}(\Omega)$, we write $\mu \geq 0$ if $\dual{\mu}{v} \geq 0$ for all $v\in H_0^{k}(\Omega)$ such that $v\geq 0$. 
Since every functional $\mu \in H^{-k}(\Omega)$ is a distribution, each $\mu$ satisfying $\mu \geq 0$ defines a nonnegative Radon measure in $\Omega$. In this case,
\begin{align*}
\int_{\Omega} v\,\mathrm{d}\mu = \dual{\mu}{v} \qquad \forall v\in H_0^{k}(\Omega).
\end{align*}

We recall that $\div\colon [L^2(\Omega)]^n \to H^{-1}(\Omega)$ is a bounded and surjective operator with
\begin{align*}
  \|\div\btau\|_{-1} = \sup_{0\neq v\in H_0^1(\Omega)} \frac{\dual{\div\btau}v}{\|v\|_1} = \sup_{0\neq v\in H_0^1(\Omega)} \frac{\ip{\btau}{\nabla v}_\Omega}{\|v\|_1} \leq \|\btau\|_\Omega.
\end{align*}
A similar statement is true for $\dDiv\colon \bbLtwosym(\Omega)\to H^{-2}(\Omega)$, i.e.,
\begin{align*}
  \|\dDiv\bN\|_{-2} \leq \|\bN\|_\Omega.
\end{align*}

For the analysis of our mixed formulations~\eqref{eq:strong_mixed_obst} and~\eqref{eq:plate:strong_mixed_obst} below, we define the spaces
\begin{subequations}\label{def:space_V}
\begin{align}
V_1 &:= \{ (\btau,\mu) \in [L^{2}(\Omega)]^{n}\times H^{-1}(\Omega) : \div\btau+\mu \in L^{2}(\Omega)\}, \\
V_2 &:= \{ (\bN,\mu) \in \bbLtwosym(\Omega)\times H^{-2}(\Omega) : \dDiv\bN - \mu \in L^{2}(\Omega)\}, 
\end{align}
\end{subequations}
and the non-empty, closed, and convex sets
\begin{align}\label{def:cone_K}
  K_1 :=\{(\boldsymbol\tau,\mu)\in V_1 : \mu \geq 0 \text{ in }\Omega\}, \qquad
  K_2 :=\{(\bN,\mu)\in V_2 : \mu \geq 0 \text{ in }\Omega\}.
\end{align}
Spaces $V_k$ are endowed with the inner products 
\begin{equation*}
\begin{array}{rll}
  \ip{(\bsigma,\lambda)}{(\btau,\mu)}_{V_1} := & \!\!\! \ip{\bsigma}{\btau}_\Omega + \ip{\div\bsigma+\lambda}{\div\btau+\mu}_\Omega \quad &\forall (\bsigma,\lambda),(\btau,\mu)\in V_1, \\
  \ip{(\bM,\lambda)}{(\bN,\mu)}_{V_2}  :=& \!\!\! \ip{\bM}{\bN}_\Omega + \ip{\dDiv\bM-\lambda}{\dDiv\bN-\mu}_\Omega \quad &\forall (\bM,\lambda),(\bN,\mu)\in V_2,
 \end{array}
\end{equation*}
and norms $\|\cdot\|_{V_k} = \sqrt{\ip{\cdot}{\cdot}_{V_k}}$ for $k=1,2$.

The next result states that $V_1$ and $V_{2}$ are a Hilbert spaces, which follows from some well-known results.
We include a proof for sake of completeness. 
\begin{proposition}[auxiliary result]
  Space $(V_k,\ip{\cdot}{\cdot}_{V_k})$ $(k=1,2)$ is a Hilbert space.
\end{proposition}
\begin{proof}
The proof is a consequence of the following general observation: 
Let $X, Y$, and $Z$ be Banach spaces such that $Z\hookrightarrow Y$ continuously embedded and let $A\colon X\to Y$ be a bounded linear operator. Then, the space
\[
    W :=\left\lbrace x\in X\mid Ax \in Z\right\rbrace
\]
equipped with the norm $\|\cdot\|_W := \|\cdot\|_X + \|A(\cdot)\|_Z$ is a Banach space. 
To prove this claim, we let $(w_j)_{j\in\N}\subset W$ ($j\in\N$) denote a Cauchy sequence with respect to $\|\cdot\|_W$. 
Then, $(w_j)_{j\in\N}$ is a Cauchy sequence in $X$ and $(Aw_j)_{j\in \N}$ is a Cauchy sequence in $Z$. Since $X$ and $Z$ are Banach spaces there exist $w\in X$ and $z\in Z$ with $w_j\to w$ in $X$ and $Aw_j \to z$ in $Z$, respectively. 
  Next, we show that $Aw = z\in Z$ which implies $x\in W$:
  \begin{equation*}
    \|Aw - z\|_Y \lesssim \|Aw-Aw_j\|_Y + \|Aw_j-z\|_Z \leq \|A\| \, \|w-w_j\|_X + \|Aw_j-z\|_Z \to 0 ~ \text{ as } ~ j\to \infty.
  \end{equation*}
Therefore, $W$ is closed under norm $\|\cdot\|_W$. 

For the case $k=1$, we apply the aforegoing result for $X = [L^2(\Omega)]^n\times H^{-1}(\Omega)$, $Y = H^{-1}(\Omega)$, $Z = L^2(\Omega)$, and $A(\btau,\mu) = \div\btau+\mu$. This, together with the norm equivalence
  \begin{align*}
    \|(\btau,\mu)\|_V\eqsim \|(\btau,\mu)\|_X + \|\div\btau+\mu\|_\Omega,
  \end{align*}
  which follows from the triangle inequality and the boundedness of $\div\colon [L^2(\Omega)]^n\to H^{-1}(\Omega)$, finishes the proof for this case. 
For the case $k=2$ one argues similarly. For brevity, we skip those details.
\end{proof}

The following result proves that smooth functions are dense in $V_k$ ($k=1,2$), which makes these spaces suitable for approximations with finite element spaces.

\begin{lemma}[density result]
$[C^\infty(\overline\Omega)]^{n+1}$ is dense in $V_1$ and $[C^\infty(\overline\Omega)]^{n\times n}\cap \bbLtwosym(\Omega)\times C^\infty(\overline\Omega)$ is dense in $V_2$.
\end{lemma}
\begin{proof}
  We show details only for the case $k=1$. For $k=2$ one argues following the same lines of the proof with obvious modifications. 
  Let $V:=V_1$.
  We prove the following equivalent statement: For any $v'\in V'$ with $v'(v)=0$ for all $v\in [C^\infty(\overline\Omega)]^{n+1}$ it follows that $v' \equiv 0$. 

  Let $v'\in V'$ be given and suppose that $v'|_{[C^\infty(\overline\Omega)]^{n+1}}=0$.
  Since $V$ is a Hilbert space, there exists a unique $(\bar{\btau},\bar{\mu})\in V$ representing $v'$ and we have that
  \begin{align*}
    \ip{(\bar{\btau},\bar{\mu})}{(\btau,\mu)}_V = v'(\btau,\mu) = 0 \quad\forall (\btau,\mu)\in [C^\infty(\overline\Omega)]^{n+1}.
  \end{align*}
  It suffices to show that the latter implies $(\bar{\btau},\bar{\mu})=(\mathbf{0},0)$. 
  Let $\btau\in [C^\infty(\overline\Omega)]^n$ be arbitrary and note that $\div\btau\in C^\infty(\overline\Omega)$. 
  We choose $\mu = -\div\btau\in C^\infty(\overline\Omega)$. This yields that
  \begin{align*}
    \ip{\bar{\btau}}{\btau}_{\Omega}  = \ip{(\bar{\btau},\bar{\mu})}{(\btau,-\div\btau)}_V = 0 \quad\forall \btau\in [C^\infty(\overline\Omega)]^n.
  \end{align*}
  By density of $[C^\infty(\overline\Omega)]^n$ in $[L^2(\Omega)]^n$ we infer that $\bar{\btau} = \mathbf{0}$. 
  Consequently, $\bar{\mu} \in L^2(\Omega)$ and 
  \begin{align*}
	\ip{\bar{\mu}}{\mu}_{\Omega} = \ip{(\bar{\btau},\bar{\mu})}{(\mathbf{0},\mu)}_V = 0 \quad\forall \mu\in C^\infty(\overline\Omega).
  \end{align*}
	Again, by density of $C^\infty(\overline\Omega)$ in $L^2(\Omega)$ we conclude $\bar{\mu} = 0$.
\end{proof}


\subsection{Finite element spaces and interpolation operators}\label{sec:fem_and_interp}

In this section, $\Omega$ denotes an open and bounded polygonal/polyhedral domain. 
Let $\mathcal{T}$ denote a conforming and quasi-uniform mesh of $\Omega$ such that $\overline{\Omega}=\cup_{T\in\mathcal{T}}\overline{T}$. 
Let $h_{\cT}\in L^{\infty}(\Omega)$ denote the mesh-size function, $h_{\cT}|_{T}:= h_T := \textrm{diam}(T)$ for $T\in \cT$. 
We set $h :=\max\{ h_T: T \in \cT \}$.
We also assume that each element $T\in\cT$ is an open triangle/tetrahedron. 
The set of edges/faces of the mesh is denoted by $\mathcal{E}$ and the set of vertices by $\mathcal{V}$; the set of interior vertices is denoted by $\mathcal{V}_{0}$.
In addition, for each $T\in\cT$, we use $\mathcal{E}_{T}$ and $\mathcal{V}_{T}$ to denote the edges/faces and vertices of $T$, respectively.
 
Given $T\in \cT$ and $m\in\mathbb{N}_{0}$, we set $\mathcal{P}^{m}(T):=\{z: T\to \mathbb{R} ; z \text{ is a polynomial of degree}\leq m\}$ and
\begin{equation*}
\mathcal{P}^{m}(\cT):=\{z\in L^2(\Omega) : z|_{T}\in \mathcal{P}^{m}(T) ~ \forall T\in\cT\}.  
\end{equation*}
The space of local Raviart--Thomas finite elements of order $m$ on $T$ is defined by $\text{RT}^{m}(T) := [\mathcal{P}^{m}(T)]^{n} + x\mathcal{P}^{m}(T)$, the global version of this space is defined by ~\cite{MR0483555}
$$
\text{RT}^{m}(\cT):=\{\btau\in \Hdivset\Omega : \btau|_{T}\in \text{RT}^{m}(T) ~ \forall T\in \cT\}.
$$

For two-dimensional domains, given $T\in\mathcal{T}$, we introduce the local space $\mathbb{X}(T) := \text{sym}(\text{RT}^{0}(T) \otimes \text{RT}^{1}(T))$ \cite[section 3.1]{2023arXiv230508693F}. 
We define the global discrete space 
\begin{equation*}
  \mathbb{X}(\cT):=\{ \bN \in \mathbb{L}^{2}_{\text{sym}}(\Omega) : \bN|_{T} \in \mathbb{X}(T) ~ \forall T\in \cT\}\cap \mathbb{H}(\dDiv;\Omega).
\end{equation*} 
The construction of $\mathbb{H}(\dDiv;\Omega)$ conforming finite elements is ongoing research, and we refer to~\cite{ChenHuang2024} for other variants and further references.

For each polynomial degree $m\in \mathbb{N}_{0}$, let $\Pi_{h}^{m}: L^{2}(\Omega) \to \mathcal{P}^{m}(\cT)$ denote the $L^{2}(\Omega)$-orthogonal projection of order $m$. 
A useful property of this operator is the following:
\begin{equation*}
\|(1 - \Pi_{h}^{m})\phi\|_{-(m+1)} 
=
\sup_{0\neq v\in H_0^{m+1}(\Omega)}\frac{\ip{(1 - \Pi_{h}^{m})\phi}{(1 - \Pi_{h}^{m})v}_{\Omega}}{\|v\|_{m+1}} 
\lesssim
\|h_{\cT}^{m+1}(1 - \Pi_{h}^{m})\phi\|_{\Omega} \quad \forall \phi \in L^{2}(\Omega).
\end{equation*}
Let $r > 1/2$. There exist canonical interpolation operators $\Pi^{\div}_{h}: [H^{r}(\Omega)]^{n}\cap \Hdivset\Omega \to \text{RT}^{0}(\cT)$ and $\Pi^{\dDiv}_{h}: [H^{1 + r}(\Omega)]^{2\times 2}\cap \mathbb{H}(\dDiv;\Omega) \to \mathbb{X}(\cT)$, which satisfy the properties
\begin{equation*}
\div \circ \,\Pi^{\div}_{h} = \Pi^{0}_{h} \circ\div,
\end{equation*}
and
\begin{equation*}
\dDiv \circ \, \Pi^{\dDiv}_{h} = \Pi^{1}_{h} \circ \dDiv,
\end{equation*}
respectively, see~\cite{2023arXiv230508693F}.


\subsection{Variational inequalities and mixed FEM}\label{sec:var_ineq_and_mix_FEM}
For the sake of future reference, we summarize here some results concerning the analysis of a particular class of variational inequalities. 


\subsubsection{Continuous problem}

Let $V$ and $M$ denote two Hilbert spaces and let $K\subset V$ denote a closed and convex set such that $0_{V}\in K$. 
We introduce the continuous bilinear forms $\mathfrak{a}: V\times V \to \mathbb{R}$ and $\mathfrak{b}: V\times M \to \mathbb{R}$ and the linear forms $\mathfrak{G}: V'\to \mathbb{R}$ and $\mathfrak{F}: M' \to \mathbb{R}$. 
In this work, we will consider formulations with the following structure: 
Find $(\boldsymbol\Sigma,\mathfrak{u})\in K\times M$ such that
\begin{equation}\label{eq:generalized_var_ineq}
\left\{
\begin{array}{lcl}
\mathfrak{a}(\boldsymbol\Sigma, \boldsymbol\Psi - \boldsymbol\Sigma) + \mathfrak{b}(\boldsymbol\Psi - \boldsymbol\Sigma, \mathfrak{u}) & \geq & \mathfrak{G}(\boldsymbol\Psi - \boldsymbol\Sigma) \\
\mathfrak{b}(\boldsymbol\Sigma, \mathfrak{v}) & = &  \mathfrak{F}(\mathfrak{v})
\end{array}
\right.
\end{equation}
for all $(\boldsymbol\Psi,\mathfrak{v})\in K\times M$. 
For the proof of well-posedness of problem \eqref{eq:generalized_var_ineq}, we require some additional assumptions. 
Suppose that $\mathfrak{a}$ is nonnegative on $V$, i.e., 
\begin{equation}\label{eq:assumption_non_negativity_a}
  \mathfrak{a}(\boldsymbol\Psi,\boldsymbol\Psi) \geq 0 \quad \forall \, \boldsymbol\Psi\in V,
\end{equation}
and that there exists $\alpha > 0$ such that 
\begin{equation}\label{eq:assumption_coercivity_a}
\mathfrak{a}(\boldsymbol\Psi,\boldsymbol\Psi) \geq \alpha\|\boldsymbol\Psi\|_{V}^{2} \quad \forall \, \boldsymbol\Psi\in N:=\{\boldsymbol\Psi\in V :  \mathfrak{b}(\boldsymbol\Psi,\mathfrak{v}) = 0 \quad \forall \,\mathfrak{v}\in M\}. 
\end{equation}
Finally, we assume an inf-sup condition for $\mathfrak{b}$: There exists a closed subspace $W\subset K$ and a constant $\beta > 0$ such that
\begin{equation}\label{eq:assumption_inf_sup_b}
\sup_{\boldsymbol\Psi\in W}\frac{\mathfrak{b}(\boldsymbol\Psi,\mathfrak{v})}{\|\boldsymbol\Psi\|_{V}} \geq \beta \|\mathfrak{v}\|_{M} \quad \forall \, \mathfrak{v}\in M.
\end{equation}

Before presenting the main result of this section, we prove the following lemma.

\begin{lemma}[stability estimate]\label{lemma:stab_estimate}
If assumptions \eqref{eq:assumption_coercivity_a} and \eqref{eq:assumption_inf_sup_b} hold, then any solution $(\boldsymbol\Sigma, \mathfrak{u})\in K\times M$ to problem \eqref{eq:generalized_var_ineq} satisfies the bound $\|\boldsymbol\Sigma\|_{V} + \|\mathfrak{u}\|_M \lesssim \|\mathfrak{F}\|_{M'} + \|\mathfrak{G}\|_{V'}$.
\end{lemma}
\begin{proof}
We proceed in two steps by following \cite[Lemma 2.2]{MR2073936}.

\underline{Step 1.} (estimation of $\|\boldsymbol\Sigma\|_{V}$) Use the inf-sup condition \eqref{eq:assumption_inf_sup_b} to conclude the existence of $\widetilde{\boldsymbol\Sigma}\in W$ satisfying $\mathfrak{b}(\widetilde{\boldsymbol\Sigma},\mathfrak{v}) = \mathfrak{F}(\mathfrak{v})$ for all $\mathfrak{v}\in M$ and $\|\widetilde{\boldsymbol\Sigma}\|_{V}\lesssim \|\mathfrak{F}\|_{M'}$. 
The previous properties of $\widetilde{\Sigma}$ imply, in view of \eqref{eq:generalized_var_ineq}, that $\mathfrak{b}(\widetilde{\boldsymbol\Sigma} - \boldsymbol\Sigma,\mathfrak{v}) = 0$ for all $\mathfrak{v}\in M$ and thus $\widetilde{\boldsymbol\Sigma} - \boldsymbol\Sigma\in N$. 
Now, we use the triangle inequality to obtain
\begin{equation}\label{eq:triangle_sigma}
\|\boldsymbol\Sigma\|_{V} 
 \leq
 \|\widetilde{\boldsymbol\Sigma} - \boldsymbol\Sigma\|_{V} + \|\widetilde{\boldsymbol\Sigma}\|_{V} 
 \lesssim  
 \|\widetilde{\boldsymbol\Sigma} - \boldsymbol\Sigma\|_{V} + \|\mathfrak{F}\|_{M'}.
\end{equation}
To estimate $ \|\widetilde{\boldsymbol\Sigma} - \boldsymbol\Sigma\|_{V} $, we set $(\boldsymbol\Psi,\mathfrak{v}) = (\widetilde{\boldsymbol\Sigma},0)$ in \eqref{eq:generalized_var_ineq} and use the fact that $\mathfrak{b}(\widetilde{\boldsymbol\Sigma} - \boldsymbol\Sigma,\mathfrak{u}) = 0$ to arrive at
\begin{equation*}
\mathfrak{a}(\widetilde{\boldsymbol\Sigma}, \widetilde{\boldsymbol\Sigma} - \boldsymbol\Sigma)  -  \mathfrak{G}(\widetilde{\boldsymbol\Sigma} - \boldsymbol\Sigma)  \geq \mathfrak{a}(\widetilde{\boldsymbol\Sigma} - \boldsymbol\Sigma, \widetilde{\boldsymbol\Sigma} - \boldsymbol\Sigma).
\end{equation*}
Since $\widetilde{\boldsymbol\Sigma} - \boldsymbol\Sigma\in N$, we invoke assumption \eqref{eq:assumption_coercivity_a} in combination with the estimate $\|\widetilde{\boldsymbol\Sigma}\|_{V}\lesssim \|\mathfrak{F}\|_{M'}$ to conclude that $\|\widetilde{\boldsymbol\Sigma} - \boldsymbol\Sigma\|_{V} \lesssim \|\mathfrak{F}\|_{M'} + \|\mathfrak{G}\|_{V'}$. 
Using this estimate in \eqref{eq:triangle_sigma} yields that $\|\boldsymbol\Sigma\|_{V} \lesssim \|\mathfrak{F}\|_{M'} + \|\mathfrak{G}\|_{V'}$.

\underline{Step 2.} (estimation of $\|\mathfrak{u}\|_M$) Since $K$ is closed and convex, we have that $\boldsymbol\Sigma \pm \widetilde{\boldsymbol\Psi} \in K$ for all $\widetilde{\boldsymbol\Psi}\in W$; see \cite[Lemma 2.1]{MR2073936}. 
Hence, replacing $\boldsymbol\Psi = \boldsymbol\Sigma \pm \widetilde{\boldsymbol\Psi}$ in \eqref{eq:generalized_var_ineq} results in 
\begin{equation*}
\mathfrak{a}(\boldsymbol\Sigma, \widetilde{\boldsymbol\Psi} ) + \mathfrak{b}(\widetilde{\boldsymbol\Psi} , \mathfrak{u}) = \mathfrak{G}(\widetilde{\boldsymbol\Psi}) \quad \forall \, \widetilde{\boldsymbol\Psi}\in W.
\end{equation*}
We thus use the inf-sup condition \eqref{eq:assumption_inf_sup_b} in combination with the bound for $\|\boldsymbol\Sigma\|_{V}$ to conclude that $\|\mathfrak{u}\|_{M} \lesssim \|\boldsymbol\Sigma\|_{V} + \|\mathfrak{G}\|_{M'} \lesssim \|\mathfrak{F}\|_{M'} + \|\mathfrak{G}\|_{M'}$.
\end{proof}

The next result  follows from \cite[Theorem 2.3]{MR2073936} and establishes the well-posedness of problem \eqref{eq:generalized_var_ineq}. For the sake of completeness, we provide a proof.

\begin{theorem}[well-posedness]\label{thm:wp_generalized}
If assumptions \eqref{eq:assumption_non_negativity_a}, \eqref{eq:assumption_coercivity_a}, and \eqref{eq:assumption_inf_sup_b} hold, then there exists a unique solution $(\boldsymbol\Sigma,\mathfrak{u})\in K\times M$ for problem \eqref{eq:generalized_var_ineq}.
\end{theorem}
\begin{proof}
We begin by reformulating problem \eqref{eq:generalized_var_ineq}. To accomplish task, we define the Hilbert space $\mathbf{H}=V\times M$ equipped with the norm 
\begin{equation*}
\|T\|_{\mathbf{H}}:=\left( \|\boldsymbol\Psi\|_{V}^{2} + \|\mathfrak{v}\|_{M}^{2}\right)^{\frac{1}{2}} \quad \forall \, T=(\boldsymbol\Psi,\mathfrak{v})\in \mathbf{H}.
\end{equation*}
Additionally, we introduce the continuous bilinear form $A:\mathbf{H}\times \mathbf{H}\to \mathbb{R}$ and the continuous linear form $L:\mathbf{H}\to \mathbb{R}$ defined, respectively, by
\begin{equation*}
\begin{array}{rcll}
A(S,T) &:= & \mathfrak{a}(\boldsymbol\Sigma,\boldsymbol\Psi) + \mathfrak{b}(\boldsymbol\Psi,\mathfrak{u}) - \mathfrak{b}(\boldsymbol\Sigma,\mathfrak{v}) \qquad &\forall \, S=(\boldsymbol\Sigma,\mathfrak{u}),\, T=(\boldsymbol\Psi,\mathfrak{v}) \in \mathbf{H},\\
L(T) & := & \mathfrak{G}(\boldsymbol\Psi) - \mathfrak{F}(\mathfrak{v})\quad &\forall \, T=(\boldsymbol\Psi,\mathfrak{v}) \in \mathbf{H}.
\end{array}
\end{equation*}
With these ingredients at hand, we reformulate problem \eqref{eq:generalized_var_ineq} as: Find $S=(\boldsymbol\Sigma,\mathfrak{u})\in K\times M$ such that 
\begin{equation}\label{eq:generalized_reformulated}
A(S, T - S) \geq L(T - S) \quad \forall \, T\in K\times M.
\end{equation}

We now proceed on two steps.

\underline{Step 1.} (existence of solution) 
To prove that \eqref{eq:generalized_reformulated} admits a solution, we proceed by using a perturbation technique. Given $\delta > 0$ small, we introduce the bilinear form $A_{\delta}:\mathbf{H}\times \mathbf{H}\to \mathbb{R}$ defined by
\begin{equation*}
A_{\delta}(S,T):= A(S,T) + \delta(\boldsymbol\Sigma,\boldsymbol\Psi)_{V} + \delta(\mathfrak{u},\mathfrak{v})_{M} \quad \forall \, S = (\boldsymbol\Sigma,\mathfrak{u}),\,T = (\boldsymbol\Psi,\mathfrak{v}) \in \mathbf{H}.
\end{equation*}
Assumption \eqref{eq:assumption_non_negativity_a} immediately implies the coercivity of $A_{\delta}$ on $\mathbf{H}$. 
Consequently, we apply Stampacchia's theorem to problem
\begin{equation}\label{eq:generalized_reformulated_perturbed}
S_{\delta}=(\boldsymbol\Sigma_{\delta},\mathfrak{u}_{\delta})\in K\times M: \quad 
A_{\delta}(S_{\delta}, T - S_{\delta}) \geq L(T - S_{\delta}) \quad \forall \, T\in K\times M,
\end{equation}
to conclude the existence of a unique parameter dependent solution $S_{\delta}=(\boldsymbol\Sigma_{\delta},\mathfrak{u}_{\delta})$.

We show now that the sequence $\{(\boldsymbol\Sigma_{\delta},\mathfrak{u}_{\delta})\}_{\delta > 0}$ is uniformly bounded in $\mathbf{H}$, by using the arguments elaborated in the proof of Lemma \ref{lemma:stab_estimate}. 
The inf-sup condition \eqref{eq:assumption_inf_sup_b} yields the existence of $\widetilde{\boldsymbol\Sigma}_{\delta}\in W$ such that $b(\widetilde{\boldsymbol\Sigma}_{\delta},\mathfrak{v}) = \mathfrak{F}(\mathfrak{v}) + \delta(\mathfrak{u}_{\delta},\mathfrak{v})$ for all $\mathfrak{v}\in M$ and $\|\widetilde{\boldsymbol\Sigma}_{\delta}
\|_{V}\lesssim \|\mathfrak{F}\|_{M'} + \delta\|\mathfrak{u}_{\delta}\|_{M}$. 
Hence, $b(\widetilde{\boldsymbol\Sigma}_{\delta}-\boldsymbol\Sigma_{\delta},\mathfrak{v}) = 0$ for all $\mathfrak{v}\in M$, which implies that $\widetilde{\boldsymbol\Sigma}_{\delta} - \boldsymbol\Sigma_{\delta}\in N$. 
To control $\|\widetilde{\boldsymbol\Sigma}_{\delta} - \boldsymbol\Sigma_{\delta}\|_{V} $, we set $T = (\widetilde{\boldsymbol\Sigma}_{\delta},\mathfrak{u}_{\delta})$ in \eqref{eq:generalized_reformulated_perturbed}, use that $\mathfrak{b}(\widetilde{\boldsymbol\Sigma}_{\delta} - \boldsymbol\Sigma_{\delta},\mathfrak{u}) = 0$, and invoke assumption \eqref{eq:assumption_coercivity_a}.
 These arguments yield
\begin{equation*}
\mathfrak{a}(\widetilde{\boldsymbol\Sigma}_{\delta}, \widetilde{\boldsymbol\Sigma}_{\delta} - \boldsymbol\Sigma_{\delta}) + \delta(\widetilde{\boldsymbol\Sigma}_{\delta}, \widetilde{\boldsymbol\Sigma}_{\delta} - \boldsymbol\Sigma_{\delta})_{V} -  \mathfrak{G}(\widetilde{\boldsymbol\Sigma}_{\delta} - \boldsymbol\Sigma_{\delta})  
 \geq \mathfrak{a}(\widetilde{\boldsymbol\Sigma}_{\delta} - \boldsymbol\Sigma_{\delta}, \widetilde{\boldsymbol\Sigma}_{\delta} - \boldsymbol\Sigma_{\delta}) + \delta\|\widetilde{\boldsymbol\Sigma}_{\delta} - \boldsymbol\Sigma_{\delta}\|_{V}^{2}
 \geq (\alpha + \delta)\|\widetilde{\boldsymbol\Sigma}_{\delta} - \boldsymbol\Sigma_{\delta}\|_{V}^{2}.
\end{equation*}
From this estimate, we obtain that $\|\widetilde{\boldsymbol\Sigma}_{\delta} - \boldsymbol\Sigma_{\delta}\|_{V} \lesssim (1+\delta)\|\widetilde{\boldsymbol\Sigma}_{\delta}\|_{V} + \|\mathfrak{G}\|_{V'} \lesssim \|\mathfrak{F}\|_{M'} + \|\mathfrak{G}\|_{V'} + \delta\|\mathfrak{u}_{\delta}\|_{M}$. Therefore, we conclude that
\begin{equation*}
\|\boldsymbol\Sigma_{\delta}\|_{V} 
\leq \|\widetilde{\boldsymbol\Sigma}_{\delta}\|_{V} + \|\widetilde{\boldsymbol\Sigma}_{\delta} - \boldsymbol\Sigma_{\delta}\|_{V} 
\lesssim
 \|\mathfrak{F}\|_{M'} + \|\mathfrak{G}\|_{V'} + \delta\|\mathfrak{u}_{\delta}\|_{M}.
\end{equation*}
To derive the bound for $\|\mathfrak{u}_{\delta}\|_{M}$ we evaluate, in \eqref{eq:generalized_reformulated_perturbed}, $T=(\boldsymbol\Sigma_{\delta}\pm\widetilde{\boldsymbol\Psi},\mathfrak{u}_{\delta})$ with $\widetilde{\boldsymbol\Psi}\in W$. This results in
\begin{equation*}
\mathfrak{a}(\boldsymbol\Sigma_{\delta},\widetilde{\boldsymbol\Psi}) + \mathfrak{b}(\widetilde{\boldsymbol\Psi},\mathfrak{u}_{\delta}) + \delta(\boldsymbol\Sigma_{\delta},\widetilde{\boldsymbol\Psi})_{V} = \mathfrak{G}(\widetilde{\boldsymbol\Psi}) \quad \forall \, \widetilde{\boldsymbol\Psi}\in W.
\end{equation*}
This equation, in combination with the inf-sup condition \eqref{eq:assumption_inf_sup_b} and the bound obtained for $\|\boldsymbol\Sigma_{\delta}\|_{V}$, yields that $\|\mathfrak{u}_{\delta}\|_{M}\lesssim  \|\mathfrak{F}\|_{M'} + \|\mathfrak{G}\|_{V'} + \delta\|\mathfrak{u}_{\delta}\|_{M}$. We have thus proved that
\begin{equation*}
\|\boldsymbol\Sigma_{\delta}\|_{V}  + \|\mathfrak{u}_{\delta}\|_{M} 
\lesssim
 \|\mathfrak{F}\|_{M'} + \|\mathfrak{G}\|_{V'} + \delta\|\mathfrak{u}_{\delta}\|_{M}.
\end{equation*}
Hence, the sequence $\{(\boldsymbol\Sigma_{\delta},\mathfrak{u}_{\delta})\}_{\delta > 0}$ is uniformly bounded in $\mathbf{H}$ when $\delta \to 0$.

From the uniform boundedness of the sequence $\{(\boldsymbol\Sigma_{\delta},\mathfrak{u}_{\delta})\}_{\delta > 0}$, we deduce the existence of a subsequence, still indexed by $\delta$ to simplify the notation, converging weakly in $\mathbf{H}$ to some $(\boldsymbol\Sigma,\mathfrak{u})\in K\times M$. 
We prove that $S=(\boldsymbol\Sigma,\mathfrak{u})$ solves \eqref{eq:generalized_reformulated}. 
To accomplish this goal, we first notice that \eqref{eq:generalized_reformulated_perturbed} can be rewritten as 
\begin{equation*}
S_{\delta}=(\boldsymbol\Sigma_{\delta},\mathfrak{u}_{\delta})\in K\times M: \quad 
A_{\delta}(S_{\delta}, T) - \mathfrak{a}(\boldsymbol\Sigma_{\delta},\boldsymbol\Sigma_{\delta}) - \delta(\|\boldsymbol\Sigma_{\delta}\|_{V}^{2} + \|\mathfrak{u}_{\delta}\|_{M}^{2}) \geq L(T - S_{\delta}) \quad \forall \, T\in K\times M.
\end{equation*}
Then, we use the convergences $A_{\delta}(S_{\delta},T) \to A(S,T)$ and $L(S_{\delta})\to L(S)$ as $\delta\to 0$, and the fact that $\liminf_{\delta\to 0}\mathfrak{a}(\boldsymbol\Sigma_{\delta},\boldsymbol\Sigma_{\delta}) \geq \mathfrak{a}(\boldsymbol\Sigma,\boldsymbol\Sigma)$; the latter follows from \eqref{eq:assumption_non_negativity_a}.
These arguments give $A(S, T) - \mathfrak{a}(\boldsymbol\Sigma,\boldsymbol\Sigma) \geq L(T - S)$ for all $T\in K\times M$. 
Therefore, since $A(S, T) - \mathfrak{a}(\boldsymbol\Sigma,\boldsymbol\Sigma) = A(S, T - S)$, we conclude that the limit point $(\boldsymbol\Sigma,\mathfrak{u})\in K\times M$ is a solution to \eqref{eq:generalized_reformulated}, equivalently, problem \eqref{eq:generalized_var_ineq}.

\underline{Step 2.} (uniqueness) Let us assume that there exists $S_{1}=(\boldsymbol\Sigma_{1},\mathfrak{u}_{1})$ and $S_{2}=(\boldsymbol\Sigma_{2},\mathfrak{u}_{2})$, solutions for problem \eqref{eq:generalized_reformulated}, equivalently, \eqref{eq:generalized_var_ineq}. 
We notice that $b(\boldsymbol\Sigma_{1} - \boldsymbol\Sigma_{2},\mathfrak{v}) = 0$ for all $\mathfrak{v}\in M$, i.e., $\boldsymbol\Sigma_{1} - \boldsymbol\Sigma_{2}\in W$. 
Selecting $T=(\boldsymbol\Sigma_{1},0)$ in \eqref{eq:generalized_reformulated}, then $T=(\boldsymbol\Sigma_{2},0)$ in \eqref{eq:generalized_reformulated}, and adding the obtained inequalities give
\begin{equation*}
\mathfrak{a}(\boldsymbol\Sigma_{1} - \boldsymbol\Sigma_{2}, \boldsymbol\Sigma_{1} - \boldsymbol\Sigma_{2})\leq 0.
\end{equation*} 
This, in view of assumption \eqref{eq:assumption_coercivity_a}, yields that $\|\boldsymbol\Sigma_{1} - \boldsymbol\Sigma_{2}\|_{V}=0$ and thus $\boldsymbol\Sigma_{1} = \boldsymbol\Sigma_{2}$. 

Let $\boldsymbol\Sigma = \boldsymbol\Sigma_{1} = \boldsymbol\Sigma_{2}$. To prove that $\mathfrak{u}_{1} = \mathfrak{u}_{2}$, we take $T=(\boldsymbol\Sigma\pm\widetilde{\boldsymbol\Psi},0)$ with $\widetilde{\boldsymbol\Psi}\in W$ in the problems that $S_{1} = (\boldsymbol\Sigma, \mathfrak{u}_{1})$ and $S_{2} = (\boldsymbol\Sigma,\mathfrak{u}_{2})$ solve to obtain $\mathfrak{a}(\boldsymbol\Sigma, \widetilde{\boldsymbol\Psi}) + b(\widetilde{\boldsymbol\Psi}, \mathfrak{u}_{1}) = \mathfrak{G}(\widetilde{\boldsymbol\Psi})$ and $\mathfrak{a}(\boldsymbol\Sigma, \widetilde{\boldsymbol\Psi}) + b(\widetilde{\boldsymbol\Psi}, \mathfrak{u}_{2}) = \mathfrak{G}(\widetilde{\boldsymbol\Psi})$, respectively. 
Subtracting these two equalities we obtain $\mathfrak{b}(\widetilde{\boldsymbol\Psi},\mathfrak{u}_{1} - \mathfrak{u}_{2}) = 0$ for all $\widetilde{\boldsymbol\Psi}\in W$. This, in conjunction with the inf-sup condition \eqref{eq:assumption_inf_sup_b}, allows us to conclude that $\|\mathfrak{u}_{1} - \mathfrak{u}_{2}\|_{M} = 0$, which ends the proof.
\end{proof}


\subsubsection{Discrete problem}
Let $V_{h}$ and $M_{h}$ be two finite dimensional sub-spaces of $V$ and $M$, respectively, and let $K_{h}\subset V_{h}$ be a closed and convex set not-necessarily contained in $K$ such that $0_{V}\in K_{h}$. 
The discrete approximation of problem \eqref{eq:generalized_var_ineq} reads as follows: Find $(\boldsymbol\Sigma_{h},\mathfrak{u}_{h})\in K_{h}\times M_{h}$ such that
\begin{equation}\label{eq:generalized_var_ineq_discrete}
\left\{
\begin{array}{lcll}
\mathfrak{a}(\boldsymbol\Sigma_{h}, \boldsymbol\Psi_{h} - \boldsymbol\Sigma_{h}) + \mathfrak{b}(\boldsymbol\Psi_{h} - \boldsymbol\Sigma_{h}, \mathfrak{u}_{h}) & \geq & \mathfrak{G}(\boldsymbol\Psi_{h} - \boldsymbol\Sigma_{h}) \\
\mathfrak{b}(\boldsymbol\Sigma_{h}, \mathfrak{v}_{h}) & = &  \mathfrak{F}(\mathfrak{v}_{h})
\end{array}
\right.
\end{equation}
for all $(\boldsymbol\Psi_{h},\mathfrak{v}_{h})\in K_{h}\times M_{h}$. 
As in the continuous case, we shall need some suitable assumptions in order to guarantee the well-posedness of problem \eqref{eq:generalized_var_ineq_discrete}. 
To be precise, we assume that $N_{h}:=\{\boldsymbol\Psi_{h}\in V_{h} : \mathfrak{b}(\boldsymbol\Psi_{h},\mathfrak{v}) = 0 \quad \forall \, \mathfrak{v}\in M_{h}\}$ satisfies 
\begin{equation}\label{eq:conforming_condition}
N_{h}\subset N.
\end{equation}
We further assume a discrete inf-sup condition: There exist a subspace $W_{h}\subset V_{h}$ satisfying $W_{h}\subset K_{h}$ and a constant $\tilde{\beta} > 0$ such that
\begin{equation}\label{eq:discrete_inf_sup}
\sup_{\boldsymbol\Psi_{h}\in W_{h}}\frac{\mathfrak{b}(\boldsymbol\Psi_{h},\mathfrak{v}_{h})}{\|\boldsymbol\Psi_{h}\|_{V}} \geq \tilde{\beta} \|\mathfrak{v}_{h}\|_{M} \quad \forall \, \mathfrak{v}_{h}\in M_{h}.
\end{equation}

With the previous ingredients at hand, we can argue as in the proofs of Lemma \ref{lemma:stab_estimate} and Theorem \ref{thm:wp_generalized} to conclude the well-posedness of problem \eqref{eq:generalized_var_ineq_discrete}.

\begin{theorem}[well-posedness discrete problem]\label{thm:wp_discrete_generalized}
If assumptions \eqref{eq:assumption_non_negativity_a},  \eqref{eq:assumption_coercivity_a}, \eqref{eq:conforming_condition}, and \eqref{eq:discrete_inf_sup} hold, then there exists a unique solution $(\boldsymbol\Sigma_{h},\mathfrak{u}_{h})\in K_{h}\times M_{h}$ for problem \eqref{eq:generalized_var_ineq_discrete}.
\end{theorem}


\subsection{Regularized maximum function}
While $\max\{f,g\}\in H^k(\Omega)$ for $f,g\in H^k(\Omega)$ is true for $k=1$, it does not hold for $k\geq 2$ in general. 
This poses challenges for the analysis, in particular the a posteriori error analysis, for the plate obstacle problem.
Therefore, we consider a regularized maximum function derived in~\cite[eq. (2.11)]{MR0579403}: Let $\varepsilon>0$ be given and define for $x\in\Omega$
\begin{equation*}
  M_{\varepsilon}\{f,g\}(x):=
\begin{cases}
  \qquad \qquad \qquad \quad f(x) &\text{ if } ~  f(x)-g(x)>\varepsilon, \\
\frac{1}{4\varepsilon}(f(x) - g(x))^{2} + \frac{1}{2}(f(x) + g(x)) + \frac{\varepsilon}{4} &\text{ if } ~ |f(x) - g(x)|\leq \varepsilon,\\
\qquad \qquad \qquad \quad g(x) &\text{ if } ~ f(x)-g(x)<-\varepsilon. \\
\end{cases}
\end{equation*}
We stress that $M_{\varepsilon}\{f,g\} \geq \max\{f,g\}$ and $\lim_{\varepsilon\to 0} M_{\varepsilon}\{f,g\} = \max\{f,g\}$ (pointwise), cf.~\cite[Section~2]{MR0579403}.
Furthermore, $M_{\varepsilon}\{f,g\}= M_{\varepsilon}\{g,f\}$ and $M_{\varepsilon}\{f,g\} = g + M_{\varepsilon}\{f-g,0\}$ as is verified by straightforward calculations.
To prove $H^2(\Omega)$ regularity of $M_{\varepsilon}\{f,g\}$ under certain assumptions on $f$ and $g$ we use the next result. We denote by $C^{k,1}(\R)$ the space of $k$ times differentiable functions such that the $k$-th derivative is Lipschitz continuous with domain $\R$.

\begin{lemma}\label{lem:chainrule}
  Suppose $F\in C^{1,1}(\R)$ and $v\in H^2(\Omega)\cap C^1(\overline\Omega)$. Then, $F\circ v \in H^2(\Omega)\cap C^1(\overline\Omega)$.
\end{lemma}
\begin{proof}
  By the chain rule we get $F\circ v \in H^1(\Omega)\cap C^0(\overline\Omega)$ and
  \begin{align*}
    \nabla(F\circ v) = (F'\circ v)\nabla v.
  \end{align*}
  Note that $F'\in C^{0,1}(\R)$. Thus, by the Stampacchia chain rule theorem we have $F'\circ v\in H^1(\Omega)$. 
  We conclude that $F'\circ v\in H^1(\Omega)\cap C^0(\overline\Omega)$. 
  Since $\nabla v\in [H^1(\Omega)\cap C^0(\overline\Omega)]^n$ by assumption, the product rule then shows that $\nabla(F\circ v) \in H^1(\Omega)^n$.
  This finishes the proof.
\end{proof}

\begin{proposition}\label{prop:max_h02}
  Let $\varepsilon>0$. 
  If $u,v\in H^2(\Omega)\cap C^1(\overline\Omega)$, then $M_\varepsilon\{u,v\} \in H^2(\Omega)\cap C^1(\overline\Omega)$.
  
  If additionally $u\in H_0^2(\Omega)$ and $v|_{\partial\Omega}<0$ then there exists $\varepsilon_0>0$ such that $M_{\varepsilon}\{u,v\}\in H_0^{2}(\Omega)$ for all $0<\varepsilon<\varepsilon_0$.
\end{proposition}
\begin{proof}
  Define
  \begin{align*}
    F_\varepsilon(t) = \begin{cases} 
         \qquad \quad t & \text{ if } t>\varepsilon, \\
      \frac{1}{4\varepsilon} t^2 + \frac12t + \frac\varepsilon4 & \text{if } |t|\leq \varepsilon, \\
         \qquad \quad  0 & \text{ if } t<-\varepsilon.
    \end{cases}
  \end{align*}
  Then, $F_\varepsilon\in C^{1,1}(\R)$. For the first assertion, let $u,v\in H^2(\Omega)\cap C^1(\overline\Omega)$ be given. Note that
  \begin{align*}
    M_\varepsilon\{u,v\} = v + F_\varepsilon(u-v).
  \end{align*}
  By Lemma~\ref{lem:chainrule}, $F_\varepsilon(u-v) \in H^2(\Omega)\cap C^1(\overline\Omega)$, and, consequently, $M_\varepsilon\{u,v\}\in H^2(\Omega)\cap C^1(\overline\Omega)$.

  Since $v|_{\partial\Omega}<0$ and $u|_{\partial\Omega} = 0$, there exists $\varepsilon_0>0$ and $\delta>0$, both sufficiently small, such that $u(x)-g(x)>\varepsilon_0$ for all $x\in\Omega_\delta = \{x\in\Omega\,:\,\mathrm{dist}(x,\partial\Omega)<\delta\}$. Therefore, $M_\varepsilon\{u,v\}(x) = u(x)$ for $x\in \Omega_\delta$ and all $\varepsilon < \varepsilon_{0}$. We thus conclude that $M_\varepsilon\{u,v\}|_{\partial\Omega} = 0$ as well as $\nabla M_\varepsilon\{u,v\}|_{\partial\Omega} = 0$, which finishes the proof.
  \end{proof}


\section{Mixed FEM for the membrane obstacle problem}\label{sec:obst_problem}
Let $\Omega\subset \mathbb{R}^{n}$ ($n = 2,3$) be an open and bounded domain with Lipschitz boundary $\partial\Omega$. Let $f\in L^2(\Omega)$ and $g\in H^{1}(\Omega)\cap C(\overline{\Omega})$ such that $g \leq 0$ on $\partial\Omega$. 
The obstacle problem then reads as follows: Find $u$ such that
\begin{equation}\label{eq:obstacle_problem}
-\Delta u \geq f \text{ in }\Omega, 
\quad u \geq g \text{ in }\Omega,
\quad (-\Delta u - f)(u - g) = 0 \text{ in }\Omega,
\quad u = 0 \text{ on }\partial\Omega.
\end{equation}
Let us introduce the set $\mathsf{V}:=\{v\in H_0^{1}(\Omega) : v \geq g \text{ a.e. in }\Omega\}$. 
Problem \eqref{eq:obstacle_problem} admits a unique solution $u\in \mathsf{V}$ and it can be equivalently characterized by the variational inequality: 
Find $u\in \mathsf{V}$ such that
\begin{equation}\label{eq:weakformObstacle}
\ip{\nabla u}{\nabla (v - u)}_{\Omega} \geq  \ip{f}{v - u}_{\Omega} \quad \forall v\in \mathsf{V};
\end{equation}
see \cite{MR1786735}.


\subsection{Mixed variational formulation}\label{sec:mixed_var_for}

Let $u\in \mathsf{V}$ denote the unique solution to the obstacle problem \eqref{eq:obstacle_problem}. 
Define $\lambda:=-\Delta u - f\in H^{-1}(\Omega)$ and $\boldsymbol\sigma:=\nabla u$. 
We can thus rewrite problem \eqref{eq:obstacle_problem} as the following first-order system: 
\begin{equation}\label{eq:strong_mixed_obst}
\left\{
\begin{array}{rcll}
-\text{div}\,\boldsymbol\sigma - \lambda & = & f &\text{ in } \Omega, \\
\boldsymbol\sigma - \nabla u & = & 0 &\text{ in }\Omega,\\
\lambda & \geq & 0 &\text{ in }\Omega, \\
u &\geq& g &\text{ in }\Omega,\\
\lambda(u - g) & = & 0 &\text{ in }\Omega,\\
u & = & 0 &\text{ on }\partial\Omega.
\end{array}
\right.
\end{equation}
We note that $\lambda$ defines a nonnegative Radon measure on $\Omega$; see section \ref{sec:notation}. Moreover, $\mathrm{supp}(\lambda)\subseteq \{x\in \Omega : u(x) = g(x)\}$.

Let $M:=L^2(\Omega)$, $V := V_{1}$, and $K:=K_{1}$, where $V_{1}$ and $K_{1}$ are defined in \eqref{def:space_V} and \eqref{def:cone_K}, respectively.
We introduce the linear functions
\begin{equation*}
F:M\to \mathbb{R}, \quad F(v):= -\ip{f}{v}_{\Omega}, 
\quad \text{and} \quad
G:K \to \mathbb{R}, \quad G((\boldsymbol\tau,\mu)):= \int_{\Omega}g\, \mathrm{d}\mu.
\end{equation*}
We note that $\int_{\Omega}g\, \mathrm{d}\mu < \infty$. 
In fact, since $\mu \geq 0$ and $g|_{\partial\Omega} \leq0$, we have that 
\begin{align*}
\int_{\Omega}g\, \mathrm{d}\mu \leq \int_{\Omega}\max\{g,0\}\mathrm{d}\mu = \langle \mu, \max\{g,0\}\rangle \leq \|\mu\|_{-1}\|\max\{g,0\}\|_{1}.
\end{align*}
Define the continuous bilinear forms $a : V \times V \to \mathbb{R}$ and $b : V \times M \to \mathbb{R}$  by
\begin{equation*}
a((\boldsymbol\tau_{1},\mu_{1}), (\boldsymbol\tau_{2},\mu_{2})) := \ip{\boldsymbol\tau_{1}}{\boldsymbol\tau_{2}}_{\Omega},\qquad
b((\boldsymbol\tau,\mu), v)  :=  \ip{\text{div}\,\boldsymbol\tau + \mu}{v}_{\Omega}.
\end{equation*}

With all the previous ingredients at hand, we derive a mixed formulation for problem \eqref{eq:strong_mixed_obst}.
First, test the second equation with $\btau - \bsigma$ with  $\btau \in [L^2(\Omega)]^n$ and use the distributional definition of $\div$ to obtain
\begin{equation}\label{eq:identity_2nd_eq}
  \ip{\bsigma}{\btau - \bsigma}_\Omega + \dual{\div(\btau - \bsigma)}{u} = 0. 
\end{equation}
Second, condition $u\geq g$ a.e. in $\Omega$ and $\lambda(u-g) = 0$ yield $\int_{\Omega}(u-g)\mathrm{d}(\mu-\lambda)\geq 0$ for all $\mu\in H^{-1}(\Omega)$ with $\mu\geq 0$. Combining this inequality with identity \eqref{eq:identity_2nd_eq} we obtain
\begin{align*}
  \ip{\bsigma}{\btau - \bsigma}_\Omega + \dual{\div(\btau-\bsigma)+\mu-\lambda}{u} \geq  \int_{\Omega}{g}\,\mathrm{d}(\mu-\lambda) \quad\text{for all } \btau\in [L^2(\Omega)]^n, \, \mu\in H^{-1}(\Omega), \, \mu\geq 0. 
\end{align*}
Restricting the test functions to the convex cone $K$, we get that $\dual{\div(\btau-\bsigma)+\mu-\lambda}{u} = \ip{\div(\btau-\bsigma)+\mu-\lambda}{u}_\Omega$.
Finally, testing the first equation in~\eqref{eq:strong_mixed_obst} with $v\in M$, we end up with the following mixed variational formulation: 
Find $(\boldsymbol\sigma,\lambda,u)\in K\times M$ such that
\begin{equation}\label{eq:weak_mixed}
\left\{
\begin{array}{lcl}
a((\boldsymbol\sigma,\lambda), (\boldsymbol\tau,\mu) - (\boldsymbol\sigma, \lambda)) + b((\boldsymbol\tau,\mu) - (\boldsymbol\sigma, \lambda), u) & \geq & G((\boldsymbol\tau,\mu) - (\boldsymbol\sigma, \lambda))\\
b((\boldsymbol\sigma, \lambda), v) & = &  F(v) 
\end{array}
\right.
\end{equation}
for all $(\boldsymbol\tau,\mu,v)\in K\times M$. In what follows we study the well-posedness of \eqref{eq:weak_mixed}. 

\begin{proposition}[equivalence]\label{prop:equivalent_obstacle}
  Problems~\eqref{eq:weakformObstacle} and~\eqref{eq:weak_mixed} are equivalent in the following sense: 
  If $u\in \mathsf{V}$ denotes the unique solution of~\eqref{eq:weakformObstacle}, then $(\nabla u,-\Delta u-f,u)\in K\times M$ solves~\eqref{eq:weak_mixed}. 
   Conversely, if $(\bsigma,\lambda,u)\in K\times M$ solves~\eqref{eq:weak_mixed}, then $u\in \mathsf{V}$ and solves~\eqref{eq:weakformObstacle}. Moreover, $\bsigma = \nabla u$, and $\lambda=-\Delta u-f$. In particular, formulation~\eqref{eq:weak_mixed} admits a unique solution.
\end{proposition}
\begin{proof}
  Let $u\in \mathsf{V}$ solve~\eqref{eq:weakformObstacle}. By construction $(\bsigma,\lambda,u) := (\nabla u,-\Delta u-f,u)\in K\times M$ solves~\eqref{eq:weak_mixed}.

  Conversely, let $(\bsigma,\lambda,u)\in K\times M$ be a solution of~\eqref{eq:weak_mixed}. Choosing $(\bsigma \pm \tilde{\btau},\lambda,0)$ with $\tilde{\btau}\in\Hdivset\Omega$ as test function we get that
  \begin{align*}
    \pm a( (\bsigma,\lambda), (\tilde{\btau},0) ) \pm b( (\tilde{\btau},0),u) \geq 0 \quad\forall \tilde{\btau}\in \Hdivset\Omega.
  \end{align*}
  This is equivalent to $\ip{\bsigma}{\tilde{\btau}}_\Omega + \ip{\div\tilde{\btau}}{u}_\Omega = 0$ for all $\tilde{\btau}\in \Hdivset\Omega$ which, in turns, implies that
  \begin{align*}
    u\in H_0^1(\Omega) \quad\text{with}\quad
    \nabla u = \bsigma. 
  \end{align*}
  Given $\mu\in H^{-1}(\Omega)$ with $\mu\geq 0$, choose $\btau\in [L^2(\Omega)]^n$ such that $\div\btau+\mu\in L^2(\Omega)$ (such an element exists since $\div\colon [L^2(\Omega)]^n\to H^{-1}(\Omega)$ is surjective). 
   Then, from $\nabla u = \bsigma$ and the distributional definition of $\div(\cdot)$ it follows that $\ip{\bsigma}{\btau-\bsigma}_{\Omega}+\dual{\div(\btau-\bsigma)}{u} = 0$. 
  The use of this identity in the first equation of \eqref{eq:weak_mixed} yields
  \begin{align*}
    \dual{\mu-\lambda}{u} = \ip{\bsigma}{\btau-\bsigma}_{\Omega}+\ip{\div(\btau-\bsigma)+\mu-\lambda}{u}_\Omega \geq \int_{\Omega}{g}\,\mathrm{d}(\mu-\lambda).
  \end{align*}
  That is $\int_{\Omega}(u - g)\mathrm{d}(\mu-\lambda)\geq 0$ for all $\mu\in H^{-1}(\Omega)$ with $\mu\geq 0$. This implies that $\int_{\Omega}(u-g)\mathrm{d}\lambda = 0$ and $u\geq g$ a.e. in $\Omega$. 
  Therefore, $u\in \mathsf{V}$.
  On the other hand, testing~\eqref{eq:weak_mixed} with $(\bsigma,\lambda,v)$ ($v\in M$), we get that $\div\bsigma+\lambda=-f$ a.e. in $\Omega$, or equivalently, $\lambda = -\div\bsigma-f = -\Delta u -f$ a.e. in $\Omega$. 
  This identity, the inequality $\int_{\Omega}(v - g)\mathrm{d}\lambda \geq 0$ (which holds for every  $v\in \mathsf{V}$), and $\int_{\Omega}(u - g)\mathrm{d}\lambda = 0$, give
  \begin{align*}
    0\leq \int_{\Omega}(v-u+u-g)\mathrm{d}\lambda = \dual{-\Delta u-f}{v-u} = \ip{\nabla u}{\nabla (v-u)}_\Omega - \ip{f}{v-u}_\Omega \quad\forall v\in \mathsf{V}.
  \end{align*}
  This is variational inequality~\eqref{eq:weakformObstacle} and therefore finishes the proof.
\end{proof}


\subsection{Finite element approximation}\label{sec:fe_approx}
In what follows we assume that $\Omega$ is an open and bounded polygonal/polyhedral domain with Lipschitz boundary. 

Using the notation introduced in section \ref{sec:fem_and_interp}, we define the discrete spaces $M_{h}:= \mathcal{P}^{0}(\cT)$ and $V_{h} := \text{RT}^{0}(\cT)\times \mathcal{P}^{0}(\cT)$, and the set
$$
K_{h}:=\{(\boldsymbol\tau_{h},\mu_{h}) \in V_{h} : \mu_{h}|^{}_{T} \geq 0 \quad \forall T\in \cT\}.
$$
We thus propose the following finite element approximation for problem \eqref{eq:weak_mixed}:
Find $(\boldsymbol\sigma_{h},\lambda_{h},u_{h})\in K_{h}\times M_{h}$ such that
\begin{equation}\label{eq:weak_mixed_discrete}
\left\{
\begin{array}{lcl}
a((\boldsymbol\sigma_{h},\lambda_{h}), (\boldsymbol\tau_{h},\mu_{h}) - (\boldsymbol\sigma_{h}, \lambda_{h})) + b((\boldsymbol\tau_{h},\mu_{h}) - (\boldsymbol\sigma_{h}, \lambda_{h}), u_{h}) & \geq & G((\boldsymbol\tau_{h},\mu_{h}) - (\boldsymbol\sigma_{h}, \lambda_{h}))\\
b((\boldsymbol\sigma_{h}, \lambda_{h}), v_{h})  & = &  F(v_{h})
\end{array}
\right.
\end{equation}
for all $(\boldsymbol\tau_{h},\mu_{h},v_{h})\in K_{h}\times M_{h}$. 

\begin{theorem}[well-posedness]
  Problem \eqref{eq:weak_mixed_discrete} is well posed.
\end{theorem}
\begin{proof}
It suffices to verify the four assumptions of Theorem \ref{thm:wp_discrete_generalized}. 
We note that 
\begin{equation*}
a((\btau,\mu),(\btau,\mu)) = \ip{\btau}{\btau}_{\Omega} = \|\btau\|_{\Omega}^2 \geq 0 \qquad \forall \,(\btau,\mu)\in V.
\end{equation*}
Hence, assumption \eqref{eq:assumption_non_negativity_a} holds. 
To verify \eqref{eq:assumption_coercivity_a}, we note that $N=\{(\btau,\mu)\in V :  b((\btau,\mu),v) = 0 \quad \forall \,v\in M\} =\{(\btau,\mu)\in V :  \div \btau + \mu = 0 \text{ a.e.  in }\Omega\}$. 
It thus follows that 
\begin{equation*}
a((\btau,\mu),(\btau,\mu)) = \|\btau\|_{\Omega}^2 = \|\btau\|_{\Omega}^2 + \|\div \btau + \mu\|_{\Omega}^{2} = \|(\btau,\mu)\|_{V}^{2}  \qquad \forall \,(\btau,\mu)\in N.
\end{equation*}
Assumption \eqref{eq:conforming_condition} is verified since $N_{h}=\{(\btau_{h},\mu_{h})\in V_{h}: b((\btau_{h},\mu_{h}),v_{h}) =  0 \quad \forall \, v\in M_{h}\} = \{(\btau_{h},\mu_{h})\in V_{h}:  \div \btau_{h} + \mu_{h} = 0 \text{ a.e.  in }\Omega\} \subset N$.
Finally, to demonstrate that assumption \eqref{eq:discrete_inf_sup} holds, we consider the subspace $W_{h} = \text{RT}^{0}(\cT)\times \{0\} \subset V_{h}$ which satisfies that $W_{h} \subset K_{h}$. 
Hence, by using the discrete inf-sup condition of Raviart--Thomas spaces \cite[Theorem 13.2]{MR1115239}, we conclude that
\begin{equation*}
\sup_{(\btau_{h},\mu_{h})\in W_{h}}\frac{b((\btau_{h},\mu_{h}),v_{h})}{\|(\btau_{h},\mu_{h})\|_{V}} 
=
\sup_{\btau_{h}\in \text{RT}^{0}(\cT)}\frac{\ip{\div \btau_{h}}{v_{h}}_{\Omega}}{\|\btau_{h}\|_{\Hdivset\Omega}} 
 \geq \tilde{\beta} \|v_{h}\|_{M} \quad \forall \, v_{h}\in M_{h},
\end{equation*}
with $\tilde{\beta} > 0$ independent of the discretization parameter $h$. 
This concludes the proof.
\end{proof}

Before presenting error estimates, we investigate some interesting properties of the mixed scheme~\eqref{eq:weak_mixed_discrete}.
  \begin{proposition}[properties of the solution]\label{prop:properties}
  Let $(\bsigma_h,\lambda_h,u_h)\in K_h\times M_h$ denote the solution to~\eqref{eq:weak_mixed_discrete}. 
  It possesses the following properties:
  \begin{itemize}
    \item[$\mathrm{(i)}$] $\div\bsigma_h + \lambda_h = -\Pi_h^0 f$, 
    \item[$\mathrm{(ii)}$] $\ip{\bsigma_h}{\btau_h}_\Omega + \ip{\div\btau_h}{u_h}_\Omega = 0$ for all $\btau_h\in \RT^0(\cT)$, 
    \item[$\mathrm{(iii)}$] $\ip{\lambda_h}{u_h-g}_\Omega  = 0$ or, equivalently, $\lambda_h(x)\cdot(u_h(x)-\Pi_h^0 g(x)) = 0$ for a.e. $x\in\Omega$, and
    \item[$\mathrm{(iv)}$] $u_h- \Pi_h^0 g\geq 0$.
  \end{itemize}
\end{proposition}
\begin{proof}
We observe that the second equation in~\eqref{eq:weak_mixed_discrete} implies that
  \begin{align*}
    \ip{(\div\bsigma_h+\lambda_h) + f}{v_h}_\Omega = 0 \quad\forall v_h\in M_h = \mathcal{P}^{0}(\cT),
  \end{align*}
  which shows $\mathrm{(i)}$. 
  For the proof of $\mathrm{(ii)}$, take $(\btau_h,\mu_h,v_h)=(\bsigma_{h}\pm\bchi_h,\lambda_h,0)$ in~\eqref{eq:weak_mixed_discrete}, with $\bchi_h\in\RT^0(\cT)$ arbitrary. This gives
  \begin{align*}
    \pm\ip{\bsigma_h}{\bchi_h}_\Omega \pm \ip{\div\bchi_h}{u_h}_{\Omega} \geq 0 \quad\forall \bchi_h\in\RT^0(\cT), 
  \end{align*}
  or, equivalently, $\ip{\bsigma_h}{\bchi_h}_\Omega + \ip{\div\bchi_h}{u_h}_{\Omega} = 0$ for all $\bchi_h\in \RT^0(\cT)$. 
  Using $\mathrm{(i)}$ and $\mathrm{(ii)}$, we see that~\eqref{eq:weak_mixed_discrete} simplifies to
  \begin{align*}
    \ip{\mu_h-\lambda_h}{u_h-g}_\Omega \geq 0 \quad\forall \mu_h\in M_h, \, \mu_h\geq 0.
  \end{align*}
  In the last inequality, choose $\mu_h = 2\lambda_h$ and then $\mu_h = 0$ resulting in $\ip{\pm\lambda_h}{u_h-g}_\Omega \geq 0$ or, equivalently, $\ip{\lambda_h}{u_h-g}_\Omega = 0$, which proves the first identity in $\mathrm{(iii)}$. 
  We notice that, if $\mathrm{(iv)}$ holds, then we have the second identity in $\mathrm{(iii)}$.
  Finally, to prove $\mathrm{(iv)}$, we take $\mu_h = \nu_h+\lambda_h\geq 0$ with $\nu_h\in M_h$, $\nu_h\geq 0$ arbitrary, and obtain $\ip{\nu_h}{u_h-\Pi_h^0g}_\Omega = \ip{\nu_h}{u_h-g}_\Omega \geq 0$. 
  This implies $\mathrm{(iv)}$ and concludes the proof.
\end{proof}

\begin{theorem}[a priori estimate]\label{thm:bestapprox}
Let $(\boldsymbol\sigma,\lambda,u)\in K\times M$ be the unique solution to problem \eqref{eq:weak_mixed} and let $(\boldsymbol\sigma_{h},\lambda_{h},u_{h})\in K_{h}\times M_{h}$ be its finite element approximation obtained as the unique solution to \eqref{eq:weak_mixed_discrete}. 
Then, we have that:
\begin{align*}
\|(\bsigma,\lambda,u)-(\bsigma_h,\lambda_h,u_h)&\|_{V\times M}^{2}
 \lesssim 
\min_{v_{h}\in M_{h}}\|u - v_{h}\|_{\Omega}^{2} \\ 
& + \min_{(\boldsymbol\tau_{h},\mu_{h})\in K_{h}}\left\{\|\boldsymbol\sigma - \boldsymbol\tau_{h}\|_{\Omega}^{2} + \|\textnormal{div}(\boldsymbol\sigma - \boldsymbol\tau_{h}) + \lambda - \mu_{h}\|_{\Omega}^{2} + \int_{\Omega}(u - g)\mathrm{d}(\mu_{h} - \lambda)\right\}.
\end{align*}
\end{theorem}
\begin{proof}

  The proof is split into three steps.

  \underline{Step 1.} (estimation of $\|u - u_{h}\|_{\Omega}$) Estimate $\|u - u_{h}\|_{\Omega}^{2} \leq 2\|u - \Pi_{h}^{0}u\|_{\Omega}^{2} + 2\| \Pi_{h}^{0}u - u_{h}\|_{\Omega}^{2}$ follows from the triangle inequality.
We note that $\|u - \Pi_{h}^{0}u\|_{\Omega}^{2} = \min_{v_{h}\in M_{h}}\|u - v_{h}\|_{\Omega}^{2}$. Hence, 
\begin{equation}\label{eq:estimate_u-uh1}
\|u - u_{h}\|_{\Omega}^{2} \leq 2 \min_{v_{h}\in M_{h}}\|u - v_{h}\|_{\Omega}^{2} + 2\| \Pi_{h}^{0}u - u_{h}\|_{\Omega}^{2}.
\end{equation}
To estimate the term $\|\Pi_{h}^{0}u - u_{h}\|_{\Omega}^{2}$ in \eqref{eq:estimate_u-uh1}, we introduce the auxiliary variable $\mathsf{v}\in H_0^{1}(\Omega)$, defined as the unique solution to $\Delta \mathsf{v} = \Pi_{h}^{0} u - u_{h}$ in $\Omega$ and $\mathsf{v} = 0$ on $\partial\Omega$. 
We also define $\btau := \nabla \mathsf{v}$ and $\btau_{h} : = \Pi_{h}^{\textrm{div}}\btau$. 
 Standard properties of the operator $\Pi_{h}^{\textrm{div}}$, property $(\mathrm{ii})$ in Proposition \ref{prop:properties}, and the fact that $\bsigma = \nabla u$ yield
\begin{multline*}
\|\Pi_{h}^{0}u - u_{h}\|_{\Omega}^{2} 
= (u - u_{h},\Pi_{h}^{0}( \div \btau))_{\Omega}
= (u - u_{h},\div\btau_{h})_{\Omega} 
= - (\boldsymbol \sigma - \boldsymbol \sigma_{h},\btau_{h})_{\Omega} \\
\leq \|\boldsymbol \sigma - \boldsymbol \sigma_{h}\|_{\Omega}\|\btau_{h}\|_{\Omega} 
\lesssim \|\boldsymbol \sigma - \boldsymbol \sigma_{h}\|_{\Omega}\|\textrm{div }\btau\|_{\Omega} 
= \|\boldsymbol \sigma - \boldsymbol \sigma_{h}\|_{\Omega}\|\Pi_{h}^{0}u - u_{h}\|_{\Omega}. 
\end{multline*}
Hence, $\|\Pi_{h}^{0}u - u_{h}\|_{\Omega} \lesssim  \|\boldsymbol \sigma - \boldsymbol \sigma_{h}\|_{\Omega}$ which implies, in view of \eqref{eq:estimate_u-uh1}, 
\begin{equation}\label{eq:estimate_u-uh2}
\|u - u_{h}\|_{\Omega}^{2} 
\lesssim 
\min_{v_{h}\in M_{h}}\|u - v_{h}\|_{\Omega}^{2} + \|\boldsymbol \sigma - \boldsymbol \sigma_{h}\|_{\Omega}^{2}.
\end{equation}

\underline{Step 2.} (estimation of $\|\boldsymbol \sigma - \boldsymbol \sigma_{h}\|_{\Omega}$) Let us start by writing 
\begin{align}\label{eq:estimate_sigma_sigmah}
\|\boldsymbol \sigma - \boldsymbol \sigma_{h}\|_{\Omega}^{2} = &~ [(\boldsymbol \sigma - \boldsymbol \sigma_{h}, \boldsymbol \sigma - \boldsymbol \sigma_{h})_{\Omega} + (\textrm{div }(\boldsymbol \sigma - \boldsymbol \sigma_{h}) + \lambda - \lambda_{h},u - u_{h})_{\Omega}] \hspace{-0.3cm} \\
 & + (\textrm{div }(\boldsymbol \sigma - \boldsymbol \sigma_{h}) + \lambda - \lambda_{h}, u_{h} - u)_{\Omega}  =: \mathsf{I} + \mathsf{II}. \nonumber
\end{align}
We estimate $\mathsf{II}$. Recall from Proposition~\ref{prop:properties} that $\div\bsigma_h+\lambda_h = -\Pi_h^0f$. Consequently, 
\begin{equation}\label{eq:f-pi_hf}
 \div(\bsigma - \bsigma_{h}) + \lambda - \lambda_{h} = \Pi_{h}^{0}f - f.
\end{equation}
We thus use \eqref{eq:f-pi_hf} and the orthogonality of $\Pi_{h}^{0}$ to conclude that
\begin{align}\label{eq:estimate_II}
\mathsf{II} = (\Pi_{h}^{0}f - f,\Pi_{h}^{0}u - u)_{\Omega} 
& \leq \frac{1}{2}\|\Pi_{h}^{0}f - f\|_{\Omega}^{2} + \frac{1}{2}\|u - \Pi_{h}^{0}u\|_{\Omega}^{2}\\
& =  \frac{1}{2}\min_{(\boldsymbol\tau_{h},\mu_{h})\in V_{h}}\|\textnormal{div}(\boldsymbol\sigma - \boldsymbol\tau_{h}) + \lambda - \mu_{h}\|_{\Omega}^{2} + \frac{1}{2}\min_{v_{h}\in M_{h}}\|u - v_{h}\|_{\Omega}^{2}. \nonumber
\end{align}

We now control the term $\mathsf{I}$. To simplify the presentation of the material, we introduce the bilinear form $c:  (K\times M) \times (K\times M) \to \mathbb{R}$  defined by 
\begin{align*}
c\{(\boldsymbol\tau_{1},\mu_{1},v_{1}), (\boldsymbol\tau_{2},\mu_{2},v_{2})\}
:= (\btau_{1}, \boldsymbol\tau_{2})_{\Omega} + (\textrm{div } \boldsymbol\tau_{2} + \mu_{2}, v_{1})_{\Omega}.
\end{align*}
Hence, for all $(\boldsymbol\tau_{h},\mu_{h},v_{h}) \in K_{h}\times M_{h}$, it follows that
\begin{align*}
\mathsf{I} 
= ~  &
c\{(\boldsymbol\sigma - \boldsymbol\sigma_{h},\lambda - \lambda_{h}, u - u_h),(\boldsymbol\sigma -\boldsymbol\sigma_{h},\lambda - \lambda_{h}, u - u_h)\} = c\{(\boldsymbol\sigma,\lambda,u),(\boldsymbol\sigma -\boldsymbol\sigma_{h},\lambda - \lambda_{h}, u - u_h)\} \\
&- c\{(\boldsymbol\sigma_{h},\lambda_{h},u_{h}), (\boldsymbol\sigma -\btau_{h},\lambda - \mu_{h}, u - v_h)\}
- c\{(\boldsymbol\sigma_{h},\lambda_{h},u_{h}), (\btau_{h} -\boldsymbol\sigma_{h},\mu_{h} - \lambda_{h}, v_{h} - u_{h})\}=: \mathsf{I}_{1} - \mathsf{I}_{2} -  \mathsf{I}_{3}.
\end{align*}
In view of \eqref{eq:weak_mixed} and \eqref{eq:weak_mixed_discrete}, we obtain that 
\begin{align*}
  \mathsf{I}_{1} \leq  \int_{\Omega}g\,\mathrm{d}(\lambda - \lambda_{h}), \quad\text{and}\quad
\mathsf{I}_{3} \geq  \ip{\mu_{h} - \lambda_{h}}{g}_{\Omega},
\end{align*}
respectively. Consequently, for all $\mu_{h}$, we have
\begin{equation}\label{eq:estimate_I_ineq}
\mathsf{I}
\leq
\int_{\Omega}g\,\mathrm{d}(\lambda - \lambda_{h}) -  \mathsf{I}_{2} -  \ip{\mu_{h} - \lambda_{h}}{g}_{\Omega}
=
\int_{\Omega}g\,\mathrm{d}(\lambda - \mu_{h}) -  \mathsf{I}_{2}.
\end{equation}
To estimate $\mathsf{I}_{2}$ in \eqref{eq:estimate_I_ineq} we use the identity
\begin{equation*}
c\{(\boldsymbol\sigma,\lambda,u), (\boldsymbol\sigma - \boldsymbol\tau_{h},\lambda - \mu_{h}, u - v_{h})\}
=
\langle \lambda - \mu_{h}, u \rangle \quad  \forall (\boldsymbol\tau_{h},\mu_{h},v_{h}) \in K_{h}\times M_{h},
\end{equation*}
which follows from an integration by parts formula and $\boldsymbol\sigma = \nabla u$, to conclude that
\begin{equation*}
-\mathsf{I}_{2} = c\{(\boldsymbol\sigma - \boldsymbol\sigma_{h},\lambda - \lambda_{h}, u - u_{h}),(\boldsymbol\sigma - \boldsymbol\tau_{h},\lambda - \mu_{h}, u - v_{h})\} - \langle \lambda - \mu_{h}, u \rangle \quad  \forall (\boldsymbol\tau_{h},\mu_{h},v_{h}) \in K_{h}\times M_{h}.
\end{equation*} 
Using the latter identity in \eqref{eq:estimate_I_ineq} and applying basic inequalities we obtain 
\begin{align}\label{eq:estimate_I_ineqii}
\mathsf{I} 
\leq & ~ 
c\{(\boldsymbol\sigma - \boldsymbol\sigma_{h},\lambda - \lambda_{h}, u - u_{h}),(\boldsymbol\sigma - \boldsymbol\tau_{h},\lambda - \mu_{h}, u - v_{h})\} + \int_{\Omega}(u - g)\mathrm{d}(\mu_{h} - \lambda) \\ 
\leq & ~
\|\boldsymbol \sigma - \boldsymbol \sigma_{h}\|_{\Omega} \|\boldsymbol \sigma - \boldsymbol \tau_{h}\|_{\Omega} + \|u - u_{h}\|_{\Omega}\|\textnormal{div}(\boldsymbol\sigma - \boldsymbol\tau_{h}) + \lambda - \mu_{h}\|_{\Omega} + \int_{\Omega}(u - g)\mathrm{d}(\mu_{h} - \lambda). \nonumber
\end{align}
Hence, the combination of the estimates \eqref{eq:estimate_sigma_sigmah}, \eqref{eq:estimate_II}, and \eqref{eq:estimate_I_ineqii} results in the bound
\begin{align*}
&\|\boldsymbol \sigma - \boldsymbol \sigma_{h}\|_{\Omega}^{2}
\leq
\min_{v_{h}\in M_{h}}\|u - v_{h}\|_{\Omega}^{2} + \min_{(\boldsymbol\tau_{h},\mu_{h})\in V_{h}}\|\textnormal{div}(\boldsymbol\sigma - \boldsymbol\tau_{h}) + \lambda - \mu_{h}\|_{\Omega}^{2} + \|\boldsymbol \sigma - \boldsymbol \sigma_{h}\|_{\Omega} \|\boldsymbol \sigma - \boldsymbol \tau_{h}\|_{\Omega} \\
~  & + \|u - u_{h}\|_{\Omega}\|\textnormal{div}(\boldsymbol\sigma - \boldsymbol\tau_{h}) + \lambda - \mu_{h}\|_{\Omega} + \int_{\Omega}(u - g)\mathrm{d}(\mu_{h} - \lambda) 
  \lesssim
(1 + \delta)\min_{v_{h}\in M_{h}}\|u - v_{h}\|_{\Omega}^{2} \\
 ~ & + (1 + \delta^{-1})\min_{(\boldsymbol\tau_{h},\mu_{h})\in K_{h}}\left\{\|\textnormal{div}(\boldsymbol\sigma - \boldsymbol\tau_{h}) + \lambda - \mu_{h}\|_{\Omega}^{2}  +  \|\boldsymbol \sigma - \boldsymbol \tau_{h}\|_{\Omega}^{2} + \int_{\Omega}(u - g)\mathrm{d}(\mu_{h} - \lambda) \right\} + \delta\|\boldsymbol \sigma - \boldsymbol \sigma_{h}\|_{\Omega}^{2},
\end{align*}
where we also have used Young's inequality with $\delta>0$, and estimate \eqref{eq:estimate_u-uh2}. Finally, taking $\delta$ small enough we conclude the bound
\begin{align}\label{eq:final_sigma_sigmah}
\|\boldsymbol \sigma - \boldsymbol \sigma_{h}\|_{\Omega}^{2} 
\lesssim &
\min_{v_{h}\in M_{h}}\|u - v_{h}\|_{\Omega}^{2} \\
& ~ + \min_{(\boldsymbol\tau_{h},\mu_{h})\in K_{h}} \left\{\|\textnormal{div}(\boldsymbol\sigma - \boldsymbol\tau_{h}) + \lambda - \mu_{h}\|_{\Omega}^{2}  +  \|\boldsymbol \sigma - \boldsymbol \tau_{h}\|_{\Omega}^{2} + \int_{\Omega}(u - g)\mathrm{d}(\mu_{h} - \lambda) \right\}. \nonumber
\end{align}

\underline{Step 3.} For the final step we note that we have already used in the previous step that 
  \begin{align*}
  \|\div(\bsigma-\bsigma_h)+\lambda-\lambda_h\|_\Omega &= \min_{(\btau_h,\mu_h)\in V_h} \|\div\bsigma+\lambda-\div\btau_h-\mu_h\|_\Omega
  \\ &\leq \min_{(\btau_h,\mu_h)\in K_h} \|\div\bsigma+\lambda-\div\btau_h-\mu_h\|_\Omega.
\end{align*}
Combining this estimate with \eqref{eq:estimate_u-uh2} and \eqref{eq:final_sigma_sigmah} concludes the proof.
\end{proof}

From the previous theorem we derive convergence rates for sufficiently smooth solutions. 

\begin{corollary}[convergence rate]\label{cor:convergence}
 Suppose that $\Omega$ is convex, $g\in H^2(\Omega)$. Then, under the framework of Theorem~\ref{thm:bestapprox}, $u\in H^2(\Omega)\cap H_0^1(\Omega)$ and 
  \begin{align*}
    \|(\bsigma,\lambda,u)-(\bsigma_h,\lambda_h,u_h)\|_{V\times M} &\lesssim  h(\|u\|_2 + \|g\|_2 + \|f\|_\Omega) + \|f-\Pi_h^0 f\|_\Omega.
  \end{align*}
\end{corollary}
\begin{proof}
Regularity results under the assumptions considered are well known for problem \eqref{eq:weakformObstacle}; see, e.g.~\cite{MR0239302}. 
Hence, $u\in H^2(\Omega)\cap H_0^1(\Omega)$ and thus $\bsigma = \nabla u\in [H^1(\Omega)]^n$. 
As a consequence, $\btau_h = \Pi_h^{\div}\bsigma$ is well defined, $\div\btau_h = \Pi_h^0\div\bsigma$, and $\|\bsigma-\btau_h\|_\Omega \lesssim h\|u\|_2$.
We also have that, $\lambda = -\Delta u -f \in L^2(\Omega)$, so $\mu_h = \Pi_h^0\lambda$ is well defined. We note that $\mu_h\geq 0$, thus, $(\btau_h,\mu_h)\in K_h$. Therefore, with $(\btau_h,\mu_h)$ at hand, we choose $v_h = \Pi_h^0 u$ and invoke Theorem~\ref{thm:bestapprox} to arrive at
  \begin{align*}
    \|(\bsigma,\lambda,u)-(\bsigma_h,\lambda_h,u_h)\|_{V\times M}^2 &\lesssim \|\bsigma-\btau_h\|_\Omega^2 + \|u-v_h\|_\Omega^2 + \|\div(\bsigma-\btau_h)+\lambda-\mu_h\|_\Omega^2 + \ip{\mu_h-\lambda}{u-g}_{\Omega}
    \\
    &\lesssim h^2\|u\|_2^2 + h^2\|u\|_1^2 + \|(1-\Pi_h^0)f\|_\Omega^2 + \ip{\Pi_h^0\lambda-\lambda}{u-g}_{\Omega}.
  \end{align*}
  Term $\ip{\Pi_h^0\lambda-\lambda}{u-g}_{\Omega}$ is estimated as in~\cite[Proof of Theorem~13]{MR4050087} giving 
  \begin{align*}
    \ip{\Pi_h^0\lambda-\lambda}{u-g}_{\Omega} \lesssim h^2(\|u\|_2+\|g\|_2+\|\lambda\|_\Omega)^2 \lesssim h^2(\|u\|_2+\|g\|_2 + \|f\|_\Omega).
  \end{align*}
  This finishes the proof.
\end{proof}


\subsection{Error estimates in weaker norm}
For deriving convergence rates in a weaker norm, we consider the following auxiliary problem: Find $\widetilde u \in \mathsf{V}$ such that
\begin{align*}
  \ip{\nabla \widetilde u}{\nabla(v-\widetilde u)}_\Omega \geq \ip{\Pi_h^0f}{v-\widetilde u}_\Omega \quad\forall v\in \mathsf{V}.
\end{align*}
This is the membrane obstacle problem replacing $f\in L^2(\Omega)$ with its piecewise constant approximation $\Pi_h^0 f$.
Setting $\widetilde\bsigma=\nabla\widetilde u$, $\widetilde\lambda = -\Delta\widetilde u-\Pi_h^0f$, we see that $(\widetilde\bsigma,\widetilde\lambda,\widetilde u)\in K\times M$ solves the mixed formulation~\eqref{eq:weak_mixed} with $f$ replaced by $\Pi_h^0 f$.
\begin{lemma}\label{lem:auxapriori}
  Let $(\bsigma,\lambda,u)\in K\times M$ denote the solution to problem \eqref{eq:weak_mixed} and let $(\widetilde\bsigma,\widetilde\lambda,\widetilde u)\in K\times M$ be the unique solution to problem \eqref{eq:weak_mixed} with $f$ replaced by $\Pi_h^0f$.
   Then, 
  \begin{align*}
    \|\bsigma-\widetilde\bsigma\|_{\Omega} + \|u-\widetilde u\|_\Omega + \|\lambda-\widetilde\lambda\|_{-1} \lesssim \|(1-\Pi_h^0)f\|_{-1} \lesssim  \|h_{\cT}(1-\Pi_h^0)f\|_\Omega.
  \end{align*}
\end{lemma}
\begin{proof}
  Definition of $\bsigma$, $\widetilde\bsigma$, $\lambda$, and $\widetilde\lambda$, the triangle inequality, and Friedrich's inequality yield $\|\bsigma-\widetilde\bsigma\|_\Omega + \|u-\widetilde u\|_{\Omega} + \|\lambda-\widetilde\lambda\|_{-1} \lesssim \|u-\widetilde u\|_{1} + \|f-\Pi_h^0f\|_{-1}$. 
  It remains to show that $\|u-\widetilde u\|_{1} \leq \|f-\Pi_h^0f\|_{-1}$. 
  Using that $\lambda = -\Delta u -f$, $\widetilde\lambda = -\Delta\widetilde u - \Pi_h^0 f$, $\int_{\Omega}(u-g)\mathrm{d}\lambda = 0 = \int_{\Omega}(\widetilde{u}-g)\mathrm{d}\widetilde{\lambda}$ as well as $\lambda,\widetilde\lambda,u-g,\widetilde u-g\geq 0$ we obtain
  \begin{align*}
    \|u-\widetilde u\|_{1}^2 &= \dual{-\Delta u+\Delta\widetilde u}{u-\widetilde u} = \dual{\lambda -\widetilde\lambda + f -\Pi_h^0f}{u-\widetilde u}
    \\
    &= \int_{\Omega}(u-g)\mathrm{d}(\lambda-\widetilde\lambda) + \int_{\Omega}(g-\widetilde u)\mathrm{d}(\lambda-\widetilde\lambda) + \dual{f-\Pi_h^0f}{u-\widetilde u}
    \\
    &= -\int_{\Omega}(u-g)\mathrm{d}\widetilde\lambda + \int_{\Omega}(g-\widetilde u)\mathrm{d}\lambda + \dual{f-\Pi_h^0f}{u-\widetilde u} \leq \dual{f-\Pi_h^0f}{u-\widetilde u} 
    \leq \|f-\Pi_h^0f\|_{-1}\|u-\widetilde u\|_{1}. 
  \end{align*}
 The proof finishes by dividing by $\|u-\widetilde u\|_{1}$ and using the property of $\Pi^{0}_{h}$ from section \ref{sec:fem_and_interp}.
\end{proof}

\begin{lemma}\label{lem:auximproved}
  Let $(\widetilde\bsigma,\widetilde\lambda,\widetilde u)\in K\times M$ be the unique solution to problem \eqref{eq:weak_mixed} with $f$ replaced by $\Pi_h^0f$ and let $(\boldsymbol\sigma_{h},\lambda_{h},u_{h})\in K_{h}\times M_{h}$ be the finite element approximation obtained as the unique solution to \eqref{eq:weak_mixed_discrete}. Then, 
  \begin{align*}
    \|\widetilde\bsigma-\bsigma_h\|_\Omega^2 + \|\widetilde\lambda-\lambda_h\|_{-1}^2 + \|\widetilde u-u_h\|_\Omega^2 &\lesssim \min_{v_h\in M_h} \|\widetilde u-v_h\|_\Omega^2
    \\
    & + \min_{(\btau_h,\mu_h)\in K_h, \, \div\btau_h+\mu_h=-\Pi_h^0f} \left\{\|\widetilde\bsigma-\btau_h\|_\Omega^2 + \int_{\Omega}(\widetilde{u} - g)\mathrm{d}(\mu_{h} - \widetilde{\lambda})\right\}.
  \end{align*}
\end{lemma}
\begin{proof}
 We use the bound $\|\widetilde\bsigma-\bsigma_h\|_\Omega + \|\widetilde\lambda-\lambda_h\|_{-1} \lesssim \|(\widetilde\bsigma-\bsigma_h,\widetilde\lambda-\lambda_h)\|_V$, which follows from the triangle inequality.
  The desired estimate stems from an application of Theorem~\ref{thm:bestapprox} with $(\bsigma,\lambda,u)$ replaced by $(\widetilde\bsigma,\widetilde\lambda,\widetilde u)$ and noting that $\div(\widetilde\bsigma-\btau_h)+\widetilde\lambda-\mu_h=0$ if $\div\btau_h+\mu_h = -\Pi_h^0 f$.
\end{proof}

\begin{theorem}[a priori estimate in weaker norm]\label{thm:error_estimate_membrane}
  Suppose that $\Omega$ is convex, $g\in H^2(\Omega)$. Then, under the framework of Theorem~\ref{thm:bestapprox}, $u\in H^2(\Omega)\cap H_0^1(\Omega)$ and 
  \begin{align*}
    \|\bsigma-\bsigma_h\|_\Omega + \|\lambda-\lambda_h\|_{-1}+\|u-u_h\|_\Omega &\lesssim  h(\|u\|_2 + \|g\|_2 + \|f\|_\Omega).
  \end{align*}
\end{theorem}
\begin{proof}
  Let $\widetilde\bsigma$, $\widetilde\lambda$, and $\widetilde u$ be defined as in Lemma~\ref{lem:auximproved}. 
  By the triangle inequality and Lemmas~\ref{lem:auxapriori} and \ref{lem:auximproved} we get
  \begin{align*}
  \|\bsigma & -\bsigma_h\|_\Omega^{2} + \|\lambda-\lambda_h\|_{-1}^{2} + \|u-u_h\|_\Omega^{2} \\
 & \lesssim
  \|\bsigma -\widetilde\bsigma\|_\Omega^{2} + \|\lambda - \widetilde\lambda\|_{-1}^{2} + \|u - \widetilde{u}\|_\Omega^{2} + \|\widetilde\bsigma-\bsigma_h\|_\Omega^2 + \|\widetilde\lambda-\lambda_h\|_{-1}^2 + \|\widetilde u-u_h\|_\Omega^2 \\
 &  \lesssim  
 h^{2} \|(1-\Pi_h^0)f\|_\Omega^{2} + \min_{v_h\in M_h} \|\widetilde u-v_h\|_\Omega^2   + \min_{(\btau_h,\mu_h)\in K_h, \, \div\btau_h+\mu_h=-\Pi_h^0f} \{\|\widetilde\bsigma-\btau_h\|_\Omega^2 + \ip{\mu_h-\widetilde\lambda}{\widetilde u-g}_{\Omega}\}.
  \end{align*}
Let us concentrate in the last term on the right-hand side of the previous inequality. Since $\ip{\widetilde\lambda}{\widetilde u-g}_{\Omega} = 0$ and $-\dual{\widetilde{\lambda}}{\widetilde{u} - u} \geq 0$, we obtain
  \begin{align*}
    \ip{\mu_h-\widetilde\lambda}{\widetilde u-g}_{\Omega}
	=
    \ip{\mu_h}{\widetilde u-g}_{\Omega}
    &=
    \dual{\mu_h - \lambda}{\widetilde u-u} + \dual{\lambda}{\widetilde u - u}  + \ip{\mu_h}{u-g}_{\Omega}    \\
    &\leq \dual{\mu_h-\lambda}{\widetilde u-u}  + \dual{\lambda - \widetilde \lambda}{\widetilde u - u} + \ip{\mu_h}{u-g}_{\Omega} \\
    &\leq \|\mu_h-\lambda\|_{-1}\|\widetilde u-u\|_1 + \|\lambda - \widetilde \lambda\|_{-1}\|\widetilde u-u\|_1 + \ip{\mu_h}{u-g}_{\Omega}  \\
    &\lesssim \|\mu_h-\lambda\|_{-1}^2 + \|\widetilde u-u\|_1^2 + \|\lambda - \widetilde \lambda\|_{-1}^{2} + \ip{\mu_h}{u-g}_{\Omega}.
  \end{align*}
  This, in conjunction with the identity $\ip{\lambda}{u-g}_{\Omega} = 0$, implies that 
   \begin{align*}
  \|\bsigma-\bsigma_h\|_\Omega^{2} + \|\lambda-\lambda_h\|_{-1}^{2} &+ \|u-u_h\|_\Omega^{2}  
  \lesssim  
  h^{2} \|(1-\Pi_h^0)f\|_\Omega^{2} +  \min_{v_h\in M_h} \|\widetilde u-v_h\|_\Omega^2  + \|\widetilde u-u\|_1^2 + \|\lambda - \widetilde \lambda\|_{-1}^{2} 
   \\
&\, + \min_{(\btau_h,\mu_h)\in K_h, \, \div\btau_h+\mu_h=-\Pi_h^0f} \{\|\widetilde\bsigma-\btau_h\|_\Omega^2 + \|\mu_h-\lambda\|_{-1}^2 + \ip{\mu_h-\lambda}{u-g}_{\Omega}\}.
  \end{align*}
  Using the inequalities $\|\widetilde u-v_h\|_\Omega \leq \|\widetilde u - u\|_\Omega + \| u - v_h\|_\Omega$ and $\|\widetilde\bsigma-\btau_h\|_\Omega \leq \|\widetilde\bsigma - \bsigma\|_\Omega + \|\bsigma - \btau_h\|_\Omega$, Lemma \ref{lem:auxapriori}, and the arguments elaborated in the proof of Corollary~\ref{cor:convergence} we conclude that
  \begin{align*}
  \|\bsigma-\bsigma_h\|_\Omega^{2} + \|\lambda-\lambda_h\|_{-1}^{2} + \|u-u_h\|_\Omega^{2}  
  &\lesssim h^{2} \|(1-\Pi_h^0)f\|_\Omega^{2} + h^2(\|u\|_2+\|g\|_2+\|f\|_\Omega)^2 + \|\lambda-\Pi_h^0\lambda\|_{-1}^2 \\
  &\lesssim h^2(\|u\|_2+\|g\|_2+\|f\|_\Omega)^2.
  \end{align*}
This finishes the proof.
\end{proof}


\section{Mixed FEM for the plate obstacle problem}\label{sec:plate}
In this section we briefly recall the \emph{plate obstacle} problem with rigid obstacle and propose a suitable mixed variational formulation.
This section follows along section~\ref{sec:obst_problem} with modifications made primarily in the definition of spaces and sets, namely, we redefine, e.g., $\mathsf{V}$, $V$, $K$ and their discrete counterparts. 

Let $\Omega\subset \mathbb{R}^{2}$ be an open and bounded polygonal domain with boundary $\partial\Omega$. Let $f\in L^2(\Omega)$ and $g\in H^{2}(\Omega)\cap C^1(\overline\Omega)$ such that $g < 0$ on $\partial\Omega$. The plate obstacle problem then reads as follows: Find $u$ such that
\begin{equation}\label{eq:plate:obstacle_problem}
\Delta^2 u \geq f \text{ in }\Omega, 
\quad u \geq g \text{ in }\Omega,
\quad (\Delta^2 u - f)(u - g) = 0 \text{ in }\Omega,
\quad u = 0 = \partial_\normal u \text{ on }\partial\Omega.
\end{equation}
Let us introduce the set $\mathsf{V}:=\{v\in H_0^{2}(\Omega) : v \geq g \text{ a.e. in }\Omega\}$. Problem \eqref{eq:obstacle_problem} admits a unique solution $u\in \mathsf{V}$ equivalently characterized by the variational inequality: Find $u\in \mathsf{V}$ such that
\begin{equation}\label{eq:plate:weakformObstacle}
\ip{\Grad\nabla u}{\Grad\nabla (v - u)}_{\Omega} \geq  \ip{f}{v - u}_{\Omega} \quad \forall v\in \mathsf{V};
\end{equation}
see, e.g., \cite{MR2904578} and the references therein. 
The upcoming analysis can be extended to the case where problem \eqref{eq:plate:obstacle_problem} is replaced by
\begin{equation*}
\dDiv\mathcal{C}\Grad\nabla u \geq f \text{ in }\Omega, 
\quad u \geq g \text{ in }\Omega,
\quad (\dDiv\mathcal{C}\Grad\nabla u - f)(u - g) = 0 \text{ in }\Omega,
\quad u = 0 = \partial_\normal u \text{ on }\partial\Omega,
\end{equation*}
where $\mathcal{C}:\mathbb{L}^{2}_{\mathrm{sym}}(\Omega) \to \mathbb{L}^{2}_{\mathrm{sym}}(\Omega)$ denotes a positive
definite isomorphism. However, for the sake of readability, we develop the analysis only for \eqref{eq:plate:obstacle_problem}.


\subsection{Mixed variational formulation}\label{sec:plate:mixed_var_for}

Let $u\in \mathsf{V}$ denote the unique solution to the plate obstacle problem \eqref{eq:plate:obstacle_problem}. 
Define $\bM:=\Grad\nabla u$ and $\lambda:=\Delta^2 u - f \in H^{-2}(\Omega)$. 
We notice that $\lambda$ defines a nonnegative Radon measure, since $\lambda \geq 0$; note also that $\mathrm{supp}(\lambda)\subseteq\{x\in\Omega : u(x) = g(x)\}$.
We can thus rewrite problem \eqref{eq:plate:obstacle_problem} as the following second-order problem: 
\begin{align}\label{eq:plate:strong_mixed_obst}
\left\{
\begin{array}{rcll}
\dDiv\bM - \lambda & = & f &\text{ in } \Omega, \\
\bM - \Grad\nabla u & = & 0 &\text{ in }\Omega,\\
\lambda & \geq & 0 &\text{ in }\Omega, \\
u &\geq& g &\text{ in }\Omega,\\
\lambda(u - g) & = & 0 &\text{ in }\Omega,\\
u & = & 0 &\text{ on }\partial\Omega, \\
\partial_\normal u & = & 0 &\text{ on }\partial\Omega.
\end{array}
\right.
\end{align}
Let $M:=L^2(\Omega)$, $V:=V_{2}$, and $K:=K_{2}$, where $V_{2}$ and $K_{2}$ are defined in \eqref{def:space_V} and \eqref{def:cone_K}, respectively. We introduce the linear functions
\begin{align*}
F:M\to \mathbb{R}, \quad F(v):= -\ip{f}{v}_{\Omega}, 
\quad \text{and} \quad
G:K \to \mathbb{R}, \quad G((\bN,\mu)):= \int_{\Omega} g\,\mathrm{d}\mu.
\end{align*}
Using $\mu\geq 0$ and the properties of $M_\varepsilon$ for sufficiently small $\varepsilon$ (Proposition~\ref{prop:max_h02}) we find that
\begin{align*}
  \int_{\Omega}g\, \mathrm{d}\mu \leq \int_{\Omega}M_{\varepsilon}\{g,0\}\mathrm{d}\mu = \langle \mu, M_{\varepsilon}\{g,0\}\rangle \leq \|\mu\|_{-2}\|M_{\varepsilon}\{g,0\}\|_{2}.
\end{align*}
Define the continuous bilinear forms $a : V \times V \to \mathbb{R}$ and $b : V \times M \to \mathbb{R}$  by
\begin{equation*}
  a((\bN_{1},\mu_{1}), (\bN_{2},\mu_{2})) := (\bN_{1},\bN_{2})_{\Omega},\qquad
  b((\bN,\mu), v)  :=  \ip{\dDiv\bN-\mu}{v}_{\Omega}.
\end{equation*}
Based on the previous ingredients, we derive a mixed variational formulation for problem \eqref{eq:plate:strong_mixed_obst}: 
First, given $\bN\in \bbLtwosym(\Omega)$, test the second equation with $\bN - \bM$ and integrate by parts to arrive at
\begin{equation}\label{eq:plate:identity_2nd_eq}
  \ip{\bM}{\bN - \bM}_\Omega - \dual{\dDiv(\bN - \bM)}{u} = 0. 
\end{equation}
Second, $u\geq g$ and $\lambda(u-g) = 0$ yield $\int_{\Omega}(u-g)\mathrm{d}(\mu-\lambda)\geq 0$ for all $\mu\in H^{-2}(\Omega)$ with $\mu\geq 0$. This and \eqref{eq:plate:identity_2nd_eq} give
\begin{align*}
  \ip{\bM}{\bN - \bM}_\Omega - \dual{\dDiv(\bN-\bM)-(\mu-\lambda)}{u} \geq  \int_{\Omega}g\mathrm{d}(\mu-\lambda)
\end{align*}
for all $\bN\in \bbLtwosym(\Omega), \, \mu\in H^{-2}(\Omega), \, \mu\geq 0$. 
Finally, testing the first equation in~\eqref{eq:plate:strong_mixed_obst} with $v\in M$ and restricting the test functions $(\bN,\mu)$ to $K$, we end up with the following mixed formulation: 
Find $(\bM,\lambda,u)\in K\times M$ such that
\begin{equation}\label{eq:plate:weak_mixed}
\left\{
\begin{array}{lcl}
a((\bM,\lambda), (\bN,\mu) - (\bM, \lambda)) - b((\bN,\mu) - (\bM, \lambda), u) & \geq & G((\bN,\mu) - (\bM, \lambda)) \\
-b((\bM, \lambda), v) & = &  F(v)
\end{array}
\right.
\end{equation}
for all $(\bN,\mu,v)\in K\times M$. In what follows we study the well-posedness of \eqref{eq:plate:weak_mixed}. 

\begin{proposition}[equivalence]
  Problems~\eqref{eq:plate:weakformObstacle} and~\eqref{eq:plate:weak_mixed} are equivalent in the following sense: 
  If $u\in \mathsf{V}$ denotes the unique solution of~\eqref{eq:plate:weakformObstacle}, then $(\Grad\nabla u,\Delta^2 u-f,u)\in K\times M$ solves~\eqref{eq:plate:weak_mixed}. 
  Reciprocally, if $(\bM,\lambda,u)\in K\times M$ solves~\eqref{eq:plate:weak_mixed}, then $u\in \mathsf{V}$ and solves~\eqref{eq:plate:weakformObstacle}. Moreover, $\bM = \Grad\nabla u$, and $\lambda=\Delta^{2} u-f$. 
  In particular, formulation~\eqref{eq:plate:weak_mixed} admits a unique solution.
\end{proposition}
\begin{proof}
  Let $u\in \mathsf{V}$ solve~\eqref{eq:plate:weakformObstacle}. By construction $(\bM,\lambda,u) := (\Grad\nabla u,\Delta^{2} u-f,u)\in K\times M$ is a solution of~\eqref{eq:plate:weak_mixed}.

  Conversely, let $(\bM,\lambda,u)\in K\times M$ be a solution of~\eqref{eq:plate:weak_mixed}. Choosing $(\bM \pm \widetilde{\bN},\lambda,0)$ with $\widetilde{\bN}\in \mathbb{H}(\dDiv;\Omega)$ as test function we get that $\ip{\bM}{\widetilde{\bN}}_\Omega - \ip{\dDiv\widetilde{\bN}}{u}_\Omega = 0$ for all $\widetilde{\bN}\in \mathbb{H}(\dDiv;\Omega)$. This identity yields that
  \begin{align*}
    u\in H_0^2(\Omega) \quad\text{with}\quad
    \bM = \Grad\nabla u. 
  \end{align*}
  Since the operator $\dDiv\colon \mathbb{L}_{\text{sym}}^2(\Omega) \to H^{-2}(\Omega)$ is surjective, given $\mu\in H^{-2}(\Omega)$ satisfying $\mu\geq 0$, we choose $\bN\in \mathbb{L}_{\text{sym}}^2(\Omega)$ such that $\dDiv\bN -\mu\in L^2(\Omega)$.  
  Then, using that $\bM = \Grad\nabla u$, we obtain the identity $\ip{\bM}{\bN-\bM} - \dual{\dDiv(\bN-\bM)}{u} = 0$. 
  Utilizing the latter in the first equation of \eqref{eq:plate:weak_mixed} gives $\int_{\Omega}(u-g)\mathrm{d}(\mu-\lambda)\geq 0$ for all $\mu\in H^{-2}(\Omega)$ with $\mu\geq 0$, which implies that $\int_{\Omega}(u-g)\mathrm{d}\lambda = 0$ and $u\geq g$. 
  Consequently, $u\in\mathsf{V}$. To prove that $u$ satisfies inequality \eqref{eq:plate:weakformObstacle}, we follow the arguments elaborated in the proof of Proposition \ref{prop:equivalent_obstacle}. This concludes the proof.
\end{proof}


\subsection{Finite element approximation}\label{sec:plate:fe_approx}

Based on the notation introduced in section \ref{sec:fem_and_interp}, we define the discrete spaces $M_{h}:= \mathcal{P}^{1}(\cT)$ and $V_{h}:= \mathbb{X}(\cT)\times M_{h}$, and the set
$$
K_{h}:=\{(\bN_{h},\mu_{h}) \in V_{h} : \mu_{h}|_{T} \geq 0 \quad  \forall T\,\in \cT\}.
$$
We propose the following finite element approximation for \eqref{eq:plate:weak_mixed}:
Find $(\bM_{h},\lambda_{h},u_{h})\in K_{h}\times M_{h}$ such that
\begin{equation}\label{eq:plate:weak_mixed_discrete}
\left\{
\begin{array}{lcl}
a((\bM_{h},\lambda_{h}), (\bN_{h},\mu_{h}) - (\bM_{h}, \lambda_{h})) - b((\bN_{h},\mu_{h}) - (\bM_{h}, \lambda_{h}), u_{h})  \! & \geq & \!  G((\bN_{h},\mu_{h}) - (\bM_{h}, \lambda_{h})) \\
-b((\bM_{h}, \lambda_{h}), v_{h})  \! & = & \!   F(v_{h})
\end{array}
\right.
\end{equation}
for all $(\bN_{h},\mu_{h},v_{h})\in K_{h} \times M_{h}$.

\begin{theorem}[well-posedness]
  Problem \eqref{eq:plate:weak_mixed_discrete} is well posed.
\end{theorem}
\begin{proof}
We verify all the assumptions of Theorem \ref{thm:wp_discrete_generalized}. 
We first note that~\eqref{eq:assumption_non_negativity_a} holds, since
\begin{equation*}
  a((\bN,\mu),(\bN,\mu))  = \|\bN\|_{\Omega}^2 \geq 0 \qquad \forall \,(\bN,\mu)\in V.
\end{equation*}
On the other hand, we note that $N=\{(\bN,\mu)\in V :  b((\bN,\mu),v) = 0 \quad \forall \,v\in M\} =\{(\bN,\mu)\in V :  \dDiv \bN - \mu = 0 \text{ a.e.  in }\Omega\}$. It thus follows that 
\begin{align*}
a((\bN,\mu),(\bN,\mu)) = \|\bN\|_{\Omega}^2 = \|\bN\|_{\Omega}^2 + \|\dDiv \bN - \mu\|_{\Omega}^{2} = \|(\bN,\mu)\|_{V}^{2}  \qquad \forall \,(\bN,\mu)\in N,
\end{align*}
which shows \eqref{eq:assumption_coercivity_a}. Assumption \eqref{eq:conforming_condition} follows from the fact that $N_{h}=\{(\bN_{h},\mu_{h})\in V_{h}: b((\bN_{h},\mu_{h}),v_{h}) =  0 \quad \forall \, v\in M_{h}\} = \{(\bN_{h},\mu_{h})\in V_{h}:  \dDiv \bN_{h} - \mu_{h} = 0 \text{ a.e.  in }\Omega\} \subset N$.
Finally, we consider the subspace $W_{h} = \mathbb{X}(\cT)\times \{0\} \subset V_{h}$; note that $W_{h} \subset K_{h}$. 
From \cite[Proof of Theorem 10]{2023arXiv230508693F}, it follows the discrete inf-sup condition
\begin{equation*}
  \sup_{\bN_{h}\in \mathbb{X}(\cT)}\frac{(\dDiv \bN_{h}, v_{h})_{\Omega}}{\|\bN_{h}\|_{\mathbb{H}(\dDiv;\Omega)}} \geq \hat{\beta} \|v_{h}\|_{M} \quad \forall \, v_{h}\in M_{h}, 
\end{equation*}
with $\hat{\beta} > 0$ independent of the discretization parameter $h$.
This shows that assumption \eqref{eq:discrete_inf_sup} is valid, and concludes the proof.
\end{proof}

We present some interesting properties of the mixed scheme~\eqref{eq:plate:weak_mixed_discrete}.

  \begin{proposition}[properties of the solution]\label{prop:plate:properties}
  Let $(\bM_h,\lambda_h,u_h)\in K_h\times M_h$ denote the solution to~\eqref{eq:plate:weak_mixed_discrete}. It possesses the following properties:
  \begin{itemize}
    \item[$\mathrm{(i)}$] $\dDiv\bM_h - \lambda_h = \Pi_h^1 f$, 
    \item[$\mathrm{(ii)}$] $\ip{\bM_h}{\bN_h}_\Omega - \ip{\dDiv\bN_h}{u_h}_\Omega = 0$ for all $\bN_h\in \mathbb{X}(\cT)$, and
    \item[$\mathrm{(iii)}$] $\int_{\Omega}(u_h-g)\mathrm{d}\lambda_{h} = (\lambda_{h}, u_{h} - g)_{\Omega}  = 0$.
  \end{itemize}
\end{proposition}
\begin{proof}
  Property $\mathrm{(i)}$ follows from the second equation in~\eqref{eq:plate:weak_mixed_discrete} and the fact that $\dDiv(\mathbb{X}(\cT)) = \mathcal{P}^{1}(\cT) = M_h$. 
  To prove $\mathrm{(ii)}$, we take $(\bN_h,\mu_h,v_h)=(\bM_{h}\pm\widetilde{\bN}_h,\lambda_h,0)$ in~\eqref{eq:plate:weak_mixed_discrete} with $\widetilde{\bN}_h\in\mathbb{X}(\cT)$ arbitrary. 
This gives
  \begin{align*}
    \pm\ip{\bM_h}{\widetilde{\bN}_h}_\Omega \mp \ip{\dDiv\widetilde{\bN}_h}{u_h}_{\Omega} \geq 0 \quad\forall \widetilde{\bN}_h\in\mathbb{X}(\cT), 
  \end{align*}
  from which we infer $\mathrm{(ii)}$.
  Using $\mathrm{(i)}$, $\mathrm{(ii)}$, and that $\mu_{h},\lambda_{h}\in L^{2}(\Omega)$, we see that~\eqref{eq:plate:weak_mixed_discrete} simplifies to
  \begin{align*}
  \int_{\Omega}(u_h-g)\mathrm{d}(\mu_h-\lambda_h) = \ip{\mu_{h} - \lambda_{h}}{u_{h} - g}_{\Omega} \geq 0 \quad\forall \mu_h\in M_h, \, \mu_h\geq 0.
  \end{align*}
  In the last inequality, choose $\mu_h = 2\lambda_h$ and then $\mu_h = 0$ resulting in $\pm\ip{\lambda_{h}}{u_h-g}_{\Omega} \geq 0$ or, equivalently, $\ip{\lambda_{h}}{u_h-g}_{\Omega} = 0$, which proves $\mathrm{(iii)}$. 
  \end{proof}
  
  \begin{theorem}[a priori estimate]\label{thm:plate:bestapprox}
Let $(\bM,\lambda,u)\in K\times M$ be the unique solution to \eqref{eq:plate:weak_mixed} and let $(\bM_{h},\lambda_{h},u_{h})\in K_{h}\times M_{h}$ be its finite element approximation obtained as the unique solution to \eqref{eq:plate:weak_mixed_discrete}. 
Then, we have that:
\begin{align*}
\|(\bM,\lambda,u) - &(\bM_h,\lambda_h,u_h)\|_{V\times M}^{2}
\lesssim \min_{v_{h}\in M_{h}}\|u - v_{h}\|_{\Omega}^{2}  \\
&+ \min_{(\bN_{h},\mu_{h})\in K_{h}}\left\{\|\bM - \bN_{h}\|_{\Omega}^{2} + \|\dDiv(\bM - \bN_{h}) - (\lambda - \mu_{h})\|_{\Omega}^{2} + \int_{\Omega}(u - g)\mathrm{d}(\mu_{h} - \lambda)\right\}.
\end{align*}
\end{theorem}
\begin{proof}
Let us estimate $\|u - u_{h}\|_{\Omega}$. 
An application of the triangle inequality and the fact that $\|u - \Pi_{h}^{1}u\|_{\Omega}^{2} = \min_{v_{h}\in M_{h}}\|u - v_{h}\|_{\Omega}^{2}$ give
\begin{equation}\label{eq:plate:estimate_u-uh1}
\|u - u_{h}\|_{\Omega}^{2} \leq 2 \min_{v_{h}\in M_{h}}\|u - v_{h}\|_{\Omega}^{2} + 2\| \Pi_{h}^{1}u - u_{h}\|_{\Omega}^{2}.
\end{equation}   
We now concentrate on $\|\Pi_{h}^{1}u - u_{h}\|_{\Omega}^{2}$. 
Let us introduce the auxiliary variable $\mathsf{v}\in H_0^{1}(\Omega)$, defined as the unique solution to $\Delta \mathsf{v} = \Pi_{h}^{1} u - u_{h}$ in $\Omega$ and $\mathsf{v} = 0$ on $\partial\Omega$. 
By elliptic regularity \cite{MR0775683}, there exists $r\in (1/2,1]$ (depending only on $\Omega$) such that $\mathsf{v}\in H^{1+r}(\Omega)$. 
Hence, the term $\mathbf{I}\mathsf{v} \in [H^{1+r}(\Omega)]^{2 \times 2}\cap \mathbb{H}(\dDiv;\Omega)$, where $\mathbf{I}$ denotes the identity matrix, and thus $\bN_{h} :=\Pi_{h}^{\dDiv}(\mathbf{I}\mathsf{v})$ is well defined with $\dDiv \bN_{h} = \Pi_{h}^{1}u - u_{h}$. 
Using the properties of the projection operator $\Pi_{h}^{\dDiv}$ provided in section \ref{sec:fem_and_interp}, the identity $\bM = \Grad\nabla u$, and property $(\mathrm{iii})$ from Proposition \ref{prop:plate:properties}, we obtain
\begin{align*}
\|\Pi_{h}^{1}u - u_{h}\|_{\Omega}^{2} 
= & \, (u - u_{h},\Pi_{h}^{1}( \dDiv \bN_{h}))_{\Omega}
= (u - u_{h},\dDiv\bN_{h})_{\Omega} 
= (\bM - \bM_{h},\bN_{h})_{\Omega} \\
&\leq \|\bM - \bM_{h}\|_{\Omega}\|\bN_{h}\|_{\Omega} 
\leq \|\bM - \bM_{h}\|_{\Omega}\|\bN_{h}\|_{\mathbb{H}(\dDiv;\Omega)} 
\lesssim \|\bM - \bM_{h}\|_{\Omega}\|\Pi_{h}^{1}u - u_{h}\|_{\Omega}. 
\end{align*}
Consequently, $\|\Pi_{h}^{1}u - u_{h}\|_{\Omega} \lesssim  \|\bM - \bM_{h}\|_{\Omega}$. This, in combination with \eqref{eq:plate:estimate_u-uh1}, results in
\begin{equation*}
\|u - u_{h}\|_{\Omega}^{2} 
\lesssim 
\min_{v_{h}\in M_{h}}\|u - v_{h}\|_{\Omega}^{2} + \|\bM - \bM_{h}\|_{\Omega}^{2}.
\end{equation*}

The rest of the proof follows a suitable adaptation of steps 2 and 3 in the proof of Theorem  \ref{thm:bestapprox}. For brevity, we skip those details.
\end{proof}

The derivation of convergence rates for the solution of the plate obstacle problem is more challenging than the case associated to the elliptic membrane obstacle problem. 
The main difficulty is given by the lack of regularity of the solution $u$. 
To be precise, it is known that under suitable assumptions on $g$ and $\Omega$, the solution to \eqref{eq:plate:weakformObstacle} belongs to $H^{3}_{\text{loc}}(\Omega)\cap C^{2}(\Omega)$; see, e.g., \cite{MR0330754,MR0324208,MR0529478}. 
Moreover, if $\Omega$ is convex, it follows that $u\in H^{3}(\Omega)$ (see \cite{MR2904578} and references therein). 
In general, this regularity cannot be improved up to $H^{4}(\Omega)$, even for the case when data is smooth \cite{MR0529478}. 
The examples of exact solutions from \cite[section 4.2]{MR4029734} and \cite[section 5]{MR3061064} suggest that the regularity of $u$ is not better than $C^{2,\frac{1}{2}}(\Omega)$ or $H^{\frac{7}{2}-\epsilon}(\Omega)$ with $\epsilon > 0$. 
This motivates the development and analysis of a posteriori error estimates for problem \eqref{eq:plate:weakformObstacle}, which will be done in section \ref{sec:apost:plate}.


\section{A posteriori estimation based on postprocessed solution}\label{sec:apost}
In this section we provide a posteriori error analysis for the mixed FEMs from sections~\ref{sec:obst_problem} and~\ref{sec:plate}.
First, we present some auxiliary results in section~\ref{sec:apost:aux}. Then, in section~\ref{sec:apost:elastic}, we study a posteriori error estimators for problem~\eqref{eq:weak_mixed_discrete}. 
Finally, in section~\ref{sec:apost:plate}, we investigate a posteriori error estimates for problem~\eqref{eq:plate:weak_mixed_discrete}.

\subsection{Auxiliary results}\label{sec:apost:aux}
The proof of the next auxiliary result follows as in~\cite[Theorem~1, Lemma~3]{MR4050087}.
\begin{lemma}[norm equivalences]\label{lem:normequiv}
  The following norm equivalences hold true for all $(v,\btau,\mu)\in H_0^1(\Omega)\times [L^2(\Omega)]^n\times H^{-1}(\Omega)$ and all $(w,\bN,\chi)\in H_0^2(\Omega)\times [L^2(\Omega)]^{n\times n}\times H^{-2}(\Omega)$: 
  \begin{subequations}\label{eq:normequiv}
  \begin{align}
    \|(v,\btau,\mu)\|_{H_0^1(\Omega)\times [L^2(\Omega)]^n\times H^{-1}(\Omega)}^2 &\eqsim \|\div\btau+\mu\|_{-1}^2 
    + \|\btau-\nabla v\|_\Omega^2 + \dual{\mu}v, \label{eq:normequiv:a} \\
    \|(w,\bN,\chi)\|_{H_0^2(\Omega)\times [L^2(\Omega)]^{n\times n}\times H^{-2}(\Omega)}^2 &\eqsim \|\dDiv\bN-\chi\|_{-2}^2 
    + \|\bN-\Grad\nabla w\|_\Omega^2 + \dual{\chi}w. \label{eq:normequiv:b}
  \end{align}
\end{subequations}
\end{lemma}
\begin{proof}
  The lower bound in~\eqref{eq:normequiv:a} follows by the triangle inequality and boundedness of the divergence operator.
  To see the upper bound, let $(v,\btau,\mu) \in H_0^1(\Omega)\times [L^2(\Omega)]^n\times H^{-1}(\Omega)$ be given. Integration by parts and two applications of Young's inequality with parameters $\delta_1>0$, $\delta_2>0$ give the following estimates: 
  \begin{align*}
  \|\div\btau+\mu\|_{-1}^2 + \|\btau-\nabla v\|_\Omega^2 &\,+ \dual{\mu}v 
     = \|\div\btau+\mu\|_{-1}^2 + \|\btau\|_\Omega^2 + \|\nabla v\|_\Omega^2 -\ip{\btau}{\nabla v}_\Omega + \dual{\div\btau+\mu}{v} 
    \\
    & \geq \|\div\btau+\mu\|_{-1}^2 + \|\btau\|_\Omega^2 + \|\nabla v\|_\Omega^2 - \|\btau\|_\Omega\|\nabla v\|_\Omega - \|\div\btau+\mu\|_{-1}\|\nabla v\|_\Omega 
    \\
    &\geq(1-\tfrac{\delta_1^{-1}}{2})\|\div\btau+\mu\|_{-1}^2 + (1-\tfrac{\delta_2^{-1}}{2})\|\btau\|_\Omega^2 + (1-\tfrac{\delta_1}{2}-\tfrac{\delta_2}{2})\|\nabla v\|_\Omega^2.
  \end{align*}
  Choosing $\delta_1 = 2/3$ and $\delta_2=2/3$, we find that 
  \begin{equation*}
    \|\div\btau+\mu\|_{-1}^2 + \|\btau-\nabla v\|_\Omega^2 + \dual{\mu}v  \geq \frac14 \|\div\btau+\mu\|_{-1}^2 + \frac14 \|\btau\|_\Omega^2 + \frac13 \|\nabla v\|_\Omega^2.
  \end{equation*}
  Then, 
  \begin{align*} 
    \|\mu\|_{-1}^2 + \|\btau\|_\Omega^2 + \|\nabla v\|_\Omega^2  
    & \lesssim \|\div\btau + \mu\|_{-1}^2+ \|\div\btau\|_{-1}^2 + \|\btau\|_\Omega^2 + \|\nabla v\|_\Omega^2\\
    & \lesssim \|\div\btau + \mu\|_{-1}^2+ \|\btau\|_\Omega^2 + \|\nabla v\|_\Omega^2
    \lesssim \|\div\btau+\mu\|_{-1}^2 + \|\btau-\nabla v\|_\Omega^2 + \dual{\mu}v
  \end{align*}
  finishes the proof of the upper bound in~\eqref{eq:normequiv:a}. 

  The proof of~\eqref{eq:normequiv:b} follows similar arguments and is omitted. 
\end{proof}

One major challenge to derive a posteriori estimators for mixed FEMs for obstacle problems is to deal with the duality pairing between $H_0^1(\Omega)$ and $H^{-1}(\Omega)$ (resp. between $H_0^2(\Omega)$ and $H^{-2}(\Omega)$). 
The reason is that one obtains only an $L^2(\Omega)$-approximation of the primal variable $u$. 
To overcome this difficulty, we consider postprocessed solutions together with Lemma~\ref{lem:normequiv}. 

To present the following result, which states a preliminary a posteriori error estimate for the membrane obstacle problem, we introduce the positive part of a function $v: \Omega \to \mathbb{R}$
\begin{equation*}
v_{+} :=\max\{v,0\}.
\end{equation*}
\begin{lemma}
\label{lem:apost:membrane}
  Let $(\bsigma,\lambda,u)=(\nabla u,-\Delta u-f,u)$ denote the unique solution of~\eqref{eq:weak_mixed}.
  Estimate
    \begin{equation*}
   \|(u - v,\bsigma-\btau,\lambda-\mu)\|_{H_0^1(\Omega)\times [L^2(\Omega)]^n\times H^{-1}(\Omega)}^2 
\lesssim \|\btau-\nabla v\|_\Omega^2 + \dual{\mu}{(v-g)_+} + \|\nabla(g-v)_+\|_\Omega^2
    + \|h_{\cT} (1-\Pi_h^0)f\|_\Omega^2
  \end{equation*}
  holds for all $v\in H_0^1(\Omega)$, $(\btau,\mu)\in K_{1}$ with $\div\btau+\mu = -\Pi_h^0f$. 
\end{lemma}
\begin{proof}
  By~\eqref{eq:normequiv:a} and identities $\div(\bsigma-\btau) + \lambda-\mu = \Pi_h^0f - f$ and $\bsigma=\nabla u$, we get that
  \begin{align*}
  \|(u - v,\bsigma-\btau,\lambda-\mu)\|_{H_0^1(\Omega)\times [L^2(\Omega)]^n\times H^{-1}(\Omega)}^2  
  & \lesssim \|\div(\bsigma-\btau)+\lambda-\mu\|_{-1}^2 + \|\btau-\nabla v\|_\Omega^2 + \dual{\lambda-\mu}{u-v}
    \\
    &\lesssim \|h_{\cT}(1-\Pi_h^0)f\|_\Omega^2 + \|\btau-\nabla v\|_\Omega^2 + \dual{\lambda-\mu}{u-v}.
  \end{align*}
  It remains to estimate the duality term $\dual{\lambda-\mu}{u-v}$. This is done with arguments found in~\cite[Proof of Proposition~4.2]{MR1860720} or in the proof of~\cite[Theorem~16]{MR4050087}. For completeness, we repeat the main steps here.
    First, note that for $\delta>0$
  \begin{align*}
    \dual{\lambda}{u-v} &= \int_{\Omega}(u-g)\mathrm{d}\lambda + \int_{\Omega}(g-\max\{g,v\})\mathrm{d}\lambda +\dual{\lambda}{\max\{g,v\}-v} \\
    &\leq \dual{\lambda}{\max\{g,v\}-v} = \dual{\lambda-\mu}{(g-v)_+} + \dual{\mu}{(g-v)_+} \\
    &\leq \frac{\delta}2 \|\lambda-\mu\|_{-1}^2 + \frac{\delta^{-1}}2 \|\nabla(g-v)_+\|_\Omega^2 + \dual{\mu}{(g-v)_+}.
  \end{align*}
  Second, due to $u\geq g$ a.e. in $\Omega$ and $\mu\geq 0$,  we have
  \begin{align*}
    -\dual{\mu}{u-v} 
    = \int_{\Omega}(g - u)\mathrm{d}\mu + \int_{\Omega}(v-g)\mathrm{d}\mu 
    \leq \int_{\Omega}(v-g)\mathrm{d}\mu.
  \end{align*}
  Combining all previous estimates we arrive at
  \begin{align*}
    &  \|(u - v,\bsigma-\btau,\lambda-\mu)\|_{H_0^1(\Omega)\times [L^2(\Omega)]^n\times H^{-1}(\Omega)}^2  \\
    &\qquad \lesssim \|h_{\cT}(1-\Pi_h^0)f\|_\Omega^2 + \|\btau-\nabla v\|_\Omega^2 + \frac{\delta}2 \|\lambda-\mu\|_{-1}^2 + \frac{\delta^{-1}}2 \|\nabla(g-v)_+\|_\Omega^2 + \dual{\mu}{(g-v)_+} + \int_{\Omega}(v-g)\mathrm{d}\mu \\
    &\qquad = \|h_{\cT}(1-\Pi_h^0)f\|_\Omega^2 + \|\btau-\nabla v\|_\Omega^2 + \frac{\delta}2 \|\lambda-\mu\|_{-1}^2 + \frac{\delta^{-1}}2 \|\nabla(g-v)_+\|_\Omega^2 + \dual{\mu}{(v-g)_+}.
  \end{align*}
  Subtracting $\frac{\delta}2 \|\lambda-\mu\|_{-1}^2$ for sufficient small $\delta>0$ finishes the proof.
\end{proof}
The estimate proved in Lemma \ref{lem:apost:membrane} can be used to define a reliable error estimator by setting $\mu = \lambda_h$, $\btau=\bsigma_h$. Note that $v$ needs to be in $H_0^1(\Omega)$ so that the choice $v=u_h\in M_h$ is not valid. However, we can postprocess and smooth the discrete solution $u_h$ to obtain an efficiently computable $v$ in the previous result. We present details in section~\ref{sec:apost:elastic} below. 

The proof of Lemma~\ref{lem:apost:membrane} can not be directly adapted for solutions of the plate obstacle problem. 
The reason is that in the proof we use that $\max\{v,w\}\in H_0^1(\Omega)$ for $v,w\in H_0^1(\Omega)$, which is in general not true if we replace $H_0^1(\Omega)$ by $H_0^2(\Omega)$. 
Nevertheless, we have the following result which combines Lemma~\ref{lem:normequiv} with arguments used to prove~\cite[Lemma~4.1]{MR3595879}.
We recall that the Lagrange multiplier $\lambda\in H^{-2}(\Omega)$ with $\lambda \geq 0$ defines a positive Radon measure. 
\begin{lemma}\label{lem:apost:plate}
  Let $(\bM,\lambda,u)=(\Grad\nabla u,\Delta^2 u-f,u)$ denote the unique solution of~\eqref{eq:plate:weak_mixed}.
  Estimate
    \begin{align*}
      \|(u - v,\bM-\bN,\lambda-\mu)\|_{H_0^2(\Omega)\times [L^2(\Omega)]^{n\times n}\times H^{-2}(\Omega)}^2 
    \lesssim \, & \|\bN-\Grad\nabla v\|_\Omega^2 + \|h_{\cT}^2 (1-\Pi_h^1)f\|_\Omega^2 \\
    & + \ip{\mu}{v-g}_\Omega + \int_\Omega(g-v)_+\,\mathrm{d}\lambda 
  \end{align*}
  holds for all $v\in H_0^2(\Omega)$, $(\bM,\mu)\in K_{2}$ with $\mu\in L^2(\Omega)$ and $\dDiv\bN-\mu =\Pi_h^1f$.
\end{lemma}
\begin{proof}
  Employing~\eqref{eq:normequiv:b}, $\dDiv\bM-\lambda = f$, and $\bM = \Grad\nabla u$, we get that
  \begin{align*}
    \|(u-v,\bM-\bN,\lambda-\mu)\|_{H_0^2(\Omega)\times [L^2(\Omega)]^n\times H^{-2}(\Omega)}^2 
    &\eqsim \|(1-\Pi_h^1)f\|_{-2}^2 
    + \|\bN-\Grad\nabla v\|_\Omega^2 + \dual{\lambda-\mu}{u-v}
    \\
    &\lesssim \|h_{\cT}^2(1-\Pi_h^1)f\|_\Omega^2 + \|\bN -\Grad\nabla v\|_\Omega^2 + \dual{\lambda-\mu}{u-v}.
  \end{align*}
  It thus remains to tackle the duality term on the right-hand side. Following a similar argumentation as in the proof of~\cite[Lemma~4.1]{MR3595879} we find
  \begin{align*}
    \dual{\lambda}{u-v} = \int_\Omega(g-v)\,\mathrm{d}\lambda \leq \int_\Omega(g-v)_+\,\mathrm{d}\lambda.
  \end{align*}
  Furthermore, using $\mu\in L^2(\Omega)$, $\mu\geq 0$, $u\geq g$, it follows that
  \begin{align*}
    -\dual{\mu}{u-v} = \ip{\mu}{v-u}_\Omega \leq \ip{\mu}{v-g}_\Omega,
  \end{align*}
  which finishes the proof.
\end{proof}


\subsection{A posteriori error analysis for the membrane obstacle problem}\label{sec:apost:elastic}
Our starting point for defining an error estimator is Lemma~\ref{lem:apost:membrane}. 
Several possibilities exist to define a suitable postprocessed solution $v\in H_0^1(\Omega)$ of $(\bsigma_h,u_h)$. 
Here, we shall apply a simple quasi-interpolator that only uses solution component $u_h$.
Let $\eta_z\in \mathcal{P}^{1}(\mathcal{T})$ denote the nodal basis function with $\eta_z(z') = \delta_{z,z'}$ for all interior vertices $z,z'\in\mathcal{V}_0$. 
Furthermore, let $\omega_z\subset \mathcal{T}$ denote the set (patch) of all elements which share vertex $z$.  
We follow~\cite[section~3]{FuehrerHm1Loads} and define the weighted Cl\'ement quasi-interpolator $J_h\colon L^2(\Omega)\to \mathcal{P}^1(\mathcal{T})\cap H_0^1(\Omega)$ by 
\begin{align*}
  J_h v = \sum_{z\in \mathcal{V}_0} \ip{\phi_z}{v}_\Omega \eta_z,
\end{align*}
where the weight functions are given by
\begin{align*}
  \phi_z|_T = \begin{cases}
    \frac{\alpha_{z,T}}{|T|} & T\in\omega_z, \\
    0 & \text{else},
  \end{cases}
\end{align*}
and the coefficients satisfy
\begin{align*}
  \sum_{T\in\omega_z} \alpha_{z,T} = 1, \quad \sum_{T\in\omega_z} \alpha_{z,T} s_T = z, \quad \alpha_{z,T}\geq 0 \quad (T\in\omega_z).
\end{align*}
Here, $s_T$ denotes the center of mass of an element $T\in\mathcal{T}$.
Note that the coefficients $\alpha_{z,T}$ are not necessarily unique (see, e.g.,~\cite[Example~10]{FuehrerHm1Loads}).
In~\cite[Theorem~11]{FuehrerHm1Loads} it was shown that $J_h$ has second-order approximation properties while the Cl\'ement operator with $0$-order moments, namely, setting $\alpha_{z,T} = |T|/|\omega_z|$ in the above definition, only has first-order approximation properties.

Given  $T\in\mathcal{T}$, we define the local error indicators 
\begin{align*}
  \rho_r^2(T) &:= \|\bsigma_h-\nabla J_h u_h\|_T^2,
  \qquad\rho_p^2(T) := \|\nabla(g-J_hu_h)_+\|_T^2, 
  \qquad\rho_c^2(T) := \ip{\lambda_h}{(g - J_hu_h)_+}_T.
\end{align*}
Here, $(\bsigma_h,\lambda_h,u_h)$ denotes the solution of~\eqref{eq:weak_mixed_discrete}.
Further, set 
\begin{align*}
  \mathrm{osc}(T)^2 = h_T^2\|(1-\Pi_h^0)f\|_T^2.
\end{align*}
Note that $\rho_r$ measures a residual, while $\rho_p$ measures the penetration of the contact condition, and  $\rho_c$ measures a violation of the obstacle condition for the postprocessed solution $J_h u_h$.
The next theorem is the main result of this section, it directly stems from Lemma~\ref{lem:apost:membrane} with $v = J_h u_h$. 
\begin{theorem}[reliability]\label{thm:apost:membrane:rel}
  Let $(\bsigma,\lambda,u)$ denote the solution of~\eqref{eq:weak_mixed} and let $(\bsigma_h,\lambda_h,u_h)$ denote the solution of~\eqref{eq:weak_mixed_discrete}. Then, 
  \begin{align*}
    \|\nabla(u-J_h u_{h})\|_\Omega^2 + \|\bsigma-\bsigma_h\|_\Omega^2 + \|\lambda-\lambda_h\|_{-1}^2 
    \lesssim 
    \sum_{T\in\mathcal{T}} 
        \left(\rho_r(T)^2 + \rho_p(T)^2 + \rho_c(T)^2 + \mathrm{osc}(T)^2\right) 
    =:\rho^2.
  \end{align*}
\end{theorem}

\begin{remark}[estimation of $\|u - u_{h}\|_{\Omega}$]
The error $u - u_{h}$ can be controlled by using the triangle inequality and Friedrich's inequality:
\begin{equation*}
\|u - u_{h}\|_{\Omega} \leq \|u - J_h u_{h}\|_{\Omega} + \|u_{h} - J_h u_{h}\|_{\Omega}
\lesssim 
\|\nabla(u-J_h u_{h})\|_\Omega +  \|u_{h} - J_h u_{h}\|_{\Omega}.
\end{equation*}
The term $\|\nabla(u-J_h u_{h})\|_\Omega$ is bounded as in Theorem \ref{thm:apost:membrane:rel}, while $\|u_{h} - J_h u_{h}\|_{\Omega}$ is a computable term that can be used as an error estimator.
\end{remark}

The triangle inequality proves the following efficiency-type estimate.
\begin{theorem}[local estimate for $\rho_{r}(T)$]\label{thm:apost:membrane:eff}
  Under the framework of Theorem~\ref{thm:apost:membrane:rel} we have, for all $T\in\mathcal{T}$, that
  \begin{align*}
    \rho_r(T)^2 \lesssim \|\bsigma-\bsigma_h\|_T^2 + \|\nabla (u - J_h u_h)\|_T^2. 
  \end{align*}
\end{theorem}


\subsection{A posteriori error analysis for the plate obstacle problem}\label{sec:apost:plate}
For the plate obstacle problem we follow similar ideas as in section~\ref{sec:apost:elastic}. 
Let $U_h^\mathrm{HCT}\subset H_0^2(\Omega)$ denote the Hsieh--Clough--Tocher finite element space~\cite[chapter 3.7]{BrennerScott} and let $E_h\colon \mathcal{P}^3(\mathcal{T})\to U_h^\mathrm{HCT}$ denote the operator defined in~\cite[section~4]{MR2777246}. 
We note that the authors of~\cite{MR2777246} define the operator for $\mathcal{P}^2(\mathcal{T})$ but the same definition can be used in our situation, i.e., for $\mathcal{P}^3(\mathcal{T})$. 
Let $B_h$ denote a first-order moment preserving operator given by
\begin{align*}
  B_h v|_T = \sum_{z\in\mathcal{V}_T} \ip{v}{\nu_{T,z}}_T\eta_T^2\eta_z|_T,
\end{align*}
where $\eta_T = \prod_{z\in \mathcal{V}_T} \eta_z|_T$ (and extended to $0$ outside of $T$) and $\nu_{T,z}$ ($z\in \mathcal{V}_T$) form a dual basis of $\mathcal{P}^1(T)$
with $\ip{\nu_{T,z}}{\eta_T^2\eta_{z'}}_T = \delta_{z,z'}$.
For an arbitrary $p=\sum_{z'\in\mathcal{V}_T}\alpha_{z'}\eta_{z'}|_T \in \mathcal{P}^1(T)$ with $\alpha_{z'}\in \mathbb{R}$, we have
\begin{align*}
  \ip{B_h v}{p}_T = \sum_{z\in\mathcal{V}_T} \ip{v}{\nu_{T,z}}_T \ip{\eta_T^2\eta_z}{p}_T 
  = \sum_{z\in\mathcal{V}_T}\sum_{z'\in\mathcal{V}_T} \alpha_{z'} \ip{v}{\nu_{T,z}}_T \ip{\eta_T^2\eta_z}{\nu_{T,z'}}_T = \ip{v}{p}_T.
\end{align*}
We thus conclude that $\Pi_h^1B_h v= \Pi_h^1 v$.
It is straightforward to verify that $\|B_hv\|_T\lesssim \|v\|_T$ for all $T\in\mathcal{T}$. Next, define
\begin{align*}
  J_h^\mathrm{HCT}v = E_h v + B_h(1-E_h)v.
\end{align*}
Clearly, $J_h^\mathrm{HCT}v\in H_0^2(\Omega)$. 
In the next result we collect some important properties of this operator. 
We use the common notation $[\cdot]$ for jumps across element interfaces, i.e., $[v]|_E$ denotes the jump across the edge $E\in\mathcal{E}$. 
\begin{lemma}[properties of $J_h^\mathrm{HCT}$]
\label{lem:estHCT}
  Operator $J_h^\mathrm{HCT}$ satisfies $\Pi_h^1J_h^\mathrm{HCT} = \Pi_h^1$ and the estimates
  \begin{align}
    \sum_{T\in\mathcal{T}}h_T^{-4} \|v-J_h^\mathrm{HCT}v\|_T^2 + \|\Grad\nabla(v-J_h^\mathrm{HCT}v)\|_T^2
    &\lesssim \sum_{T\in\mathcal{T}} h_T^{-3}\|[v]\|_{\partial T}^2 + h_T^{-1}\|[\partial_\normal v]\|_{\partial T}^2, \\
    h_T^{-2}\|v-J_h^\mathrm{HCT}v\|_{L^\infty(T)}^2 
    &\lesssim \sum_{E\in\mathcal{E}, \mathcal{V}_E\cap\mathcal{V}_T\neq \emptyset} h_T^{-3}\|[v]\|_{E}^2 + h_T^{-1}\|[\partial_\normal v]\|_{E}^2
  \end{align}
  for all $v\in \mathcal{P}^3(\mathcal{T})$ and $T\in\mathcal{T}$.
\end{lemma}
\begin{proof}
  First, $\Pi_h^1J_h^\mathrm{HCT} = \Pi_h^1 E_h + \Pi_h^1B_h(1-E_h) = \Pi_h^1$ by the properties of $B_h$. 
  Second,
  \begin{align*}
    \|v-J_h^\mathrm{HCT}v\|_T = \|(1-B_h)(1-E_h)v\|_T \lesssim \|(1-E_h)v\|_T.
  \end{align*}
  Then, following the very same arguments as in the proof of~\cite[Lemma~4.1]{MR2777246} shows that
  \begin{align*}
    h_T^{-4}\|v-E_hv\|_T^2 \lesssim \sum_{E\in\mathcal{E},\,\mathcal{V_T}\cap\mathcal{V}_E\neq\emptyset} h_T^{-3}\|[v]\|_{E}^2 + h_T^{-1}\|[\partial_\normal v]\|_{E}^2.
  \end{align*}
  Finally, the remaining estimates follow by inverse estimates.
\end{proof}

Let $(\bM_h,\lambda_h,u_h)$ denote the solution of~\eqref{eq:plate:weak_mixed_discrete}. 
We define the locally postprocessed solution $u_h^\star \in \mathcal{P}^3(\mathcal{T})$ on each $T\in\mathcal{T}$ by
\begin{align*}
  \ip{\Grad\nabla u_h^\star}{\Grad\nabla v}_T &= \ip{\bM_{h}}{\Grad\nabla v}_T \quad\forall v\in \mathcal{P}^3(T), \\
  \Pi_h^1u_h^\star|_T &= u_h|_T.
\end{align*}

Define the indicators resp. data oscillations %
\begin{align*}
  \xi_r(T)^2 &= \|\bM_{h} -\Grad\nabla u_h^\star\|_T^2 + h_T^{-3}\|[u_h^\star]\|_{\partial T}^2 + h_T^{-1}\|[\partial_\normal u_h^\star]\|_{\partial T}^2, \\
  \mathrm{osc}(T)^2 &= h_T^4\|(1-\Pi_h^1)f\|_T^2,
  \\ 
  \xi_r^2 &= \sum_{T\in\mathcal{T}} \xi_r(T)^2, \quad \mathrm{osc}^2 = \sum_{T\in\mathcal{T}} \mathrm{osc}(T)^2,
\end{align*}
and
\begin{align*}
  \xi_\infty &= \|(g-u_h^\star)_+\|_{L^\infty(\Omega)} + \max_{T\in\mathcal{T}} \sum_{E\in\mathcal{E},\, \mathcal{V}_E\cap\mathcal{V}_T\neq \emptyset}
  h_T^{-3/2}\|[u_h^\star]\|_{E} + h_T^{-1/2}\|[\partial_\normal u_h^\star]\|_{E}.
\end{align*}

\begin{theorem}[reliability]
\label{thm:apost:plate:rel}
  Let $(\bM,\lambda,u)$ denote the solution of~\eqref{eq:plate:weak_mixed} and let $(\bM_h,\lambda_h,u_h)$ denote the solution of~\eqref{eq:plate:weak_mixed_discrete}. Then, 
  \begin{align*}
    \|\Grad\nabla(u-J_h^\mathrm{HCT} u_h^\star)\|_\Omega^2 + \|\bM-\bM_h\|_\Omega^2 + \|\lambda-\lambda_h\|_{-2}^2 
    &\lesssim \xi_r^2 + \mathrm{osc}^2 + \lambda(\Omega)\xi_\infty.
  \end{align*}
\end{theorem}
\begin{proof}
  We employ Lemma~\ref{lem:apost:plate} with $\bN = \bM_h$, $v = J_h^\mathrm{HCT}u_h^\star$ and $\mu=\lambda_h$ which yields
  \begin{align*}
    \|\Grad\nabla(u & -J_h^\mathrm{HCT} u_h^\star)\|_\Omega^2 + \|\bM-\bM_h\|_\Omega^2 + \|\lambda-\lambda_h\|_{-2}^2 \\
    &\lesssim 
    \|\bM_h-\Grad\nabla J_h^\mathrm{HCT}u_h^\star\|_\Omega^2 
    + \|h_{\mathcal{T}}^2(1-\Pi_h^1)f\|_\Omega^2 
    + \ip{\lambda_h}{J_h^\mathrm{HCT}u_h^\star-g}_\Omega + \int_\Omega(g-J_h^\mathrm{HCT}u_h^\star)_{+}\,\mathrm{d}\lambda. \nonumber
  \end{align*}
  With Lemma~\ref{lem:estHCT} we further estimate
  \begin{align}\label{eq:M_h-u_star}
    \|\bM_h-\Grad\nabla J_h^\mathrm{HCT}u_h^\star\|_\Omega^2 &\lesssim \sum_{T\in\cT} \Big(\|\bM_{h} -\Grad\nabla u_h^\star\|_T^2 +  h_T^{-3}\|[u_h^\star]\|_{\partial T}^2 + h_T^{-1}\|[\partial_\normal u_h^\star]\|_{\partial T}^2\Big).
  \end{align}
  The use of $\Pi_h^1J_h^\mathrm{HCT}=\Pi_h^1$, $\lambda_h\in \mathcal{P}^1(\mathcal{T})$, $\Pi_h^1u_h^\star = u_h$, and $\ip{\lambda_h}{u_h-g}_\Omega = 0$ leads to 
  \begin{align}\label{eq:identity_jhHCT}
    \ip{\lambda_h}{J_h^\mathrm{HCT}u_h^\star-g}_\Omega = \ip{\lambda_h}{\Pi_h^1u_h^\star-g}_\Omega = \ip{\lambda_h}{u_h-g}_\Omega = 0.
  \end{align}
  It remains to estimate $\int_\Omega(g-J_h^\mathrm{HCT}u_h^\star)_+\,\mathrm{d}\lambda$. We do this by following the same steps from the proof of~\cite[Theorem~4.2]{MR3595879}.
  We have that, with $\lambda(\Omega)$ denoting the $\lambda$-measure of $\Omega$,
  \begin{align*}
\int_\Omega(g-J_h^\mathrm{HCT}u_h^\star)_+\,\mathrm{d}\lambda \leq \lambda(\Omega)(\|(g-u_h^\star)_+\|_{L^\infty(\Omega)} + \|u_h^\star-J_h^\mathrm{HCT}u_h^\star\|_{L^\infty(\Omega)}).
  \end{align*}
  Then, with Lemma~\ref{lem:estHCT} we conclude that
  \begin{align*}
    \|u_h^\star-J_h^\mathrm{HCT}u_h^\star\|_{L^\infty(\Omega)} \lesssim \max_{T\in\mathcal{T}} \sum_{E\in\mathcal{E},\, \mathcal{V}_E\cap\mathcal{V}_T\neq \emptyset}
    h_T^{-3/2}\|[u_h^\star]\|_{E} + h_T^{-1/2}\|[\partial_\normal u_h^\star]\|_{E}.
  \end{align*}
  This finishes the proof.
\end{proof}

The previous results can be seen as an analogous version of the reliability estimate presented in \cite[Thereom 4.2]{MR3595879}. 
As in that case, since the term $\lambda(\Omega)$ is not known, the obtained error bound is not a genuine a posteriori error estimate \cite[Remark 4.3]{MR3595879}.
In spite of this fact, it can be proved that $\lambda(\Omega) \leq C$, where $C>0$ is a computable constant (see \cite[Remark 4.4]{MR3595879}), and thus use such a constant to remove the term $\lambda(\Omega)$ from Theorem \ref{thm:apost:plate:rel} to obtain a genuine a posteriori error estimate.

Let $\varepsilon > 0$ be sufficiently small (cf. Proposition \ref{prop:max_h02}). Given $T\in\mathcal{T}$, we introduce the local error indicators
\begin{align*}
  \xi_{p}^{2}(T):=\|\Grad\nabla M_{\varepsilon}\{g-J_h^\mathrm{HCT}u_h^\star,0\}\|_{T}^{2}, 
\qquad
\xi_{c}^{2}(T):=\ip{\lambda_{h}}{M_{\varepsilon}\{g-J_h^\mathrm{HCT}u_h^\star,0\}}_{T}.
\end{align*}
In the next result we prove a different reliability estimate for the plate obstacle problem by using the regularized maximum function $M_{\varepsilon}$.

\begin{theorem}[reliability]
\label{thm:apost:plate:rel_max}
  Let $(\bM,\lambda,u)$ denote the solution of~\eqref{eq:plate:weak_mixed} and let $(\bM_h,\lambda_h,u_h)$ denote the solution of~\eqref{eq:plate:weak_mixed_discrete}. Then, 
  \begin{align*}
    \|\Grad\nabla(u-J_h^\mathrm{HCT} u_h^\star)\|_\Omega^2 + \|\bM-\bM_h\|_\Omega^2 + \|\lambda-\lambda_h\|_{-2}^2 
    \lesssim 
    \sum_{T\in\mathcal{T}}\left(\xi_r^2(T) + \mathrm{osc}(T)^2 + \xi_{p}^{2}(T) + \xi_{c}^{2}(T)\right) =: \xi^2.
  \end{align*}
\end{theorem}
\begin{proof} We use \eqref{eq:normequiv:b} with $w= u - J_h^\mathrm{HCT}u_h^\star, \bN = \bM - \bM_{h}$, and $\chi = \lambda - \lambda_{h}$, in conjunction with $\bM=\Grad\nabla u$, $\dDiv \bM - \lambda = f$, and property \textrm{(i)} from Proposition \ref{prop:plate:properties} to obtain 
\begin{align*}
    \|\Grad\nabla(u-J_h^\mathrm{HCT} u_h^\star)\|_\Omega^2 & + \|\bM-\bM_h\|_\Omega^2  + \|\lambda-\lambda_h\|_{-2}^2 \\
     &\lesssim
     \|(1 - \Pi_{h}^{1})f\|_{-2}^{2} + \|\bM_{h} - \Grad\nabla  J_h^\mathrm{HCT}u_h^\star\|_{\Omega}^{2} + \dual{\lambda - \lambda_{h}}{u -  J_h^\mathrm{HCT}u_h^\star}.
\end{align*}
Utilizing the estimates $\|(1 - \Pi_{h}^{1})f\|_{-2}^{2}\leq \|h_{\mathcal{T}}^{2}(1 - \Pi_{h}^{1})f\|_{\Omega}^{2}$ and \eqref{eq:M_h-u_star} in the previous bound results in
\begin{align*}
    &\|\Grad\nabla(u-J_h^\mathrm{HCT} u_h^\star)\|_\Omega^2 + \|\bM-\bM_h\|_\Omega^2 + \|\lambda-\lambda_h\|_{-2}^2 \\
 & \lesssim    \|h_{\mathcal{T}}^{2}(1 - \Pi_{h}^{1})f\|_{\Omega}^{2} 
  + \sum_{T\in\cT} \Big(\|\bM_{h} -\Grad\nabla u_h^\star\|_T^2 +  h_T^{-3}\|[u_h^\star]\|_{\partial T}^2 + h_T^{-1}\|[\partial_\normal u_h^\star]\|_{\partial T}^2\Big) + \dual{\lambda - \lambda_{h}}{u -  J_h^\mathrm{HCT}u_h^\star}\\
  &\lesssim  \xi^2 + \dual{\lambda - \lambda_{h}}{u -  J_h^\mathrm{HCT}u_h^\star}.
\end{align*}
We now estimate $\dual{\lambda - \lambda_{h}}{u -  J_h^\mathrm{HCT}u_h^\star} = \dual{\lambda}{u -  J_h^\mathrm{HCT}u_h^\star} - \dual{\lambda_{h}}{u -  J_h^\mathrm{HCT}u_h^\star} =: \mathsf{I} - \mathsf{II}$. To estimate $\mathsf{II}$, we use that $\lambda_{h}\geq 0$, $g-u\leq 0$, and identity \eqref{eq:identity_jhHCT} to obtain
\begin{equation*}
- \mathsf{II} = \ip{\lambda_{h}}{J_h^\mathrm{HCT}u_h^\star - u}_{\Omega} = \ip{\lambda_{h}}{J_h^\mathrm{HCT}u_h^\star - g}_{\Omega} + \ip{\lambda_{h}}{g - u}_{\Omega} \leq  0.
\end{equation*}
On the other hand, using that $\int_{\Omega}(u-g)\,\mathrm{d}\lambda=0$, $\lambda\geq 0$, and $M_{\varepsilon}\{g,J_h^\mathrm{HCT}u_h^\star\} \geq \max\{g,J_h^\mathrm{HCT}u_h^\star\}$, we arrive at
\begin{equation*}
\mathsf{I} 
=
\int_{\Omega}(g -  \max\{g,J_h^\mathrm{HCT}u_h^\star\})\,\mathrm{d}\lambda + \dual{\lambda}{\max\{g,J_h^\mathrm{HCT}u_h^\star\} -  J_h^\mathrm{HCT}u_h^\star}
 \leq 
 \dual{\lambda}{M_{\varepsilon}\{g,J_h^\mathrm{HCT}u_h^\star\} -  J_h^\mathrm{HCT}u_h^\star}.
\end{equation*}
From this estimate, it follows that $\mathsf{I} \leq \dual{\lambda - \lambda_{h}}{M_{\varepsilon}\{g,J_h^\mathrm{HCT}u_h^\star\} -  J_h^\mathrm{HCT}u_h^\star} + \dual{\lambda_{h}}{M_{\varepsilon}\{g,J_h^\mathrm{HCT}u_h^\star\} -  J_h^\mathrm{HCT}u_h^\star}$, which implies, for all $\delta >0$, that
\begin{align*}
\mathsf{I} 
\leq 
\frac{\delta}{2}\|\lambda - \lambda_{h}\|_{-2}^{2} + \frac{\delta^{-1}}{2}\|\Grad\nabla(M_{\varepsilon}\{g, J_h^\mathrm{HCT}u_h^\star\} -  J_h^\mathrm{HCT}u_h^\star)\|_{\Omega}^{2} +  \ip{\lambda_{h}}{M_{\varepsilon}\{g,J_h^\mathrm{HCT}u_h^\star\} -  J_h^\mathrm{HCT}u_h^\star}_{\Omega}.
\end{align*}
We conclude the proof by using that $M_{\varepsilon}\{g,J_h^\mathrm{HCT}u_h^\star\} = J_h^\mathrm{HCT}u_h^\star + M_{\varepsilon}\{g-J_h^\mathrm{HCT}u_h^\star,0\}$ and taking $\delta$ sufficiently small.
\end{proof}

If we replace $J_{h}^{\text{HCT}}u_{h}^{\star}$ with the continuous solution $u$, then in general $M_\varepsilon\{g-u,0\}\neq 0$ for $0<|u-g|<\varepsilon$. Thus the estimators $\xi_p$ and $\xi_c$ are not consistent with the exact solution.
Despite this limitation, the estimators can be used to drive adaptive refinements, see section~\ref{sec:num_ex:plate} below.


\section{Numerical examples}
\label{sec:num_ex}

In this section we conduct numerical examples that support our theoretical findings for the two particular problems considered in our work. 
The experiments were performed with a code implemented in MATLAB, and the underlying schemes associated to the obstacle problems were solved with the primal-dual active set strategy from \cite[section 4]{KKT03}.

We use the bulk criterion 
\begin{align*}
  \theta \mathrm{est}^2 \leq \sum_{\mathcal{M}\subset\mathcal{T}} \mathrm{est}(T)^2
\end{align*}
to determine a (minimal) set of elements $\mathcal{M}$ that are marked for refinement. 
Here, we employ the newest vertex bisection algorithm for realizing the mesh-refinements. 
For the membrane obstacle problems we use $\mathrm{est} = \rho$ and for the plate obstacle problems we use $\mathrm{est} = \xi$. The adaptive loop consists of repeating the four major steps, \emph{Solve}, \emph{Estimate}, \emph{Mark}, \emph{Refine}. We choose $\theta = \tfrac12$ in the bulk criterion.
For the plate obstacle problem we use $M_\varepsilon\{\cdot,\cdot\}$ with $\varepsilon = 10^{-10}$.


\subsection{Examples for the membrane obstacle problem}\label{sec:num_ex:membrane}


\subsubsection{Smooth solution (\cite[Section~5.3]{MR4050087})}\label{sec:membrane:num_ex_smooth}

We consider $\Omega=(0,1)^{2}$ and set for $(x,y)\in \Omega$, $u(x,y) := x(1-x)y(1-y)$,
\begin{align*}
f(x,y) 
:=
\begin{cases}
\qquad \quad 0 &\text{ if } x < \frac{1}{2},\\
-\Delta u(x,y) &\text{ if } x \geq \frac{1}{2},
\end{cases}
\quad \text{and} \quad
g(x,y)
:=
\begin{cases}
\qquad	u(x,y) &\text{ if } x < \frac{1}{2},\\
\widetilde{g}(x)y(1-y) &\text{ if } x \geq \frac{1}{2},
\end{cases}
\end{align*}
where $\widetilde{g}$ denotes a suitable polynomial of degree 3 such that $g$ and $\nabla g$ are continuous at the line $x=\tfrac{1}{2}$. Then, $u$ solves the obstacle problem \eqref{eq:obstacle_problem} with data $f$ and obstacle $g$. Moreover, $g\in H^{2}(\Omega)$ and $\lambda = -\Delta u - f\in H^{1}(\Omega)$.


\begin{figure}
  \begin{tikzpicture}
  \begin{groupplot}[group style={group size= 2 by 1},width=0.5\textwidth,cycle list/Dark2-6,
                      cycle multiindex* list={
                          mark list*\nextlist
                          Dark2-6\nextlist},
                      every axis plot/.append style={ultra thick},
                      grid=major,
                      xlabel={number of elements $\#\mathcal{T}$},]
         \nextgroupplot[title={Error contributions},ymode=log,xmode=log,
           legend entries={\tiny{$\|u-u_{h}\|_{\Omega}$},\tiny{$\|\boldsymbol\sigma-\boldsymbol\sigma_{h}\|_{\Omega}$}},
                      legend pos=south west]
                \addplot table [x=dofs,y=error_u] {data/Errors_ex1.dat};
                \addplot table [x=dofs,y=error_sigma] {data/Errors_ex1.dat};
                
                \logLogSlopeTriangle{0.9}{0.2}{0.4}{0.5}{black}{{\small $0.5$}};
                \logLogSlopeTriangleBelow{0.8}{0.2}{0.1}{0.5}{black}{{\small $0.5$}}
              \nextgroupplot[title={Estimator contributions},ymode=log,xmode=log,
           legend entries={\tiny{$\rho_{c}$},\tiny{$\rho_{p}$},\tiny{$\rho_{r}$},\tiny{$\rho$}},
                      legend pos=south west]
                \addplot table [x=dofs,y=estLambdaU] {data/Estimators_ex1.dat};
                \addplot table [x=dofs,y=estViol] {data/Estimators_ex1.dat};
                \addplot table [x=dofs,y=estGradVsigma] {data/Estimators_ex1.dat};
                \addplot table [x=dofs,y=estTot] {data/Estimators_ex1.dat};
                
                \logLogSlopeTriangle{0.9}{0.2}{0.55}{0.5}{black}{{\small $0.5$}};
                \logLogSlopeTriangleBelow{0.8}{0.2}{0.1}{0.5}{black}{{\small $0.5$}}
    \end{groupplot}
\end{tikzpicture}
  \caption{Experimental rates of convergence for the errors $\|u - u_{h}\|_{\Omega}$ and $\|\bsigma - \bsigma_{h}\|_{\Omega}$ (left) and the error estimator $\rho$ with its individual contributions (right), with uniform refinement for the problem from Section \ref{sec:membrane:num_ex_smooth}.}
\label{fig:ex_1}
\end{figure}
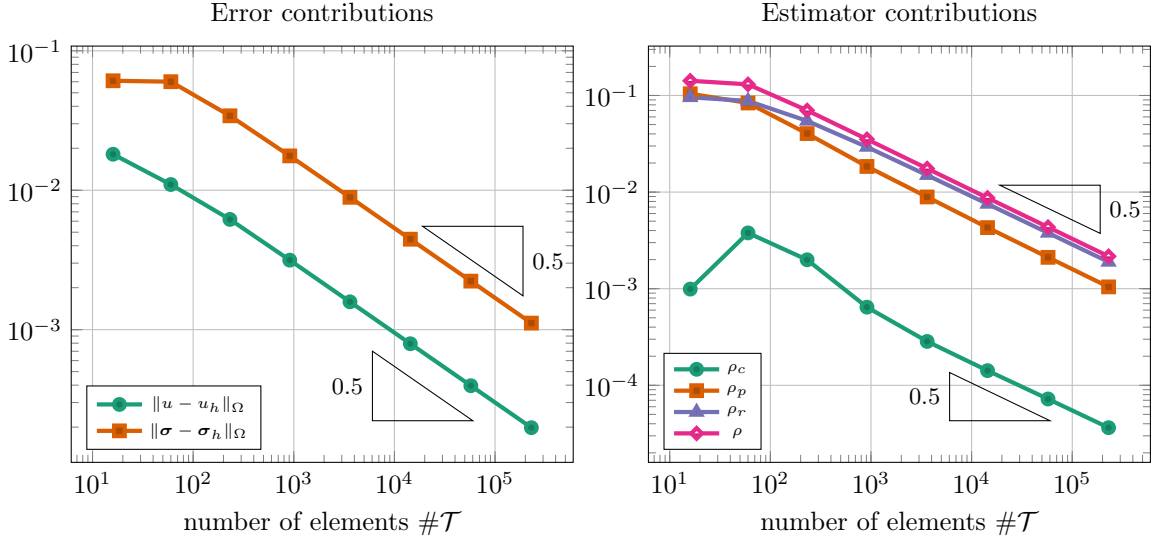

We observe, in Figure \ref{fig:ex_1}, optimal experimental rates of convergence ($\mathcal{O}( (\#\mathcal{T})^{-1/2})\approx \mathcal{O}(h)$) for the errors $\|u-u_{h}\|_{\Omega}$ and $\|\bsigma - \bsigma_{h}\|_{\Omega}$; the same is observed for the a posteriori error estimator $\rho$ and its individual contributions. We recall that $\rho$ is defined in Theorem \ref{thm:apost:membrane:rel}. Note that the convergence rate observed for the errors is in agreement with the error estimate provided in Theorem \ref{thm:error_estimate_membrane}. 


\subsubsection{Unknown solution on L-shaped domain (\cite[Section~5.5]{MR4050087})}\label{sec:membrane:num_ex_unk_L}

In this example, we consider $\Omega=(-1,1)^{2}\setminus [-1,0]^{2}$ and set for $(x,y)\in \Omega$, $f(x,y) = 1$, and a pyramid-like obstacle $g(x,y)=\max\{0,\text{dist}((x,y),\partial\Omega_{u})-\tfrac{1}{4}\}$ for $(x,y)\in\Omega_{u}:=(0,1)^{2}$ and $g(x,y)=0$ for $(x,y)\in\Omega\setminus\Omega_u$.

In Figure \ref{fig:ey_1} we present experimental rates of convergence for the total error estimator $\rho$ and its individual contributions, with uniform and adaptive refinement. We observe that the designed adaptive procedure outperforms uniform refinement. In Figure \ref{fig:ey_2} we present a sequence of adaptively refined meshes and the corresponding approximate solution component $u_{h}$. We observe that the refinement is being concentrated around the re-entrant corner $(0,0)$ and around the point $(\tfrac{1}{2},\tfrac{1}{2})$, which coincides with the tip of the pyramid-like obstacle $g$.


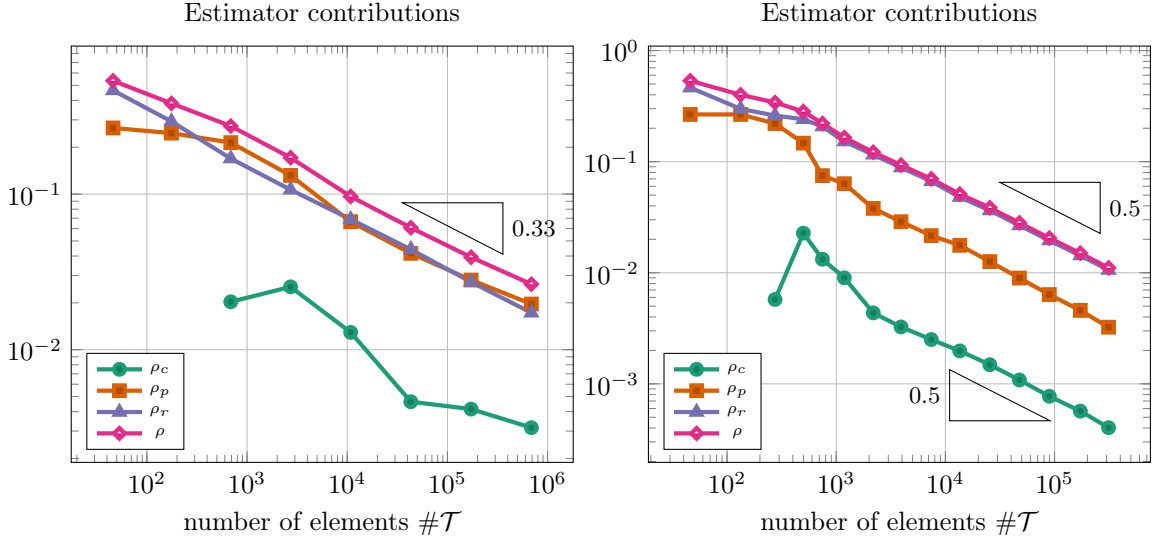
\begin{figure}
  \begin{tikzpicture}
  \begin{groupplot}[group style={group size= 2 by 1},width=0.5\textwidth,cycle list/Dark2-6,
                      cycle multiindex* list={
                          mark list*\nextlist
                          Dark2-6\nextlist},
                      every axis plot/.append style={ultra thick},
                      grid=major,
                      xlabel={number of elements $\#\mathcal{T}$},]
              \nextgroupplot[title={Estimator contributions},ymode=log,xmode=log,
           legend entries={\tiny{$\rho_{c}$},\tiny{$\rho_{p}$},\tiny{$\rho_{r}$},\tiny{$\rho$}},
                      legend pos=south west]
                \addplot table [x=dofs,y=estLambdaU] {data/Estimators_ex2_unif.dat};
                \addplot table [x=dofs,y=estViol] {data/Estimators_ex2_unif.dat};
                \addplot table [x=dofs,y=estGradVsigma] {data/Estimators_ex2_unif.dat};
                \addplot table [x=dofs,y=estTot] {data/Estimators_ex2_unif.dat};
         
               \logLogSlopeTriangle{0.86}{0.2}{0.5}{0.33}{black}{{\small $0.33$}};
                
              \nextgroupplot[title={Estimator contributions},ymode=log,xmode=log,
           legend entries={\tiny{$\rho_{c}$},\tiny{$\rho_{p}$},\tiny{$\rho_{r}$},\tiny{$\rho$}},
                      legend pos=south west]
                \addplot table [x=dofs,y=estLambdaU] {data/Estimators_ex2_adap.dat};
                \addplot table [x=dofs,y=estViol] {data/Estimators_ex2_adap.dat};
                \addplot table [x=dofs,y=estGradVsigma] {data/Estimators_ex2_adap.dat};
                \addplot table [x=dofs,y=estTot] {data/Estimators_ex2_adap.dat};
                
                \logLogSlopeTriangle{0.9}{0.2}{0.55}{0.5}{black}{{\small $0.5$}};
                 \logLogSlopeTriangleBelow{0.8}{0.2}{0.1}{0.5}{black}{{\small $0.5$}}
    \end{groupplot}
\end{tikzpicture}
  \caption{Error estimator $\rho$ and its individual contributions with uniform refinement (left) and adaptive refinement (right) for the problem from Section \ref{sec:membrane:num_ex_unk_L}.}
\label{fig:ey_1}
\end{figure}


\begin{figure}
  \begin{tikzpicture}
  \begin{groupplot}[
      group style={group size=2 by 3, horizontal sep=1cm, vertical sep=1cm},
      width=0.5\textwidth,
      ylabel={$y$},
      xlabel={$x$},
    ]
    \nextgroupplot[title={$\#\mathcal{T}=333$},axis equal,hide axis]
      \addplot[patch,color=white,
        faceted color = black, line width = 0.15pt,
      patch table ={data/elements/ele333.dat}] file{data/coordinates/coo333.dat};
    \nextgroupplot[zmin=-0.05,zmax=0.3,colorbar]
      \addplot3[patch] table{data/solution333.dat};
    \nextgroupplot[title={$\#\mathcal{T}=616$},axis equal,hide axis]
      \addplot[patch,color=white,
        faceted color = black, line width = 0.15pt,
      patch table ={data/elements/ele616.dat}] file{data/coordinates/coo616.dat};
    \nextgroupplot[zmin=-0.05,zmax=0.3,colorbar]
      \addplot3[patch] table{data/solution616.dat};
    \nextgroupplot[title={$\#\mathcal{T}=1109$},axis equal,hide axis]
      \addplot[patch,color=white,
        faceted color = black, line width = 0.15pt,
      patch table ={data/elements/ele1109.dat}] file{data/coordinates/coo1109.dat};
    \nextgroupplot[zmin=-0.05,zmax=0.3,colorbar]
      \addplot3[patch] table{data/solution1109.dat};
  \end{groupplot}
\end{tikzpicture}
  \caption{Meshes and solution component $u_h$ for the problem from Section \ref{sec:membrane:num_ex_unk_L}.}
\label{fig:ey_2}
\end{figure}
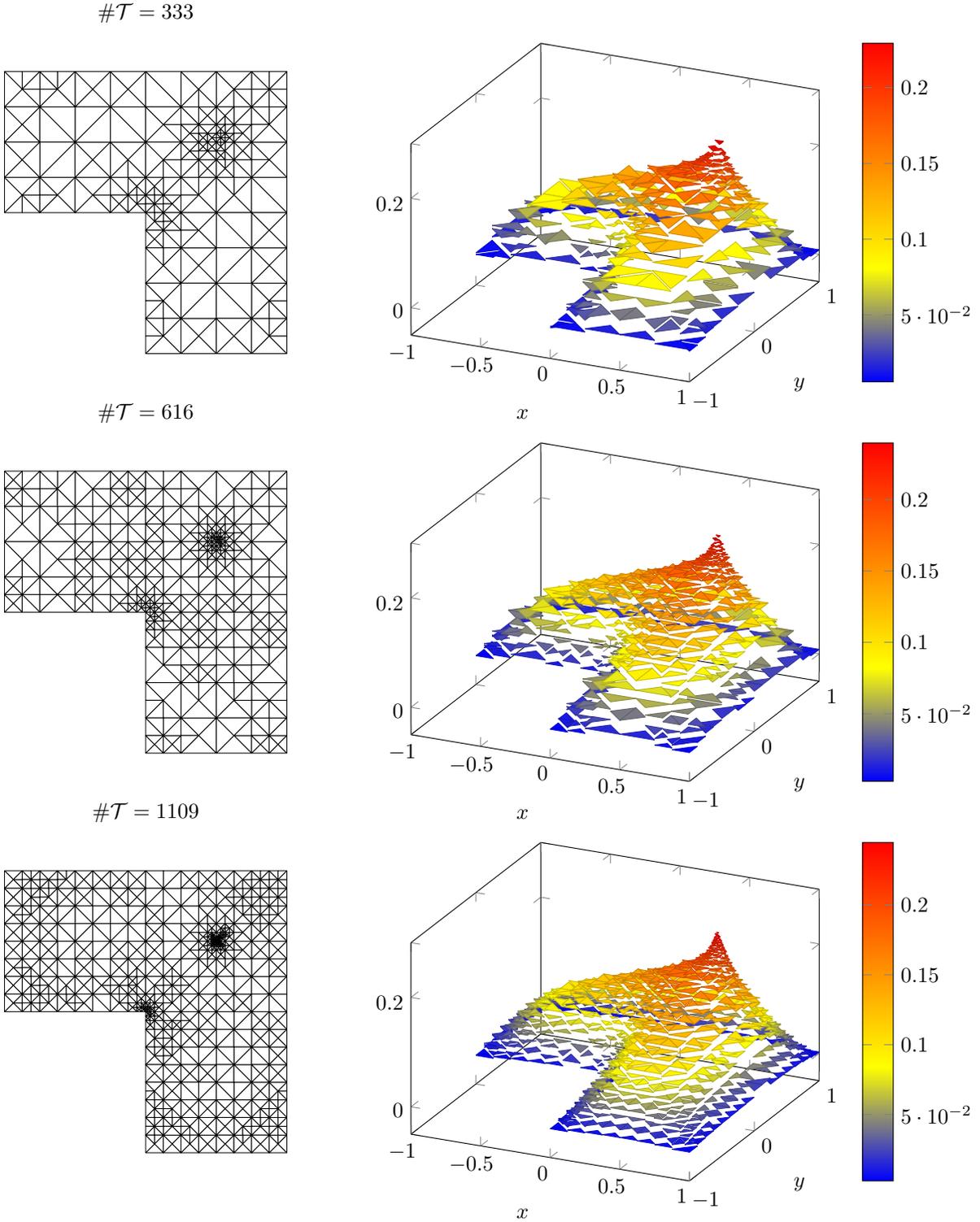


\subsection{Examples for the plate obstacle problem}\label{sec:num_ex:plate}

\subsubsection{Example with known solution}\label{sec:num_ex:plate1}
Let $\Omega = (-1,1)^2$ and set for $(x,y)\in\Omega$, 
\begin{align*}
  u(x,y) &= \begin{cases}
    \big(1-(x^2+y^2)\big)^4 & \text{if }x^2+y^2<1, \\
    0 & \text{if }x^2+y^2\geq 1,
  \end{cases}
  \\
  g(x,y) &= \begin{cases}
    \big(1-(x^2+y^2)\big)^4 & \text{if } x^2+y^2<\frac1{16}, \\
    -\frac{100697}{36864}(x^2+y^2) -\frac{20803}{73728}\sqrt{x^2+y^2} + \frac{74741}{73728} & \text{if }x^2+y^2\geq \frac{1}{16},
  \end{cases}
  \\
  f(x,y) &= \begin{cases}
    \Delta^2 u-100 & \text{if } x^2+y^2<\frac1{16}, \\
    \Delta^2 u & \text{if } x^2+y^2 \geq \frac1{16}.
  \end{cases}
\end{align*}
Note that $u\in H^4(\Omega)\cap H_0^2(\Omega)$, $f\in L^2(\Omega)$ and $g\in H^2(\Omega)\cap C^1(\overline\Omega)$ is piecewise smooth.
We solve the discrete plate obstacle problem on a sequence of uniformly refined meshes.

\begin{figure}
  \begin{center}
    \begin{tikzpicture}
\begin{loglogaxis}[
width=0.49\textwidth,
cycle list/Dark2-6,
cycle multiindex* list={  mark list*\nextlist
                          Dark2-6\nextlist},
xlabel={$\dim(\mathbb{X}(\mathcal{T})\times P^1(\mathcal{T})\times P^1(\mathcal{T}))$},
grid=major,
legend entries={{\tiny $\|u-u_h\|_\Omega$},{\tiny $\|\bM-\bM_h\|_\Omega$},{\tiny $\sqrt{\xi_r^2+\mathrm{osc}^2}$},{\tiny $\sqrt{\xi_c^2+\xi_p^2}$},{\tiny $\lambda_h(\Omega)^{1/2}\xi_\infty^{1/2}$}},
legend pos=north east,
every axis plot/.append style={ultra thick},
]
\addplot table [x=dofMIXED,y=errU] {data/ExamplePlateSmooth.dat};
\addplot table [x=dofMIXED,y=errM] {data/ExamplePlateSmooth.dat};
\addplot table [x=dofMIXED,y=estMalt] {data/ExamplePlateSmooth.dat};
\addplot table [x=dofMIXED,y=estG] {data/ExamplePlateSmooth.dat};
\addplot table [x=dofMIXED,y=estInf] {data/ExamplePlateSmooth.dat};

\logLogSlopeTriangle{0.9}{0.2}{0.42}{1}{black}{{\small $1$}};
\logLogSlopeTriangleBelow{0.8}{0.2}{0.1}{1}{black}{{\small $1$}}

\end{loglogaxis}
\end{tikzpicture}
  \end{center}
  \caption{Error and estimators for the problem from Section~\ref{sec:num_ex:plate1}.}
  \label{fig:plate:ex1}
\end{figure}
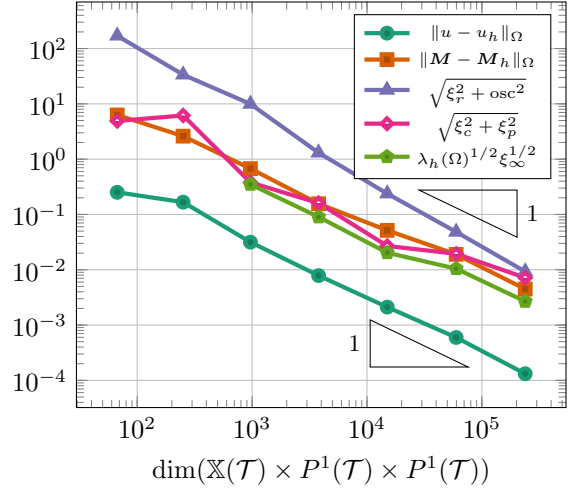

Figure~\ref{fig:plate:ex1} shows the $L^2(\Omega)$ errors of $u-u_h$ and $\bM-\bM_h$ and compares them to the different estimator contributions. We observe optimal experimental convergence rates of the error as well as the estimators, i.e., $\mathcal{O}( (\#\mathcal{T})^{-1})$ behavior.

\subsubsection{Example~\cite[Example~7.3]{MR3595879}}\label{sec:num_ex:plate2}
We consider the setup of~\cite[Example~7.3]{MR3595879}. Set $f(x,y)=0$, and 
\begin{align*}
  g(x,y) = 1-\left( \frac{(x+\tfrac14)^2}{0.2^2} + \frac{y^2}{0.35^2} \right), \quad (x,y)\in \Omega:=(-\tfrac12,\tfrac12)^2\setminus[0,\tfrac12]^2.
\end{align*}
Note that the obstacle function $g$ is smooth. Nevertheless, we expect that the exact solution to have a geometric singularity due to non-convexity of $\Omega$. 
This is confirmed by our numerical examples on a sequence of uniformly refined meshes where a reduced convergence order for the estimator is observed, see Figure~\ref{fig:plate:ex2}.
The use of the adaptive algorithm steered by the estimator $\xi$ recovers rates $\mathcal{O}((\#\cT)^{-1})$.

\begin{figure}
  \begin{center}
    \begin{tikzpicture}
\begin{loglogaxis}[
width=0.49\textwidth,
cycle list/Dark2-6,
cycle multiindex* list={  mark list*\nextlist
                          Dark2-6\nextlist},
xlabel={$\dim(\mathbb{X}(\mathcal{T})\times P^1(\mathcal{T})\times P^1(\mathcal{T}))$},
grid=major,
legend entries={{\tiny $\sqrt{\xi_r^2+\mathrm{osc}^2}$ (u)},{\tiny $\sqrt{\xi_c^2+\xi_p^2}$ (u)},{\tiny $\lambda_h(\Omega)^{1/2}\xi_\infty^{1/2}$ (u)},{\tiny $\sqrt{\xi_r^2+\mathrm{osc}^2}$ (a)},{\tiny $\sqrt{\xi_c^2+\xi_p^2}$ (a)},{\tiny $\lambda_h(\Omega)^{1/2}\xi_\infty^{1/2}$ (a)}},
legend pos=outer north east,
every axis plot/.append style={ultra thick},
]
\addplot table [x=dofMIXED,y=estMalt] {data/ExamplePlate2_unif.dat};
\addplot table [x=dofMIXED,y=estG] {data/ExamplePlate2_unif.dat};
\addplot table [x=dofMIXED,y=estInf] {data/ExamplePlate2_unif.dat};
\addplot table [x=dofMIXED,y=estMalt] {data/ExamplePlate2_adap.dat};
\addplot table [x=dofMIXED,y=estG] {data/ExamplePlate2_adap.dat};
\addplot table [x=dofMIXED,y=estInf] {data/ExamplePlate2_adap.dat};

\logLogSlopeTriangle{0.87}{0.2}{0.55}{0.27}{black}{{\small $0.27$}};
\logLogSlopeTriangleBelow{0.8}{0.2}{0.1}{1}{black}{{\small $1$}}

\end{loglogaxis}
\end{tikzpicture}
  \end{center}
  \caption{Estimators on a sequence of uniformly (u) and adaptively (a) refined meshes for the problem from Section~\ref{sec:num_ex:plate2}.}
  \label{fig:plate:ex2}
\end{figure}
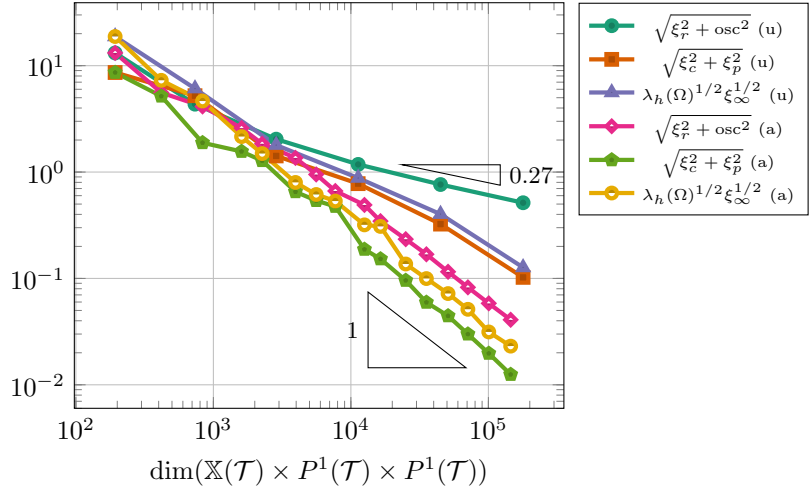

Figure~\ref{fig:plate:ex2meshes} visualizes three meshes generated by the adaptive algorithm. Strong refinements towards the reentrant corner of the domain and in the vicinity of the contact region can be observed. 
This is in agreement with~\cite[Example~7.3]{MR3595879} where similar observations have been made for the $C^0$ interior penalty method studied there.
\begin{figure}
  \begin{center}
%
\begin{tikzpicture}
  \begin{groupplot}[group style={group size= 3 by 1},width=0.33\textwidth,axis equal,hide axis]
    \nextgroupplot[title={$\#\mathcal{T}=857$}]
    \addplot[patch,color=white,
      faceted color = black, line width = 0.15pt,
    patch table ={data/ExamplePlate2_ele_00857.dat}] file{data/ExamplePlate2_coo_00857.dat};
    \nextgroupplot[title={$\#\mathcal{T}=1122$}]
    \addplot[patch,color=white,
      faceted color = black, line width = 0.15pt,
    patch table ={data/ExamplePlate2_ele_01122.dat}] file{data/ExamplePlate2_coo_01122.dat};
    \nextgroupplot[title={$\#\mathcal{T}=1714$}]
    \addplot[patch,color=white,
      faceted color = black, line width = 0.15pt,
    patch table ={data/ExamplePlate2_ele_01714.dat}] file{data/ExamplePlate2_coo_01714.dat};
  \end{groupplot}
\end{tikzpicture}
  \end{center}
  \caption{Three consecutive meshes generated by the adaptive algorithm for the problem from Section~\ref{sec:num_ex:plate2}.}
  \label{fig:plate:ex2meshes}
\end{figure}
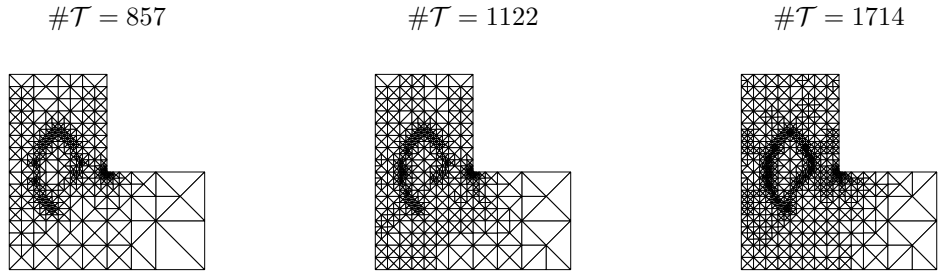

\subsubsection{Example with non-smooth obstacle on L-shaped domain}\label{sec:num_ex:plate3}
For this last example we consider the L-shaped domain $\Omega = (-1,1)^2\setminus[-1,0]^2$, $f(x,y) = 0$ and 
\begin{align*}
  g(x,y) = \frac14 - \big((x-\tfrac12)^2+(y-\tfrac12)^2\big)^{3/4}, \quad (x,y)\in \Omega. 
\end{align*}
Note that $g\in H^2(\Omega)\cap C^1(\overline\Omega)$ but $g\notin H^k(\Omega)$ for $k\geq 3$.
Due to non-convexity of the domain and the regularity of $g$ one expects reduced regularity of the exact solution $u$ to the plate obstacle problem. 
Again, uniformly refined meshes lead to reduced convergence rates whereas the adaptive algorithm seems experimentally recover optimal rates, see Figure~\ref{fig:plate:ex3}. 
Furthermore, Figure~\ref{fig:plate:ex3sol} shows two meshes together with corresponding solution component $u_h$. We observe that the meshes are strongly refined towards the origin $(0,0)$ of the domain as well as towards $(\tfrac12,\tfrac12)$ where the obstacle touches the plate and has reduced regularity.

\begin{figure}
  \begin{center}
    \begin{tikzpicture}
\begin{loglogaxis}[
width=0.49\textwidth,
cycle list/Dark2-6,
cycle multiindex* list={  mark list*\nextlist
                          Dark2-6\nextlist},
xlabel={$\dim(\mathbb{X}(\mathcal{T})\times P^1(\mathcal{T})\times P^1(\mathcal{T}))$},
grid=major,
legend entries={{\tiny $\sqrt{\xi_r^2+\mathrm{osc}^2}$ (u)},{\tiny $\sqrt{\xi_c^2+\xi_p^2}$ (u)},{\tiny $\lambda_h(\Omega)^{1/2}\xi_\infty^{1/2}$ (u)},{\tiny $\sqrt{\xi_r^2+\mathrm{osc}^2}$ (a)},{\tiny $\sqrt{\xi_c^2+\xi_p^2}$ (a)},{\tiny $\lambda_h(\Omega)^{1/2}\xi_\infty^{1/2}$ (a)}},
legend pos=outer north east,
every axis plot/.append style={ultra thick},
]
\addplot table [x=dofMIXED,y=estMalt] {data/ExamplePlate3_unif.dat};
\addplot table [x=dofMIXED,y=estG] {data/ExamplePlate3_unif.dat};
\addplot table [x=dofMIXED,y=estInf] {data/ExamplePlate3_unif.dat};
\addplot table [x=dofMIXED,y=estMalt] {data/ExamplePlate3_adap.dat};
\addplot table [x=dofMIXED,y=estG] {data/ExamplePlate3_adap.dat};
\addplot table [x=dofMIXED,y=estInf] {data/ExamplePlate3_adap.dat};

\logLogSlopeTriangle{0.89}{0.2}{0.69}{0.27}{black}{{\small $0.27$}};
\logLogSlopeTriangle{0.78}{0.11}{0.3}{1}{black}{{\small $1$}};
\logLogSlopeTriangleBelow{0.78}{0.18}{0.03}{1}{black}{{\small $1$}}

\end{loglogaxis}
\end{tikzpicture}
  \end{center}
  \caption{Estimators on a sequence of uniformly (u) and adaptively (a) refined meshes for the problem from Section~\ref{sec:num_ex:plate3}.}
  \label{fig:plate:ex3}
\end{figure}
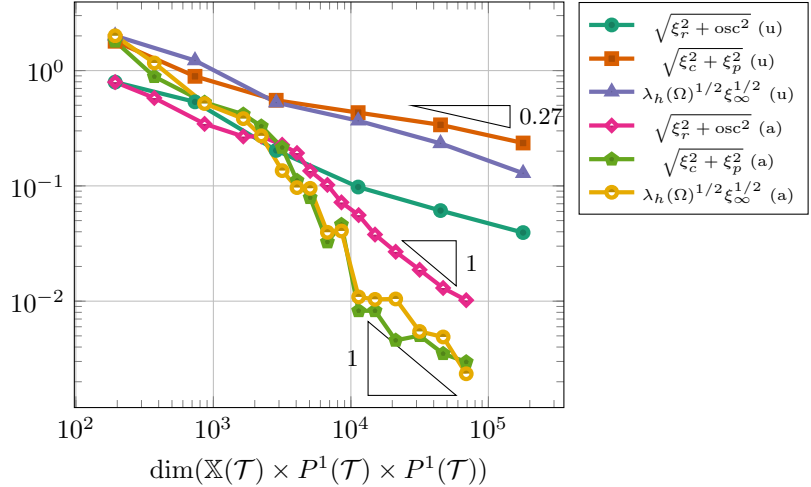

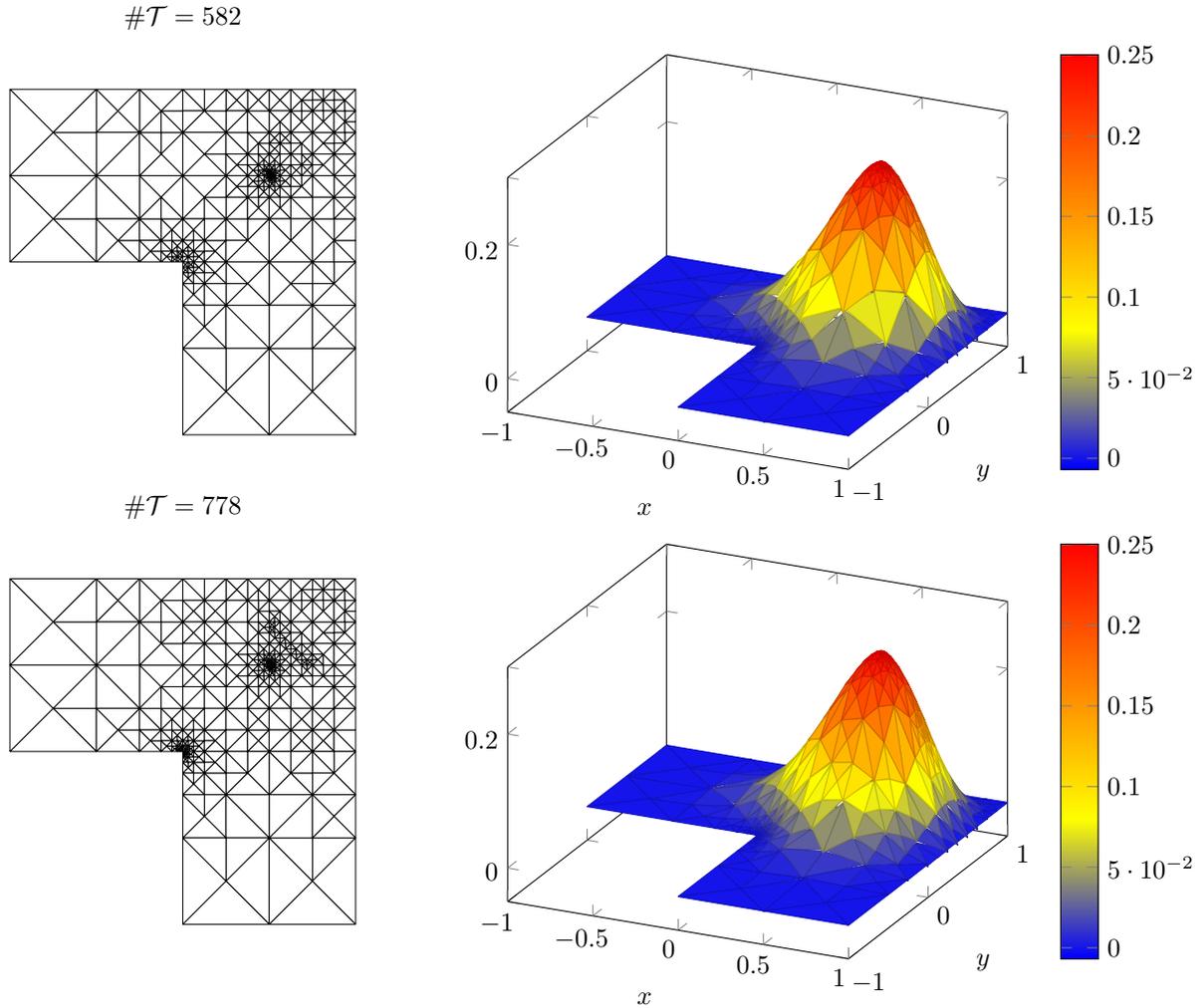
\begin{figure}
  \begin{center}
    \begin{tikzpicture}
  \begin{groupplot}[
      group style={group size=2 by 2, horizontal sep=1cm, vertical sep=1cm},
      width=0.5\textwidth,
      ylabel={$y$},
      xlabel={$x$},
    ]
    \nextgroupplot[title={$\#\mathcal{T}=582$},axis equal,hide axis]
      \addplot[patch,color=white,
        faceted color = black, line width = 0.15pt,
      patch table ={data/ExamplePlate3_ele_00582.dat}] file{data/ExamplePlate3_coo_00582.dat};
    \nextgroupplot[zmin=-0.05,zmax=0.3,colorbar]
      \addplot3[patch] table{data/ExamplePlate3_sol_00582.dat};
    \nextgroupplot[title={$\#\mathcal{T}=778$},axis equal,hide axis]
      \addplot[patch,color=white,
        faceted color = black, line width = 0.15pt,
      patch table ={data/ExamplePlate3_ele_00778.dat}] file{data/ExamplePlate3_coo_00778.dat};
    \nextgroupplot[zmin=-0.05,zmax=0.3,colorbar]
      \addplot3[patch] table{data/ExamplePlate3_sol_00778.dat};
  \end{groupplot}
\end{tikzpicture}
  \end{center}
  \caption{Meshes and solution component $u_h$ for the problem from Section~\ref{sec:num_ex:plate3}.}
  \label{fig:plate:ex3sol}
\end{figure}

\bibliographystyle{siam}
\bibliography{biblio}
\end{document}